\documentclass{report}
\usepackage{algorithm}%to remove
\usepackage{fullpage}
\usepackage[utf8]{inputenc}
\usepackage[pdftex,hidelinks]{hyperref}

\usepackage{comment}
\usepackage{color}

\newcommand{\TODO}[2][]{TODO $\Big\langle${\ifthenelse{\equal{#1}{P}}{\color{green}}{\ifthenelse{\equal{#1}{B}}{\color{blue}}{\color{red}}}{#2}}$\Big\rangle$}

\newcommand{\dfn}[1]{\textbf{#1}}
\newcommand{\xpr}[1]{``{#1}''}
\newcommand{\resp}[1]{\ (resp. #1)}
\newcommand{\ie}{\textit{i.e.}, }

\usepackage{amsmath,amsfonts,pifont,amssymb,latexsym,amsthm}
\newcommand\impl{\Rightarrow}

\newcommand\defeq{:=}
\newcommand\eqdef{=:}

\newtheorem{theorem}{Theorem}
\newtheorem{proposition}[theorem]{Proposition}
\newtheorem{lemma}[theorem]{Lemma}
\newtheorem{corollary}[theorem]{Corollary}
\newtheorem{remark}[theorem]{Remark}
\newtheorem{fact}[theorem]{Fact}

\newcommand{\sett}[2]{\left\{\left.#1\vphantom{#2}\right|#2\right\}}
\newcommand{\set}[3]{\sett{#1\in #2}{#3}}
\newcommand{\card}[1]{\left|#1\right|}
\newcommand{\restr}[1]{_{\left|#1\right.}}
\newcommand\subfini{\subset_{\text{finite}}}

\newcommand\Rb{\mathbbm P}%R\sqcup\{\infty\}}
\newcommand\NE{\mathcal{N}}
\newcommand\lin[1]{\mathbf{#1}}

\newcommand{\ifnv}[2]{\ifthenelse{\equal{#1}{}}{}{#2}} %tests argument emptiness
\newcommand{\mappl}[4][]{\begin{array}{rrcl}\ifnv{#1}{#1:}&#2&\to&#3\\
\displaystyle #4\end{array}}
\newcommand{\appl}[5][]{\mappl[#1]{#2}{#3}{&#4&\mapsto&#5}}
\newcommand{\pmappl}[4][]{\begin{array}{rrcl}\ifnv{#1}{#1:}&#2&\pto&#3\\
\displaystyle #4\end{array}}
\newcommand{\pappl}[5][]{\pmappl[#1]{#2}{#3}{&#4&\mapsto&#5}}
\newcommand\si{\text{ if }}
\newcommand\pto{\nrightarrow} %to use for partial maps
\newcommand{\both}[1]{\left\{\everymath{\displaystyle\everymath{}}\begin{array}{l}#1\end{array}\right.}

\newcommand{\bothrl}[1]{\left\{\everymath{\displaystyle\everymath{}}\begin{array}{rl}#1\end{array}\right.}

\newcommand{\soit}[1]{\left|\everymath{\displaystyle\everymath{}}\begin{array}{ll}#1\end{array}\right.}
\DeclareMathOperator*{\id}{id}

\newcommand{\dom}{\mathcal D}%\DeclareMathOperator*{\dom}{dom}
\newcommand{\rock}[2][]{\tilde{\mathcal D}^{#1}(#2)}
\newcommand{\rocks}[3][]{\tilde{\mathcal D}^{#1}_{#2}(#3)}
\newcommand\I[1][]{\mathcal I_{#1}}
\newcommand\simu[1][]{\underset{#1}\succeq}

\usepackage{bbm}
\newcommand{\Z}{\mathbb Z}
\newcommand{\N}{\mathbb N}
\newcommand{\MM}{\mathbb M}
\newcommand{\Ns}{{\mathbb N}_1}

\newcommand{\R}{\mathbb R}
\newcommand{\Q}{\mathbb Q}
\newcommand{\ipart}[1]{\left\lfloor #1\right\rfloor}
\newcommand{\spart}[1]{\left\lceil #1\right\rceil}
\newcommand{\abs}[1]{\left|#1\right|}
\usepackage{stmaryrd} % [| |]
\newcommand{\co}[2]{\left\llbracket #1,#2\right\llbracket}
\newcommand{\cc}[2]{\left\llbracket #1,#2\right\rrbracket}

\newcommand\A{\mathcal A}
\newcommand\B{\mathcal B}
\newcommand{\C}[1][]{\mathcal C_{{#1}}}
\newcommand{\haine}[1]{\mathbbm F_{#1}}

\newcommand{\f}{\mathcal{F}}
\newcommand{\lang}{\mathcal L}
\newcommand\motvide\epsilon
\newcommand{\length}[1]{\left|#1\right|}
\newcommand\norm[1]{\left\Vert #1\right\Vert}%ngth{#1}}%changer pour n'avoir qu'une?
\newcommand{\bina}[1]{\underline{#1}}
\newcommand{\anib}[2][]{\overline{#2}^{#1}}
\newcommand\sh[2][]{{\left\langle#2\right\rangle_{#1}}}%\mathcal{\chi}(#1)}
\newcommand{\Chi}[2][]{\chi^{#1}\left(#2\right)}
\newcommand{\hs}[1]{{\left\rangle#1\right\langle}}

\renewcommand{\vec}[1]{\mathbf{#1}}
\newcommand{\seq}[1]{\ifthenelse{\equal{#1}{\Phi}}{\Phi}{\ifthenelse{\equal{#1}{\tau}}{\tau}{\mathfrak{#1}}}}

\newcommand{\parint}{\mathbf{\mathcal{B}}}

\newcommand{\am}{\A^{\Z^2}}
\newcommand{\azd}{\A^{\Z^d}}

\newcommand{\az}{\A^{\Z}}
\newcommand{\bz}{\B^{\Z}}
\newcommand{\dinf}[1]{\vphantom{#1}^\infty{#1}^\infty}
\newcommand{\uinf}[1]{#1^\infty}%remplacer les \omega par \infty?
\newcommand{\pinf}[1]{\vphantom{#1}^\infty{#1}}
\newcommand\orb[2][]{\mathcal O_{#2}\ifnv{#1}{(#1)}{}}
\newcommand\bulk[2][]{#2_{[#1]}}
\newcommand\X{\mathfrak Z}
\newcommand{\emp}[2][]{\mathcal{S}^{#1}_{#2}}
\newcommand{\per}[2][]{\mathcal{P}^{#1}_{#2}}
\newcommand\col[2]{B_{#1}^{#2}}
\newcommand{\gr}[3]{\Sigma^{#1,#2}_{#3}}
\newcommand{\gra}[4]{\Sigma^{#1,#2,#3,#4}}

\newcommand{\grs}[2]{\Sigma^{#1,#2}}

\newcommand{\M}{\mathcal{M}}
\newcommand\U{\mathcal U}
\newcommand\halt[3]{\mathcal H_{#1}^{#2}(#3)}

\newcommand{\maxaddr}{\texttt{MAddr}}
\newcommand{\maxage}{\texttt{MClock}}
\newcommand{\addr}{\texttt{Addr}}
\newcommand{\age}{\texttt{Clock}}
\newcommand{\info}{\texttt{Tape}}
\newcommand{\mail}{\texttt{Tape}}
\newcommand{\prog}{\texttt{Prog}}
\newcommand{\newinfo}{\texttt{NTape}}
\newcommand{\schedule}{\texttt{Alarm}}
\newcommand\revprog{\texttt{RProg}}
\newcommand{\otherinfo}{\texttt{OTape}}
\newcommand{\other}{\texttt{Other}}
\newcommand{\lev}{\texttt{Level}}
\newcommand{\head}{\texttt{Head}}
\newcommand\field{\texttt{Field}}
\newcommand\double[1]{\overline{#1}^{\haine5}}

\newcommand{\mshift}{\texttt{MShift}}

\newcommand{\mhist}{\texttt{MHist}}

\newcommand{\selfs}{\texttt{Self}}

\usepackage{algorithmic,algorithm,framed,caption}%mdframed}
\newenvironment{algo}[3]{\begin{framed}%mdframed}[skipabove=10,innertopmargin=5,skipbelow=10,leftmargin=10,rightmargin=10]%,frametitlebackgroundcolor=gray!20,frametitlerule=true,frametitle=
$\hypertarget{#1}{#2}[#3]$\vspace{5pt}\begin{algorithmic}[1]
}{\end{algorithmic}\end{framed}}
\newcommand\incr{\textbf{incr}}

\newcommand\compute{{\hyperlink{compute}{Compute}}}
\newcommand\chekk{\hyperlink{chekk}{Check}}

\newcommand\rite{\hyperlink{rite}{Write}}
\newcommand\unive{\hyperlink{unive}{Simulate}}
\newcommand\shift{\hyperlink{shift}{MacroShift}}
\newcommand\exch{\hyperlink{exch}{Swap}}
\newcommand\coordi{\hyperlink{coordi}{Grid}}
\newcommand\chekka{\hyperlink{chekka}{CkAlph}}
\newcommand\hier{\hyperlink{hier}{HCheck}}
\newcommand\hsim{\hyperlink{hsim}{HSimul}}
\newcommand\intru{\hyperlink{intru}{Univ}}

\newcommand\reali{\hyperlink{reali}{Reali}}

\newcommand\syncomp{\hyperlink{syncomp}{SyncComput}}

\usepackage{url}
\begin{document}
\bibliographystyle{plain}

\title{Hierarchy and Expansiveness in Two-Dimensional Subshifts of Finite Type}

\author{Charalampos Zinoviadis}

%\tucsnumber{209}
%\isbn{978-952-12-3336-4}

\maketitle

\pagestyle{plain}
\setcounter{page}{1}
\pagenumbering{roman}

% Include the file containing the abstract of the thesis

\chapter*{Abstract}

Subshifts are sets of configurations over an infinite grid defined by a set of forbidden patterns. In this thesis, we study two-dimensional subshifts of finite type ($2$D SFTs), where the underlying grid is $\Z^2$ and the set of forbidden patterns is finite. We are mainly interested in the interplay between the computational power of $2$D SFTs and their geometry, examined through the concept of expansive subdynamics. $2$D SFTs with expansive directions form an interesting and natural class of subshifts that lie between dimensions $1$ and $2$. An SFT that has only one non-expansive direction is called extremely expansive. We prove that in many aspects, extremely expansive $2$D SFTs display the totality of behaviours of general $2$D SFTs.

For example, we construct an aperiodic extremely expansive $2$D SFT and we prove that the emptiness problem is undecidable even when restricted to the class of extremely expansive $2$D SFTs. We also prove that every Medvedev class contains an extremely expansive $2$D SFT and we provide a characterization of the sets of directions that can be the set of non-expansive directions of a $2$D SFT. Finally, we prove that for every computable sequence of $2$D SFTs with an expansive direction, there exists a universal object that simulates all of the elements of the sequence. We use the so called hierarchical, self-simulating or fixed-point method for constructing $2$D SFTs which has been previously used by G\'{a}cs, Durand, Romashchenko and Shen.

\tableofcontents
\cleardoublepage

\pagenumbering{arabic}

\chapter{Historical overview}

This thesis is about two-dimensional subshifts of finite type ($2$D SFTs), and more specifically, the behaviour of 2D SFTs with respect to a dynamical-geometrical notion called expansive subdynamics.

The mathematical study of 2D SFTs began with the paper of Wang \cite{wang}. A \dfn{Wang tile set} consists of a finite number of unit squares with coloured edges, which are called \dfn{tiles}. A valid tiling is a way to fill the entire plane with tiles such that the squares are edge-to-edge and such that the colors of abutting edges are the same. Wang asked the following question about Wang tile sets, which is called the \dfn{tiling problem}: Does there exist an algorithm that takes as input an arbitrary Wang tile set and decides whether it admits a valid tiling? He conjectured that the answer to this question is positive and proved that the problem is strongly correlated to the problem of the existence of an \dfn{aperiodic tile set}, that is a tile set that admits some valid tiling but no periodic valid tiling.

However, Berger \cite{berger} proved that this is not the case. In fact, he proved that the tiling problem is undecidable. In addition, his proof contained an explicit construction of an aperiodic tile set. Later, several authors have given alternative constructions of  aperiodic tile sets and proofs of the undecidability of the tiling problem \cite{robinson,jarkkosmall,jarkkoundec}.

There is an alternative way of looking at and talking about the same problem. Let $\A$ be a finite set, called the \dfn{alphabet}. A (two-dimensional, or $2$D) \dfn{configuration} is a map $c \colon \Z^2 \to \A$. The set of all configurations $\A^{\Z^2}$ is called the \dfn{full shift}. A \dfn{pattern} is a map $p \colon D \to \A$, where $D \subseteq \Z^2$ is a \emph{finite} set. Let $\f$ be a set of \dfn{forbidden} patterns. The corresponding ($2$D) \dfn{subshift} $X_{\f}$ is the set of all  configurations that avoid the patterns of $\f$: for all finite $D \subseteq \Z^2$ and $c \in X_{\f}$, $c\restr{D} \notin \f$. If $X=X_{\f}$ for some \emph{finite} set of forbidden patterns, then it is called a 2D subshift of finite type (SFT). In this thesis, we will only talk about $2$D subshifts and SFTs, so that we will usually omit the dimension, except in statements of theorems.

It is not difficult to see that the set of valid tilings of a Wang tile set is an SFT. In addition, for every SFT, we can construct a Wang tile set whose set of valid tilings is, in some sense, equivalent to the given SFT. The tiling problem can thus be rephrased as the \dfn{emptiness problem} for SFTs: Given a finite set of forbidden patterns, can we algorithmically decide whether $X_{\f} \neq \emptyset$? The undecidability of the tiling problem then is then immediately translated to the undecidability of the emptiness problem of SFTs.

Wang tiles and forbidden patterns give a geometrical definition of SFTs, but there also exists an equivalent dynamical definition. First of all, the full shift can be endowed with the product topology of the discrete topology on $\A$. This gives rise to a compact, metrizable topological space. The \dfn{horizontal} and \dfn{vertical shifts}, which consist in moving a configuration one step to the left and up, respectively, are continuous with respect to this topology and obviously commute. This defines a $\Z^2$ action over the full shift and we can study it using the usual tools of topological dynamics.

For example, one can prove that subshifts are exactly the closed, shift-invariant subsets of the full shift, or, equivalently the subsystems of the full shift. SFTs correspond to the chain-mixing subsystems of the full shift. More importantly, for the purposes of this thesis, we can study $2$D SFTs from the point of view of their \dfn{expansive subdynamics}. This notion was defined by Boyle and Lind \cite{expsubd} as a tool for studying multidimensional dynamical systems by looking at the (lower-dimensional) actions induced by the subgroups of the original action. Intuitively, this is the same as when we look at the lower-dimensional projections of a surface in order to understand some of its properties. 

The general definition of expansive subdynamics and the main results of \cite{expsubd} fall out of the scope of this thesis. However, for $2$D subshifts there exists an equivalent, natural geometrical definition.  Let $X \subseteq \A^{\Z^2}$ be a subshift, $l\in\Rb\defeq \R \sqcup \{\infty\}$ a slope and $\lin{l}\subset\R^{2}$ the corresponding line that passes through the origin. We say that $l$ is an \dfn{expansive direction} of $X$ if there exists a finite shape $V\subset\R^2$ such that, for all $x,y \in X$,
\[x\restr{(\lin{l}+V) \cap \Z^2}=y\restr{(\lin{l}+V) \cap \Z^2}\impl x=y~.\]
In other words, there exists a fixed width $b >0$ such that every configuration of $X$ is uniquely defined by its restriction to the strip of slope $l$ and width $b$ that goes through the origin (in fact, by shift invariance, by any strip). Geometrically, this means that  in $X$ the ($2$D) information of the configuration is \xpr{packed} inside the one-dimensional strip of slope $l$. In some sense, even though $X$ is a two-dimensional object, it is determined by a one-dimensional strip, so that subshifts with directions of expansiveness are somewhere between dimensions $1$ and $2$.

A direction that is not expansive is called \dfn{non-expansive}. Let $\NE(X)$ be the set of non-expansive directions of $X$.
Boyle and Lind proved that $\NE(X)$ is closed in the one-point compactification topology of $\Rb$ and that $\NE(X) \neq \emptyset$ if and only if $X$ is infinite. Since finite subshifts are rather trivial, the most restricted non-trivial case with respect to non-expansive directions is the case when $X$ has a unique direction of non-expansiveness. We call such a subshift \dfn{extremely expansive}. Extremely expansive SFTs form the main object of interest in this thesis. We prove that in many aspects, extremely expansive SFTs are computationally as powerful as general SFTs.

Before stating the results, we find it useful to talk about another class of SFTs with many directions of non-expansiveness, namely those that arise from  deterministic tile set. A tile set is called \dfn{NW-deterministic} (the initials stand for North and West) if every tile is uniquely determined by the colors of its top and left sides\cite{nilpind}. Similarly, we can define SW, SE and NE deterministic tile sets (S and E stand for South and East, respectively). A tile set is called \dfn{4-way deterministic} if it is SW,NW,SE and NE deterministic \cite{karipapasoglou}. One can easily see that for the SFT associated to a 4-way deterministic tile set and for every direction $l$ that is not the vertical or the horizontal one (slopes $\infty$ and $0$, respectively), $l$ is an expansive direction. Guillon, Kari and Zinoviadis recently proved \cite{pierreunpub} that the vertical and the horizontal direction must indeed be non-expansive unless the associated SFT is in some sense trivial, namely vertically or horizontally periodic. 

We can now start stating the results of the thesis. The first result concerns the existence of an aperiodic extremely expansive SFT. As mentioned earlier, for the unrestricted case, there exist various constructions of aperiodic SFTs. Kari and Papasoglou \cite{karipapasoglou} have constructed an aperiodic 4-way deterministic tile set. According to what was said in the previous paragraph, the SFT associated to this tile set has exactly two non-expansive directions, the vertical and the horizontal one. We prove that

\begin{theorem}
There exists an aperiodic extremely expansive $2$D SFT.
\end{theorem}

Of course, our construction does not use a 4-way deterministic tile set. It might seem that this result is strictly better than the one using 4-way deterministic tile sets, since we have one non-expansive direction less. However, there exists a small nuance here: 4-way deterministic tile sets give rise to SFTs with so-called \dfn{bounded radii of expansiveness}, while our construction does not have this property. In addition, in \cite{pierreunpub} it is also proved that every aperiodic SFT with bounded radii of expansiveness must have at least two non-expansive directions. Therefore, the 4-way deterministic construction is also optimal, in the class of SFTs with bounded radii of expansiveness, and it might be more precise to say that the two results are incomparable.

As mentioned already, the existence of an aperiodic tile set was originally constructed in order to prove that the tiling problem is undecidable. Kari \cite{nilpind} prove that the tiling problem for NW-deterministic tile sets is undecidable. In addition, Lukkarila \cite{lukkarila} used the 4-way deterministic tile set of Kari and Papazoglou in order to prove that the tiling problem is undecidable for 4-way deterministic tile sets as well. As the reader has probably guessed already, we prove that

\begin{theorem}
The emptiness problem of extremely expansive $2$D SFTs is undecidable. More precisely, the emptiness problem is undecidable for $2$D SFTs such that the vertical direction is the only non-expansive direction.
\end{theorem}

One should understand the previous statement in the following sense: even if one is given an SFT $X$ (as a finite set of forbidden patterns) and is given the additional information that $X$ is either empty or extremely-expansive (and in this case $\NE(X)=\{\infty\}$), even then it is not possible to decide whether $X=\emptyset$. In other words, it is not possible to algorithmically separate the sets of forbidden patterns that define empty SFTs from those that define extremely expansive non-empty SFTs.

The third result can be considered a stronger version of the undecidability of the emptiness problem. We prove that there exist extremely expansive SFT whose configurations are computationally as complicated as possible.

In order to describe this result, we need to introduce some classical notions of computation theory. For the purposes of this introduction, a \dfn{computable function} will mean a function $f \colon \A^{\N} \to \A^{\N}$ such that there exists a Turing Machine that outputs $f(c)$ when originally its reading tape contains $c$ (\ie it outputs $f(c)$ with oracle $c$). Using an effective enumeration of $\Z^2$, it is possible to talk about computable functions with domain or range $\A^{\Z^2}$, and in general $\A^\MM$, where $\MM$ is any effectively enumerable set.

We say that $d \in \A^{\MM}$ is \dfn{reducible} to $c \in \A^{\MM'}$ if there exists a computable function $f$ such that $f(c)=d$. This means that $c$ is computationally at least as complicated as $d$, since it is possible to obtain $d$ using $c$ and a computable function. A subset $Y \subseteq \A^{\MM}$ is called \dfn{Medvedev} reducible to $X \subseteq \A^{\MM'}$ if every point of $Y$ is reducible to some point of $X$. Intuitively, we can compute any point of $Y$ with the help of a suitable point of $X$ and a computable function. The relation of Medvedev reducibility is a pre-order on subsets.

Two sets are called Medvedev equivalent if they are Medvedev reducible to each other. This is an equivalence relation, whose equivalence classes are called \dfn{Medvedev degrees}. There exists a partial order on the set of Medvedev degrees given by the natural lift of the Medvedev reducibility pre-order. Computable sets are the least element of this order and, in a certain sense, the higher a set is in this hierarchy, the more difficult it is to compute a point of this set relative to the sets that lie lower in the hierarchy. The survey \cite{hinman} contains a thorough study of Medvedev degrees.

A set $X \subseteq \A^{\MM}$ is called \dfn{effectively closed} if its complement is semi-decidable. Effectively closed sets form the so-called $\Pi_{0}^{1}$ sets and they play a very important role in computation theory. It is easy to see that SFTs are effectively closed, even though there exist many effectively closed sets (and even effectively closed subshifts) that are not SFTs. However, Simpson \cite{simpson} proved that every effective Medvedev degree (\ie the Medvedev degree of an effectively closed set) contains a 2D SFT. Therefore, in some sense, not only is the emptiness problem undecidable for 2D SFTs, but their points can be as difficult to compute as possible. We improve this result to the extremely expansive case:

\begin{theorem}
Every effective Medvedev degree contains an extremely expansive 2D SFT. In other words, for every effectively closed set $Z\subseteq \A^{\MM}$, there exists an extremely expansive 2D SFT $Y$ that is Medvedev equivalent to $Z$.
\end{theorem}

%For every effectively closed set $Z\subseteq \A^{\MM}$
%, there exists
In fact, we prove something stronger, giving a complete characterization of the so-called \dfn{Turing degrees} of $Y$ relative to those of $Z$, but it is not necessary to go into these details here.

The next result is of a dynamical flavour and it does not concern extremely expansive SFTs, but sequences of SFTs with a common rational direction of expansiveness. It also uses the notion of simulation, which is of central importance in the proofs of the previous results and, in general, for the whole thesis, even though it wasn't mentioned until now.

We say that subshift $X \subseteq \A^{\Z^2}$ \dfn{simulates} subshift $Y \subseteq \B^{\Z^2}$ with parameters $(S,T)$ if there exists a $\B$-colouring of the $S \times T$ blocks of $X$ with the following property: Every configuration of $X$ can be partitioned in a unique way into $S \times T$ rectangles such that when we color these rectangles with the $\B$-colouring we obtain a configuration of $Y$. Inversely, every configuration of $Y$ can be obtained in this way. 

This is weaker than the notion of simulation that we actually use, but it follows from it,  is enough to describe the result and is much easier to describe. It corresponds to the definitions in \cite{drs}.

It was proved in \cite{laffite} that for every computable sequence of SFTs, there exists an SFT that simulates all of them. This is a surprising and really strong result. We prove a version of it in the case where all the SFTs of the sequence have a common, rational expansive direction (which without loss of generality we assume to be the horizontal one):

\begin{theorem}
Let $X_0,X_1, \ldots$ be a computable sequence of 2D SFTs such that $0 \in \NE(X_i)$, for all $i \in \N$. Then, there exists a 2D SFT $X$ such that $X$ simulates $X_i$ for all $i \in \NE$ and $0 \in \NE(X)$.
\end{theorem}

We note that there cannot exist a 2D SFT with an expansive direction that simulates all 2D SFTs with the same expansive direction, because this would imply the decidability of the emptiness problem for extremely expansive SFTs, according to an argument of Hochman \cite{hochmanuniv}.

The final result of the thesis answers a natural question which arises immediately after the construction of an extremely expansive SFT. As stated already, the unique non-expansive direction of the SFT that we construct is the vertical one. Which other directions can be the unique direction of non-expansiveness for 2D SFTs? Obviously, we can achieve any rational direction by rotating with elements of $SL_2(\Z)$, but can we do more? More generally, what are the sets of directions that can be the set of non-expansive directions of a 2D SFT?

Hochman \cite{nexpdir} proved that for general 2D subshifts (not necessarily of finite type, or even effective), any closed set of directions can be the set of non-expansive directions, while any direction can be the unique direction of non-expansiveness. Recall that Boyle and Lind proved that the sets of non-expansive directions must be closed, so it turns out that in the case of general subshifts this necessary topological condition is also sufficient.

In the case of SFTs, there is an additional necessary condition, namely that the set of non-expansive directions be \dfn{effectively closed}, which is equivalent to saying that its complement is the union of an increasing, computable sequence of open intervals. It turns out that this condition is necessary and sufficient for 2D SFTs:

\begin{theorem}
A set of directions $\NE$ is the set of non-expansive directions of a 2D SFT if and only if it is effectively closed. More precisely, a direction $l$ is the unique direction of non-expansiveness of a 2D SFT if and only if it is computable.
\end{theorem}

This answers Question~11.2 in Boyle's Open Problems for Symbolic Dynamics \cite{opsd}.

Using our methods, we could easily prove Theorems~1-3 for SFTs whose unique direction of non-expansiveness is $l$, where $l$ is any computable direction. This is a stronger version of the results, which we do not prove for lack of space. In any case, once one has mastered our method, it is possible to prove various new results and variants of already proved ones. Since this method is as important (if not more) as some of our results, it is probably worth saying some words about its history, too. 

It is the so-called \dfn{fixed-point tile} or \dfn{self-simulating} method for constructing 2D SFTs. It was firstly described by Kurdyumov \cite{kurdyumov} in order to give a counterexample to the Positive Rates conjecture, even though only a sketch of a proof was included in this paper. It was G\'{a}cs \cite{gacs1} who elaborated Kurdyumov's idea into a full proof of the positive rates conjecture and formalized the notion of a hierarchy of simulating SFTs (he talks about 1D cellular automata, but this does not make a big difference). Later, he significantly improved his construction and the result in a notoriously lengthy and difficult paper \cite{gacs}. Gray's reader guide to that paper \cite{gray} and the description therein of self-simulation and the problems one encounters when trying to construct a self-simulating SFT are also a very useful exposition of the ideas of G\'{a}cs and Kurdymov. It was not until the work of Durand, Romashchenko and Shen \cite{drs} that the method became accessible to a broader mathematical audience. They work in the framework of 2D SFTs, which allows for a more clear, geometrical description of the basic ideas.

G\'{a}cs' construction did not have any direction of expansiveness, because it was a non-reversible cellular automaton. Nonetheless, it had the horizontal direction as a direction of \xpr{semi-expansiveness}. On the other hand, the construction of Durand, Romashchenko and Shen did not have neither directions of expansiveness neither directions of \xpr{semi-expansiveness}. A large part of this thesis consists in making their construction expansive in the horizontal direction. We need to introduce some tricks in order to do this, but once we achieve it, then self-simulation and a previous result of Hochman immediately give an extremely expansive aperiodic SFT. Something similar was also done in \cite{zinoviadis1}, but the construction of that paper was significantly easier because we dealt with non-reversible cellular automata, so that we only needed a direction of \xpr{semi-expansiveness}. Our current construction can be seen as an improvement of the construction of that paper, and using it we can easily retrieve its main result, which was a characterization of the numbers that can appear as the topological entropy of a (not necessarily reversible) CA.

One thing that all the constructions have in common, including ours, is that they are complicated and rather difficult to explain (for the writer) and understand (for the reader). This is unavoidable, in some degree, and the author's personal opinion is that there does not exist a \xpr{perfect} way to write them. Either the exposition is very formal, covering all details and defining every little thing, which is the road that we have chosen, or the construction is informal, in which case it is not clear what exactly the constructed SFT is, over which alphabet it is defined etc., which is the choice made by Durand, Romashchenko and Shen. Taking the middle road, as was more or less done by G\'{a}cs, does not help very much, either.

Our opinion is that the best thing is to be familiar with \emph{all} the constructions and use them accordingly. On the one hand, the constructions of Durand, Romashchenko and Shen are convincing for someone already familiar with the technique and they allow to explain a new idea concisely and efficiently, as was recently done in \cite{drs2}, while on the other hand our more formal presentation can be used to acquire mastery with the technique by dealing with all the unexpected little problems that arise during the construction and to convince those people who want to understand all the details.

Let us now describe the structure of the thesis:

In Chapter~\ref{s:prelim}, we give the basic definition that we will need throughout the paper. In Chapter~\ref{s:simul}, we define the precise notion of simulation that we will use and give some of its properties. We believe that some of the results of this chapter are of independent interest. In Chapter~\ref{c:programming}, we describe a pseudo-programming language that will be used to describe 2D SFTs in a concise way. In Chapter~\ref{construction}, we construct a family of SFT (which depend on the parameters $S,T$) with $0$ as a direction of expansiveness which are, in some sense, universal: They can simulate every SFT with $0$ as a direction of expansiveness, provided that its alphabet size is small compared to $S,T$ and it can be computed fast compared to $S,T$. This family of SFTs is of great importance for all subsequent constructions. This is the part of the thesis where we modify the construction of Durand, Romashchenko and Shen so as to make it reversible. In Chapter~\ref{s:hierarchy}, we prove Theorems~1 - 4. The constructions and the proofs all follow the same pattern, but we give as many details as possible for all of them for reasons of completeness. Finally, in Chapter~\ref{sec:expdir}, we prove Theorem~5. This proof is a modification of the proof of the result in \cite{nexpdir}. We try to explain what are the differences between that construction and ours and why the changes that we make are necessary.

Finally, let us mention that all of the aforementioned results have been obtain in collaboration with Pierre Guillon during various visits by him in Turku as well as of the author in Marseille. Currently, a series of joint papers is under construction that will contain even more applications of our method. Theorem~1 has also appeared in \cite{zinoviadis2}, even though because of lack of space, most of the details of the construction do not appear in that paper.

\chapter{Preliminaries}\label{s:prelim}

\section{Basic definitions}\label{s:encoding}
We will denote by $\Z$, $\N$, $\Ns$, $\Q$ and $\R$ the sets of integers, non-negative integers, positive integers, rational and real numbers, respectively, by $\co ij$  and $\cc ij$ the integer intervals $\{i,\ldots,j-1\}$ and $\{i,\ldots,j\}$, respectively, while $[\varepsilon,\delta]$ will denote an interval of real numbers. If $f,g \colon \N \to \N$, then we will use the classical $f \in O(g)$ notation to denote that $f(n) \le c g(n)$, for some constant $c$ and all $n \in \N$.

If $f \colon X \pto Y$ is a partial function, then its \dfn{domain} $\dom(f) \subseteq X$ is the set of elements of $X$ whose image through $f$ is defined. Two partial functions are equal when they have the same domain and they agree on their common domain. If $f \colon X \pto Y$ and $g \colon Y \pto Z$ are partial functions, then $g\circ f \colon X \pto Z$ is the partial function defined in the usual way (\ie $g(f(w))$ does not exist if either $w \notin \dom(f)$ or $f(w) \notin\dom(g)$).
A \dfn{partial permutation} is a bijection over its domain onto its range, \ie an injective partial map.
In the following, when defining a partial function, it will be implicit that any non-treated argument has an undefined image, and that saying that two partial functions are equal means in particular that their domains are the same.
If $Z\subset X$ and $f:X\pto Y$, we may abusively consider $f\restr{Z}$ as a partial map from $X$ to $Y$ whose domain is $Z\cap\dom(f)$.

For $m\ge n$, $\vec T\defeq(T_i)_{0\le i<m}$ and $\vec t\defeq(t_i)_{0\le i< n}$ such that for all $i$, $0\le t_i<T_i$, we note $\anib[\vec T]{\vec t}\defeq\sum_{0\le i < n}t_i\prod_{0\le j \le i}T_j$ the numeric value represented by the adic representation $\vec t$ in base $\vec T$. In general, $T_i$ and $t_i$ can belong in $\R$, not necessarily in $\N$.
By convention, if $\vec t$ has length $0$, then $\anib[\vec T]{\vec t}\defeq0$.
Similarly, for a sequence $\seq T\defeq(T_i)_{i\in\N}$, we note $\anib[\seq T]{\vec t}\defeq\anib[T_{\co0n}]{\vec t}$. For a sequence $\seq t \defeq(t_i)_{i \in \N}$, we note $\anib[\seq T]{\seq t}\defeq\lim\anib[T_{\co0n}]{t_{\co{0}{n}}}$, when this limit exists.

An \dfn{alphabet} is any finite set, whose elements are often called \dfn{symbols}.
If $\A$ is an alphabet, $\A^*\defeq\bigcup_{n\in\N}\A^n$ denotes the set of finite \dfn{words} over $\A$, and $\A^{**} \defeq \bigcup_{m\in\N}{(\A^*)^m}$ the set of finite tuples of words. (Notice that the notation $\A^{**}$ is a little ambiguous as it could also stand for the set $\bigcup_{m\in\N}{(\A^m)^*}$. Obviously, the two interpretations are isomorphic, but they are different objects.) The \dfn{empty} word is denoted by $\motvide \in \A^*$.

If $w \in \A^n$, we write $w=w_0\cdots w_{n-1}$, and call $\length{w}\defeq n$ the \dfn{length} of $w$. 
If $\vec u\in (\A^*)^m$, we write $\vec u=(u_0,\ldots,u_{m-1})$, and $\length{\vec u}\defeq(\length{u_0},\ldots,\length{u_{m-1}})$. %, and call $m$ the number of \dfn{fields} of $u$. %Note that the numbering of fields starts from $1$ (and not from $0$ as for words). %This feels more natural to us, although, of course, it doesn't make any difference.) 
For every $i \in \N$, we define the projection $\pi_i$ as a partial function $\pi_i \colon \A^{**} \pto \A^{*}$: $\pi_i(\vec u)=u_i$ if $\vec u \in (\A^{*})^m$ with $m\geq i$ (and $\pi_i(\vec u)$ is undefined otherwise).
A \dfn{field} is a projection $\pi_i$ together with a label \field, written in type-writer form. %We often denote by $\vec u.\field=\pi_i(\vec u)$ the value of \field\ in $\vec u$.	
The notion of fields is simply a convenient way of talking about tuples of words. The names of the fields will be chosen so as to reflect the role that the field plays in the construction.

We note $\N^*\defeq\bigcup_{m\in\N}\N^m$ the set of integer tuples of any dimension, where $m$ is the \dfn{dimension} of the tuple $\vec k\defeq(k_0,\ldots,k_{m-1})\in\N^m$. Let $\A^{\vec{k}}\defeq\A^{k_0}\times\ldots\times\A^{k_{m-1}}$;
any subalphabet of $\A^{\vec{k}}$ is said to have \dfn{constant lengths}.
%We also note $\A^{\le\vec k}\defeq\bigcup_{\vec{k'}\le\vec k}A^{\vec{k'}}$, where $\le$ is extended to a partial order over tuples of the same dimension componentwise.

We will mainly use the special alphabets $\haine n\defeq\{0,\ldots,n-1\}$, for $n\in\{2,3,4,5\}$. Of course, instead of $\haine5$ we could use any alphabet with $n$ letters. However, since some letters will have a fixed role throughout the thesis, it is better to fix the notation and get used to these roles.

The non-negative integers can be easily embedded into $\haine2^*$ thanks to the injection $n\mapsto\anib n$ which gives the shortest binary representation of $n \geq 1$. $\norm n\defeq\length{\anib n}=\spart{\log_2n}+1$ is the \dfn{length of $n$}. By definition, $\anib{0} \defeq \motvide$ and $\norm{\motvide} \defeq 0$.  Inversely, if $u \in \haine2^*$, then $\bina u$ is the number represented by $u$ in the binary representation system: for all $ u\in\haine2^*,\anib{\bina u}$ is the suffix of $u$ that is obtained after removing the initial $0$s. (The \xpr{lower bar} is applied before the \xpr{top bar}.)

We will also need to embed some finite sets in $\haine2^*$. For instance, we will say that $\{-1,+1\}$ is $\haine2$ by identifying $-1$ with $0$ and $+1$ with $1$. Finite alphabets of bigger cardinality can be embedded into $\haine2^k$, for some suitable $k$.

%In order to use various objects, in particular integers and tuples of words, in our computations, let us define canonical encodings for them into ternary words.
%\begin{itemize}
%\item If $n \in \N$, $\Chi[\N]n \in \deux^{\norm n}$ is the shortest binary representation of $n$, where $\norm n\defeq\spart{\log_2n}$ is its length.
%\item We will assume that any finite set $Q$ of cardinal $\card Q\le2^{\norm Q}$, where $\norm Q\defeq\spart{\log_2{\card Q}}$, is endowed with an injection \[\appl Q{\deux^{\norm Q}}s{\Chi[Q]s}\].
%For instance, the set $\{-1,+1\}$ can be encoded into $\deux$ through a map $i\mapsto\Chi[Q]i$ by identifying $-1$ with $0$.
%\item If $u=(u_0,\ldots,u_{m-1})\in\quatre^{*m}$ for some $m\in\N$, then $\Chi[\quatre^{**}]u$ is defined as the concatenation $\double{u^0}\$\double{u^1}1\$\ldots\$\double{u^{m-1}}\$$, where $v\mapsto\double v$ is a word monoid morphism from $\quatre^*$ to $\deux^*$ such that $\double0\defeq00$, $\double1=11, \double\$=10$ and $\double\sharp=01$.
%\end{itemize}
%All of these encodings are injective.
%In particular, we note $\bina u$ the number represented by word $u$ in the binary representation system.
%In general, the image set is in $\trois^*$
%When it is clear from the context, we may write $\Chi u$ while leaving the domain implicit.

Now, in the perspective of computing functions with many arguments, we are going to use symbol $2$ to encode tuples into words.
If $\vec u\in(\haine5^*)^m$ for some $m\in\N$, then $\Chi{\vec u}$ is defined as the concatenation 
\begin{equation*}
\Chi{u_0}\Chi{u_1}\ldots\ldots\Chi{u_{m-1}}\in\haine3^*,
\end{equation*}
where $\Chi v\defeq2\double{v}$, and $v\mapsto\double v$ is some monoid injection (\ie code) from $\haine5^*$ to $\haine2^*$. In this paper, we will use the code defined by $\double0\defeq000$, $\double1\defeq001$, $\double2\defeq010$, $\double3\defeq011$, $\double4\defeq100$.
Note that the structure of the encoding of word tuples depends only on $\length{\vec{u}}$. %, as will be formalized by Remark~\ref{encodings}.
We can also define $\Chi{\seq u}\defeq\Chi{u_0}\Chi{u_1}\ldots\in\haine3^\N$ for $\seq u\in\haine5^{\N}$. %If $n \in \N$, then $\Chi{n} \defeq \Chi{\anib{n}}$.

Let us now prove a basic fact about $\Chi{\cdot}$. Namely, for every $\vec{k} \in \N^*$, there exists an easily computable function that gives the positions of the $2$s in encodings of $\haine5^{\vec{k}}$ and the positions of the encodings of the components of a letter. %In fact, there exists a single function that does the work for all $\vec{k} \in (\N^{*})^{M}$, where $M \in \N$. We will give a slightly more general version where, in some sense, we pre-compose with another polynomially computable function that takes as input only the lengths of the encoded letters and outputs some new lengths. This will be needed in general hierarchical simulations where the alphabets change in every level of the hierarchy.

\begin{fact}\label{f:encodings}
Let $M \in \N$ and $\vec{k} \in \N^M$. For all $0 \le i <M$, let us define $l_{\vec{k},i}\defeq3\sum_{j=0}^{i-1}{k_j}+i$. Then, for all $\vec{u} \in \haine5^{\vec{k}}$:
			\begin{enumerate}
				\item $\norm{\Chi{\vec u}}=l_{\vec{k},M}$,
				\item ${\Chi{\vec u}}_{\co{l_{\vec{k},i}}{l_{\vec{k},i+1}}}= 2 \double{\pi_i(\vec{u})}$
				%\item if $v \colon \haine5^{**} \to \N$ is a valuation, then ${\Chi{\vec u}}_{\cc{l_{\vec{k},i}}{l_{\vec{k},i}+3v(\vec{u})}} 
				%= \Chi{\pi_i(\vec w)_{\co0{v(\vec{u})}}}
				%=2\double{\pi_i(\vec u)_{\co0{v(\vec{u})}}}$.
			\end{enumerate}
%Moreover, the function $(\vec{k},\vec{u}) \rightarrow l_{\vec{k},i}(\vec{u}) \colon (\N^*)^M \times (\haine5^*)^M \to \N$ is polynomially computable.
\end{fact}
These statements correspond to what Durand, Romashchenko and Shen refer to as ``the TM know the place where such and such information is held in the encoding''.

Symbol $3$ will be used in Subsection~\ref{s:turing} to encode the start and the end of the tape of a Turing machine.

Symbol $4$ will be used in order to construct alphabets with constant lengths.
In the computation, we indeed want words of various lengths to be able to represent the same objects.
For this, we define $\sh[l]u\defeq4^{l-\length u}u$, for every $l\in\N$ and $u\in\haine4^*$ with $\length{u} \le l$ ($\sh[l]u$ is undefined otherwise).
For instance, $\sh[\norm n]{\anib n}=\anib n$ for any integer $n\in\N$, and the encoding $\sh[l]\motvide=4^l$ of the empty word is a sequence of $4$s.
It is clear that the partial function \[\pappl%[{{\sh[\cdot]{\cdot}}}]
{\N\times\haine4^*}{4^*\haine4^*}{(l,u)}{\sh[l]u}\] is injective (over its domain) and surjective; let us write $\hs w\in\haine4^*$ for the longest suffix in $\haine4^*$ of a word $w\in4^*\haine4^*$, in such a way that $\hs{\sh[l]u}=u$ for any $l\geq \length{u}$ and $u\in\haine4^*$.
These two maps can be adapted to vectors in the obvious way: $\sh[\vec k]{\vec u}\defeq(\sh[k_0]{u_0},\ldots,\sh[k_{m-1}]{u_{m-1}})$ for any $\vec k\in\N^m$, $m\in\N$ and $\vec u\defeq(u_0,\ldots,u_{m-1})\in\haine4^{*m}$.
Note that this is defined if and only if $\vec k\ge\length{\vec u}$.
Similarly, $\hs{\vec w}\defeq(\hs{w_0},\ldots,\hs{w_{m-1}})$ for any $\vec w\defeq(w_0,\ldots,w_{m-1})\in(4^*\haine4^*)^m$.

%If $f \colon \haine4^{**} \to X$ is a function, then it can be extended to a function $\sh{f} \colon \haine5^{**} \to X$ in the following way: $\sh{f}(x)=\f(\hs{x})$, for all $x \in \haine4^{**}$ (and $\sh{f}$ is undefined otherwise).

Recall that a partial permutation is simply an injective partial map. If $\alpha:\haine4^{**}\pto\haine4^{**}$ is a partial permutation that preserves the number of fields (\ie $\alpha(\haine4^*)^l \subseteq (\haine4^*)^l$ for all $l \in \N$), we can transform it into an equivalent permutation that also preserves the lengths:
\[\pappl[\sh\alpha]{(4^*\haine4^*)^*}{(4^*\haine4^*)^*}{\vec w}{\sh[\length{\vec w}]{\alpha(\hs{\vec w})}~.}\]
\begin{remark}\label{sharpization}~
\begin{itemize}
\item For any $\vec k\in\N^*$, $\sh{\alpha}$ is also a partial permutation.
\item The restriction of $\alpha$ to any subalphabet is implemented by that of $\sh\alpha$ to large enough words:%, in the following sense:
\[\forall \vec{u}\in\haine4^{**},\forall
\vec k\ge\max\{\length{\vec{u}},\length{\alpha(\vec{u})}\},\sh\alpha(\sh[\vec k]{\vec{u}})=\sh[\vec k]{\alpha(\vec{u})}~.\]
\end{itemize}\end{remark}

\begin{proof}
For the first part, assume that $\sh\alpha(\vec w)=\sh\alpha(\vec{w'})$. This implies that $\length{\vec w}=\length{\vec{w'}}$. In addition, $\alpha(\hs{\vec{w}})=\hs{\sh\alpha(\vec w)}=\hs{\sh\alpha(\vec w')}=\alpha(\hs{\vec{w'}})$. Since $\alpha$ is a partial permutation, this implies that $\hs{\vec{w}}=\hs{\vec{w'}}$. Therefore, $\vec{w}=\sh[\length{\vec{w}}]{\hs{\vec{w}}}=\sh[\length{\vec{w'}}]{\hs{\vec{w'}}}=\vec{w'}$.

For the second part, let $\vec{u}\in\haine4^{**}$ and $\vec k\ge\max\{\length{\vec{u}},\length{\alpha(\vec{u})}\}$. Then, $\sh[\vec{k}]{\vec{u}}$ and $\sh[\vec{k}]{\alpha(\vec{u})}$ exist and $\length{\sh[\vec{k}]{\vec{u}}}=\length{\sh[\vec{k}]{\alpha(\vec{u}})}=\length{\vec{k}}$. Therefore, $\sh{\alpha}(\sh[\vec{k}]{\vec{u}})=\sh[\vec{k}]{\alpha(\hs{\sh[\vec{k}]{\vec{u}}})}=\sh[\vec{k}]{\alpha(\vec{u})}$.
\end{proof}

In the rest of the paper, we will often implicitly use Remark~\ref{sharpization} both to construct partial permutations that preserve the lengths of the fields, as well as to state and prove things about them. It allows us to describe the behaviour of a partial permutation $\alpha$, and then translate this result into the behaviour of $\sh{\alpha}$, provided that the lengths of the fields are sufficiently large, thus omitting the (confusing) $\hs{\cdot}$ and $\sh{\cdot}$ symbols. 

Let $i_1, \ldots, i_l$ be a set of fields, and $w\in\haine5^*$. Then, 
\begin{equation*}
\emp[w]{i_1,\ldots,i_l} \defeq \set{\vec{u}}{\haine5^{**}}{\hs{\pi_{i_k}(u)}=w, \text{ for } k=1, \ldots,l}
\end{equation*}
 is the set of all symbols that have fields $i_1,\ldots, i_l$ equal to $w$ (up to the application of $\hs{\cdot}$).
If $n\in\N$, let 
\begin{equation*}
\emp[n]{i_1,\ldots,i_l}\defeq\set{\vec u}{\haine5^{**}}{\bina{\hs{\pi_{i_k}(u)}}=n, \text{ for } k = 1, \ldots, l}
\end{equation*}
be the set of all symbols who have the values $n$ (in binary form) in the fields $i_1,\ldots,i_l$.

\section{Computation}
\subsection{Turing machines}\label{s:turing}

The reader is assumed to be familiar with classical concepts in computability theory. We just fix some terminology and give a variant of a definition of Turing machines, imposing some additional technical restrictions which, however, do not restrict the computational power.

A \dfn{Turing machine} (TM) is a partial (\xpr{global}) map $\M$ from $\haine4^\Z\times Q\times\Z$ into itself, where $Q\subset\haine2^*$ is a finite set of \dfn{states} containing the \dfn{initial} state $0$ and the \dfn{accepting} state $\motvide$, and depending on a partial \dfn{transition map} $\delta_\M:\haine4\times Q\setminus\{\motvide\}\pto\haine4\times Q\times\{-1,+1\}$ such that: %, and $q_i,q_f\in Q$ are the \dfn{initial} and \dfn{accepting} states, respectively.

%\left|\everymath{\displaystyle\everymath{}}\begin{array}{lll}\end{array}\right.
\[\M(z,q,j)=\soit{(z,q,j)&\si q=\motvide\\(z',q',j')& \text{ otherwise, where } (z'_j,q',j'-j)=\delta_\M(z_j,q)\\ & \text{ and } z'_i=z_i, \forall i\ne j~,}\]
for any $(z,q,j)\in\haine4^\Z\times Q\times\Z$, which will sometimes be called a \dfn{machine configuration}, the first component being the \dfn{tape content}, the second the (head) \dfn{internal state}, the third the \dfn{head position}.

The model of TM that we use satisfies the following assumptions, which, as can be easily seen, do not restrict the computational power of TM.
\begin{itemize}
\item There is only one tape, from which the TM reads the input and on which it writes the output.
\item The internal states are words of $\haine2^*$ (this is just a semantic restriction).
\item All machines have the same initial and accepting states $0$ and $\motvide$, respectively.
\item T%hough no transition is defined from the accepting state, t
he global map is still defined after having accepted, and is then equal to the identity.
\item There is no precise rejecting state (instead, we use undefined transitions over non-accepting states).
\item In every accepting transition, the head disappears and moves to the right. In other words, every accepting transition is of the form $\delta(q,a)=(a',\motvide,+1)$. This is a technical assumption which simplifies the construction of an IPPA that simulates $\M$ in Section~\ref{Jarkko}.
%\item All accepted inputs and all outputs are assumed of the form $w\in\trois^*$ (the $\sharp$ being reserved for delimiting it).
\end{itemize}

%\TODO{define $\pinf u.w\uinf v$? or avoid?}
We denote by $\M^t$ the $t$'th power of the (global) map $\M$. If %$(z,q,j)$ and $(z',q',j')$ are machine configurations and
%there exists $t \geq 0$ such that $\M^t(%z,q,j)=(z',q',j')$, then we write that $(z,q,j) \rightarrow (z',q',j')$. This defines a partial ordering of the machine configurations. If $(
$\M^t(\pinf 3. \Chi{\vec{u}}3^{\infty},0,0)= %\rightarrow
 (\pinf 3. \Chi{\vec{u'}}3^{\infty},\epsilon,j)$, for some $t \in \N, j \in \Z$, then we say that $\M$ \dfn{halts} over (or \dfn{accepts}) \dfn{input} $\vec{u}\in\haine5^{**}$, and \dfn{outputs} $\vec{u'}\in\haine5^{**}$, and we define $f_\M(\vec{u})\defeq \vec{u'}$ and $t_\M(\vec{u})$ as the minimal $t$ for which this holds (if this never holds, or if $\pinf 3. \Chi{\vec{u}}3^{\infty}$ is rejected, then $t_\M(\vec{u})$ is undefined). 

Notice that $f_{\M}(\vec{u})$ is well-defined, since when the accepting state $\motvide$ appears, the machine configuration is no more modified.

We say that $\M$ \dfn{computes} the partial map $f_\M:\haine5^{**}\pto\haine5^{**}$, with time \dfn{complexity} \[\appl[t_\M]\N\N n{\max_{\length{\Chi{\vec{u}}}=n}t_\M(\vec{u})~,}\]

where, by definition, the $\max$ is taken only over \emph{accepted} inputs. $t_{\M}$ is well-defined since there are only finitely many accepted inputs of each length.

\subsection{Computability}\label{ss:comput}

A partial function $f:\haine5^{**}\pto\haine5^{**}$ is called \dfn{computable} if there exists a TM $\M$ such that $f=f_\M$.
Recall that integers (and finite sets) can be identified %(through the encoding $\anib.$ in Subsection \ref{s:encoding}) 
to words, hence allowing us to talk about computable maps between Cartesian products involving $\N$ and finite sets.
We also say that a set $X \subseteq \haine4^{**}$ is \dfn{computable} if its characteristic function $\iota_X\colon \haine4^{**} \to \haine2$ is computable, and that it is \dfn{computably enumerable} if it is the domain of a computable %(constant) 
function.
We will say that a partial function $f:X\pto\haine5^{**}$, with $X\subset\haine5^{**}$ is \dfn{computable} if both $X$ and the extension of $f$ to $\haine5^{**}$ (by not defining images outside of $X$) are computable.

A partial function $\Phi:\haine2^\N\pto\haine2^{\N}$ is called \dfn{computable} if there exists a TM $\M$ such that $x\in\dom(\Phi)$ if and only if for all $n\in\N$, there exists $m\in\N$ such that $f_\M(x_{\co{0}{m}},n)$ is defined, in which case it is equal to $\Phi(x)_n$. Finally, by parametrizing $\Z$ with $\N$, we can talk about computable functions $\Phi:\haine2^\Z\pto\haine2^{\Z}$. An equivalent definition is that  $\Phi:\haine2^\Z\pto\haine2^{\Z}$ is computable if there exists a TM $\M$ such that  $x\in\dom(\Phi)$ if and only if for all $n\in\N$, there exists $m\in\N'$ such that $f_\M(x_{\co{-m}{m}},n)$ is defined, in which case it is equal to $\Phi(x)_n$.

Since $\R$ can be identified with $\haine2^{\N}$, we can also talk about computable functions of real numbers. A partial function $\Psi:\R\pto\R$ is \dfn{computable} if there exists a computable function $f \colon \R \times \N \to \Q$ with the following property: $\abs{\Psi(x)-f(x,n)} < 2^{-n}$. This is the classical definition of computability for real functions and it says that we can compute better and better approximations of $x$. 

If $\M$ is a TM, let 
\begin{equation*}
\X_{\M}\defeq \set {z}{\haine2^\N}{\forall t\in\N,\M^t(\uinf3.z,0,0) \text{ exists  and is not in } \haine4^{\Z}\times\{\motvide\}\times\Z}
\end{equation*}
  be the set of one-sided binary sequences over which $\M$ runs for an infinite amount of time.
We say that a subset $X\subset\haine2^{\N}$ is \dfn{effectively closed} (or $\Pi^0_1$) if $\Chi{X}=\X_{\M}$ for some TM $\M$, or equivalently if the set of words that do not prefix any sequence in it is computably enumerable.
This can be extended to sets of sequences that can be encoded with words, in particular over finite alphabets: a subset $X\subset\prod_{t\in\N}\A_t$, where $\A_t$ is a finite subalphabet of $\haine5^*$, is \dfn{effectively closed} if $\Chi X=\X_{\M}$ for some program $\M$ (we encode every finite alphabet with $\haine2^k$, for some suitable $k$ which depends on $t \in \N$).
%We can also allow distinct words to encode the same element of the set (this corresponds to a quotient topology).

$\M$ is called \dfn{polynomial} if $t_{\M} \in O(P)$, for some polynomial $P$. A partial function $f$ is called \dfn{polynomially computable} if $f_{\M}=f$ for some polynomial TM $\M$. It is easy to see that the class of (polynomially) computable functions with this version of TM corresponds to the classical one.
Analogously, $X$ is a polynomially computable set if its characteristic function $\iota_X$ is polynomially computable.
We say that a function (or sequence) $f$ is \dfn{polynomially checkable} if it can be computed in time $O(P(\log{f}))$, for some polynomial $P$. The terminology comes from the fact that even though $f$ might not be polynomially computable, its graph (\ie the set of pairs element-image) is a polynomially computable set. For example $f(n) = 2^{2^n}$ is a polynomially checkable sequence even though it is not polynomially computable.

Instead of a universal TM, we use the following essentially equivalent:
\begin{fact}
There exists an injection that associates to each TM $\M$ a \dfn{program} $p_{\M} \in \haine4^{*}$ such that
if we denote by $Q_p$ the state set of the TM corresponding to program $p$,
then 
\begin{itemize}
\item The language $\sett{p_{\M}}{\M \text{ is a TM}} \subseteq \haine4^*$ is polynomially decidable.
\item The characteristic function $(p,q)\mapsto\iota_{Q_p}(q)$ that checks whether $q \in Q_p$ is polynomially computable.
\item The \xpr{universal} transition rule \[\pappl[\delta_\U]{\haine4\times\haine4^*\times\haine4^*}{\haine4\times\haine4^*\times\{-1,+1\}}{(a,q,p_\M)}{\delta_\M(a,q)~}\]
is polynomially computable.
\item In addition, $\card{Q_p} \le \length{p}$. (We can assume that $p$ contains a list of the states of $Q_p$.)
\end{itemize}
\end{fact}

We will use the following notations: If $p$ is the program of a TM that computes a reversible function $f$, then $p^{-1}$ will denote the program of the inverse function $f^{-1}$ (it will always be computable in our constructions). Also, $t_p$ and $\X_p$ will be used to denote $t_{\M_p}$ and $\X_{\M_p}$, where $\M_p$ is the TM that corresponds to the program $p$.

The first examples of polynomially computable functions, which will be most useful in the sequel, are the encodings presented in Subsection~\ref{s:encoding}.
Clearly, $\sh[\cdot]\cdot$ and its (right) inverse $\hs\cdot$ are polynomially computable.
Moreover, the projections $\pi_i:\haine5^{**}\to\haine5^*$, for $i\in\N$, are polynomially computable and so are the functions $(\vec{k},i) \to l_{\vec{k},i}$ (as defined in Fact~\ref{f:encodings}) and $\Chi{\cdot}$. %It is easy to see this, keeping in mind that in our model both the input and the output are encoded with $\Chi{\cdot}$.

\subsection{Degrees}
%\begin{definition}
%Assume that countable sets $\MM$ and $\MM'$ (typically $\Z$) are endowed with a fixed \dfn{computable representation}, that is a bijection $\kappa:\N\to\MM$. %surjection $\haine2^*\pto\M$ with computable domain, that we will abusively note $\bina.$, since it corresponds to the previously defined $\bina.$ for $\N$; it is not much harder to define for $\Z$ (say one bit initial bit gives the sign).
%\begin{itemize}\item

% A partial function $\Phi:\haine2^\N\pto\haine2^{\N}$ is called \dfn{computable} if there exists a TM $\M$ such that $x\in\dom(\Phi)$ if and only if for all $n\in\N$, there exists $m\in\N'$ such that $f_\M(x_{\co{0}{m}},n)$ is defined, in which case it is equal to $\Phi(x)_n$.

In the following,  $\MM$ and $\MM'$ can stand for either $\N$ or $\Z$.

Two sets $X,Y \in \haine2^{\MM}$ are \dfn{computably homeomorphic} if there exists a computable bijection between them.

 We say that $d\in\haine2^{\MM'}$ is \dfn{Turing-reducible} to $c\in\haine2^{\MM}$ if $d = \Phi(c)$, for some computable function $\Phi$.
This yields a preorder over configurations, whose equivalence classes are called \dfn{Turing degrees}. If $d$ is Turing-reducible to $c$, then in a computational sense, $c$ is more complicated than $d$. A \dfn{cone} over degree $d$ is the set of Turing degrees that are higher than $d$.
 
 Moreover, we say that subset $Y\subset\haine2^{\MM'}$ is \dfn{Medvedev-reducible} to subset $X\subset\haine2^{\MM}$ if there is a computable partial function $\Phi:\haine2^{\MM}\pto\haine2^{\MM'}$  such that $\dom(\Phi) \supseteq X$ and $\Phi(X)\subseteq Y$. This also yields a pre-order over sets, whose equivalence classes are called \dfn{Medvedev degrees}.
%\item
Finally, we say that subset $Y\subset\haine2^{\MM'}$ is \dfn{Mu\v cnik-reducible} to subset $X\subset\haine2^\MM$ if every point of $X$ is Turing-reducible to some point of $Y$ (but not in a uniform way, as in Medvedev-reducibility). This again yields a pre-order over sets, whose equivalence classes are called \dfn{Mu\v cnik degrees}.
%\end{itemize}\end{definition}

Medvedev and Mu\v cnik degrees of a set are an attempt to formalize the notion of how computationally difficult it is to compute a point of the set.
Of course, computable homeomorphism implies having the same Turing degrees, which implies Medvedev-equivalence, which in turns implies Mu\v cnik-equivalence.

%If $p$ is a program, let $\X_p\defeq \set {z}{\haine4^\N}{\forall t\in\N,\M_p^t(\uinf3.z,0,0)\notin\haine4^{\Z}\times\{\motvide\}\times\Z}$ be the set of one-sided binary sequences over which $p$ does not halt.
%We say that a subset $X\subset\haine2^{\N}$ is \dfn{effectively closed} (or $\Pi^0_1$) if $\Chi{X}=\X_p$ for some program $p$, or equivalently if the set of word tuples that do not prefix any sequence in it is computably enumerable.
%This can be extended to sets of sequences over sets that can be encoded into words, in particular over finite alphabets: a subset $X\subset\prod_{t\in\N}\A_t$, where $\A_t$ is a finite subalphabet of $\haine5^*$, is \dfn{effectively closed} if $\Chi X=\X_p$ for some program $p$.
%We can also allow distinct words to encode the same element of the set (this corresponds to a quotient topology).
%

We do not get too much into details, but the notion holds in the large setting of effective topological spaces (see for instance \cite{gacshoyruprojas}).
\section{Symbolic dynamics}

$\azd$ is the set of $d$-dimensional \dfn{configurations}, endowed with the product of the discrete topology, and with the \dfn{shift} dynamical system $\sigma$, defined as the action of $\Z^d$ by $(\sigma^\vec{i})_{\vec{i} \in \Z^d}$, where $\sigma^\vec{i}(x)_\vec{k}\defeq x_{\vec{i}+\vec{k}}$ for any configuration $x\in\azd$ and any $\vec{i},\vec{k}\in\Z^d$.

A \dfn{pattern} over a (usually finite) \dfn{support} $D\subset\Z^d$ is a map %(that one can see as a multidimensional word) 
$p\in\A^D$.

Two patterns $u_1\colon D_1 \to \A$ and $u_2 \colon D_2 \to \A$ are called \dfn{disjoint} if $D_1$ and $D_2$ are disjoint shapes of $\Z^d$. If $u_1,u_2$ are disjoint, let $u_1\vee u_2$ be the pattern over shape $D_1\sqcup D_2$ defined by $(u_1\vee u_2)(\vec{i})=u_{j}(\vec{i})$, if $\vec{i}\in D_j$, $j=1,2$. Inductively, we can define $\bigvee_{1\le i\le k}u_i$, when $u_1,\ldots,u_k$ are mutually disjoint pattens.

Let $E,D\subset\Z^2$ be two shapes, and $u\in\A^D$ be a 2D pattern. We denote $u_E$ the restriction of $u$ to $D\cap E$ (this is a pattern with support $D \cap E$). %If $I \subseteq \Z$, then $u_I^h\defeq u_{\Z\times I}$, $u_I^v\defeq u_{I\times\Z}$, $u_l^h\defeq u_{\{l\}}^h$ and $u_l^v\defeq u_{\{l\}}^v$.
%($u_l^h$ is the restriction of $u$ to the $l$'th horizontal line of the plane).

If $I \subseteq \Z$ and $(c_i)_{i \in I}$ is a family of configurations of $\A^{\Z}$, $|(c_i)_{i \in I}$ denotes the (possibly infinite) pattern $u \colon \Z \times I \to \A$ such that $u_{\Z \times \{i\}} = c_i$, for all $i \in I$. Here we implicitly identify patterns on horizontal strips up to vertical translation. Formally, the domains of $u_{\Z \times \{i\}}$ and $c_i$ are not the same. 

If $I = \co{0}{n}$, then $|(c_0,\ldots,c_{n-1})$ is the horizontal strip of width $n$ obtained by putting $c_0, \ldots, c_{n-1}$ on top of each other (in this order). If $I = \Z$, then we obtain a configuration in $\A^{\Z^2}$.

Let $x \in \azd$ and $\vec S\defeq(S_0,\ldots,S_{d-1}) \in \Ns^{d}$. The \dfn{$\vec S$-bulking} (or higher-power representation) of $x$ is the configuration $\bulk[\vec S]x\in%(\A^{\vec S})^{\Z^d}\defeq
(\A^{S_0\times\ldots S_{d-1}})^{\Z^d}$ such that for any $\vec i=(i_0,\ldots,i_{d-1})\in\Z^d$, 
\begin{equation*}
{\bulk[\vec S]x}_{\vec i}\defeq x_{\co{i_0S_0}{(i_0+1)S_0}\times\ldots\times\co{i_{d-1}S_{d-1}}{(i_{d-1}+1)S_{d-1}}}.
\end{equation*}

A ($d$-dimensional) \dfn{subshift} is a closed set $X\subset\azd$ such that $\sigma^\vec{i}(X)=X$ for all $\vec{i}\in\Z^d$.
Equivalently, $X$ is a subshift if and only if there exists a family of patterns $\f\subset\bigcup_{D\subfini\Z^d}{\A^{D}}$ such that 
\begin{equation*}
X=\set x{\A^{\Z^d}}{\forall \vec{i}\in \Z^d,\forall D\subfini\Z^d,\sigma^{\vec{i}}(x)\restr{D}\notin\f}.
\end{equation*}
If $\f$ can be chosen finite, we say that $X$ is a \dfn{subshift of finite type} (SFT).

%When $d=2$, if $\f$ can be chosen in $\A^{\{(0,0),(0,1)\}}\cup\A^{\{(0,0),(1,0)\}}$, $X$ is called a \dfn{nearest-neighbor SFT}. It is then represented, in a non-unique way, in $\parts{\A^2}\times\parts{\A^2}$. (Of course, a similar definition can be given for all  dimensions).%it can be defined by a family $\f$ of connected patterns of size $2$.
%When $d=1$ (though this definition can be easily extended by fixing a computable enumeration), i
If $\f$ can be chosen computably enumerable, then $X$ is called an \dfn{effective subshift}. %Equivalently, it is effectively closed, as a set of letter sequences.

A continuous map $\Phi$ from subshift $X$ to subshift $Y$ is a \dfn{morphism} if $\Phi\sigma=\sigma\Phi$. If it is surjective, then it is a \dfn{factor map}, and $Y$ is a \dfn{factor} of $X$ (this defines a preorder); if it is bijective, then it is a \dfn{conjugacy}, and $X$ and $Y$ are \dfn{conjugate} (this defines an equivalence relation). A subshift $Y \subseteq \A^{\Z^d}$ is called \dfn{sofic} if it is a factor of some SFT, which is then called a \dfn{cover} for $Y$. %Equivalently, there exist an SFT $X \subseteq \A^{\Z^d}$ and an alphabet projection $\pi \colon \A \to \B$, such that $\Pi (X) = Y$, where $\Pi$ is the parallel synchronous application of $\pi$ over all cells in the configuration.

A configuration $x \in \A^{\Z^d}$ is called \dfn{periodic} with \dfn{period} $\vec{j} \neq \vec{0}\in \Z^d$ if $\sigma^\vec{j}(x)=x$. A subshift $X$ is called \dfn{aperiodic} if it does not contain any periodic configurations.

Abusing notation, we use the notations $\emp[w]{i_1,\ldots,i_l}$ and $\emp[n]{i_1, \ldots, i_l}$ (where $w \in \haine5^{**}$ and $n \in \N$) also for configurations. For example, if $c \in (\haine5^{**})^{\Z}$, we will say that $c \in \emp[w]{i_1,\ldots,i_l}$ if $c_i \in \emp[w]{i_1,\ldots,i_l}$ for all $i \in \Z$. Finally, for $N \in \N$  and $n \in \co{0}{N}$, let 
\begin{equation*}
\per[n,N]{i_1,\ldots,i_l} \defeq \{c \in (\haine5^{**})^{\Z} \colon \bina{\hs{\pi_{i_k}(c_j)}}=j+n \mod N, \text{ for all } j \in\Z, 1 \le k \le l\}
\end{equation*}
  be the set of all configurations such that 
\begin{equation*}  
  \pi_{i_k}(c)=.\dinf{(n\ldots(N-1)01\ldots(n-1))}, \text{ for } 1 \le k \le l,
\end{equation*}
where for all $w$, $.\dinf{w}$ denotes the configuration $c$ which satisfies that $c_{\co{j\length{w}}{(j+1)\length{w}}}=w$, for all $j \in \Z$.

\section{Cellular automata}\label{s:ca}
A (1D) \dfn{partial cellular automaton} (PCA) is a partial (\xpr{global}) continuous function $F:\az\pto\az$ whose domain is an SFT, and such that $F\sigma=\sigma F$.
Equivalently by some extension of the so-called Curtis-Lyndon-Hedlund theorem, there exist a \dfn{neighbourhood} $V\subfini\Z$ and a partial \dfn{local rule} $f:\A^{V}\pto\A$ such that for all $z\in\az$, $F(z)$ is defined if and only if $f(z\restr{i+V})$ is defined for all $i\in\Z$, in which case $F(z)_i \defeq f(z\restr{i+V})$. %=f(z\restr{i+\cc{-r}{r}})$; in particular, $F(z)$ is undefined if and only if $f(z\restr{i+\cc{-r}{r}})$ is undefined, for some $i \in \Z$.
If $V \subseteq \cc{-r}{r}$, then $r$ is  called a \dfn{radius} of the PCA. The radius of a PCA is not uniquely determined.%(even if it is not the minimal radius)

A PCA is called \dfn{reversible} (RPCA) if it is injective. In this case, it is known that there exists another RPCA, denoted by $F^{-1}$, such that $FF^{-1}$ and $F^{-1}F$ are restrictions of the identity, and $\dom(F^{-1})=F(\az)$  (the argument for this is similar to the one in \cite{hedlund}). In particular, there exist so-called inverse radius and inverse local rule.
If $r$ is both a radius and an inverse radius for an RPCA $F$, we call it a \dfn{bi-radius} for $F$.
In the rest of the paper, we only consider RPCA with bi-radius $1$. This is not a significant restriction, since these PCA and RPCA exhibit the whole range of computational and dynamical properties of general PCA and RPCA.

For $t\in\N$, the $t^{\text{th}}-$\dfn{order range} of $F$ is the (sofic) subshift $\Omega_F^t\defeq F^t(\az)\cap F^{-t}(\az)$ and its \dfn{limit set} is the (effective) subshift $\Omega_F\defeq\Omega_F^\infty\defeq\bigcap_{t\in\Z}\Omega_F^t$, containing all the configurations that are \emph{not} \dfn{ultimately rejected} (either in the past or the future).
There is a canonical way to associate a 2D SFT $\orb F$ to an RPCA $F$: it consists of the infinite space-time diagrams of the configurations that are not ultimately rejected. Formally, $\orb F\defeq\sett{\orb[x]F}{x\in\Omega_F}$, where $\orb[x]F\defeq|(F^t(x))_{t\in\Z}\in\am$ for any $x\in\Omega_F$.
One can see that $\orb F$ is conjugate to the $\Z^2$-action of $(F,\sigma)$ over $\Omega_F$.
Note nevertheless that the same SFT may correspond to distinct RPCA (if the RPCA have different transient phases, \ie they reject some configurations after different amounts of steps).

A pattern $w\in\A^D$, with $D\subset\Z^2$, is \dfn{locally valid} for $f$ if for any $(i,t)\in D$ such that $C\defeq(i+\cc{-1}1)\times\{t-1\}\subset D$, we have $p_{(i,t)}=f(p\restr C)$. Note that, in general, this notion depends on the local rule and not only on the RPCA. By compactness, if there exist locally valid square patterns of arbitrarily large height and width, then $\orb F\ne\emptyset$, \ie there are configurations which are never rejected. If $x\in F^{-t}(\az)$, then $|(x,F(x),\ldots,F^t(x))$ is a \dfn{locally valid horizontal strip} of height $t+1$.
The notion of a locally-valid horizontal strip depends only on the RPCA and not on the local rule, \ie it is a ''global`` notion.

For every $m\in \N$, $\vec{\delta}=(\delta_0,\ldots,\delta_{m-1})\in \{-1,0,1\}^m$, %$\vec{k} \in \N^m$,
we define the shift product $\sigma^{\vec{\delta}}
= \sigma^{\delta_0} \times \ldots \times \sigma^{\delta_{m-1}}$. %see whether we really need the Bk
A \dfn{partial partition} (cellular) \dfn{automaton} (PPA) is a PCA $F=\sigma^{\vec\delta}\circ\alpha$ over some alphabet $\A=\A_0\times\ldots\times \A_{m-1}$, where $\alpha$ is (the parallel synchronous application of) a partial permutation of $\A$. $-\delta_i$ is called the \dfn{direction} of field $i$. The (counter-intuitive) \xpr{$-$} is due to the fact that the normal definition of $\sigma$ shifts everything to the left, while we are used to thinking of the positive direction as going to the right. So, if we want to have a field with speed $+1$, then we should apply $\sigma^{-1}$ to it.

Every PPA is a RPCA with bi-radius $1$ %(see for instance \cite{poupetjac08} for a translation into CA),
 and conversely every RPCA is essentially a PPA (see for instance \cite[Proposition~53]{jarkkoppa}).
Note, however, that the inverse of a PPA is not, formally, exactly a PPA: the permutation is performed after the shifts, in the form $\alpha^{-1}\circ\sigma^{-\vec{\delta}}$. Nevertheless, it is conjugate, via $\alpha$, to the corresponding PPA.

In order to define families of PPA that are somehow uniform, we consider the corresponding objects acting on infinite alphabets.
A \dfn{partial partition automaton with infinite alphabet} (IPPA) is a partial map $F:(\haine5^{*m})^\Z\pto(\haine5^{*m})^\Z$, where $m\in\N$, $F=(\sigma^{\delta_0}\times\ldots\times\sigma^{\delta_{m-1}})\circ\alpha$, the $\sigma^{\delta_j}$ are shifts over infinite $\haine5^*$ (that is $\sigma(y)_i=y_{i+1}$ for any $y\in(\haine5^*)^\Z$ and $i\in\Z$), and $\alpha \colon \haine5^{*m} \to \haine5^{*m}$ is a partial (infinite) \dfn{permutation}.
By restricting the domain and the co-domain of an IPPA to finite subsets of $\haine5^{*m}$, we obtain normal (finite) PPA. In our constructions, the permutation $\alpha$ will always be length-preserving and the restriction will be taken over an alphabet of the form $\haine5^{\vec{k}}$.

If $F \colon \A^{\Z} \to \A^{\Z}$ and $G \colon \B^{\Z} \to \B^{\Z}$ are PCA, then we say that $G$ is a \dfn{factor} of $F$ if there exists a continuous map $H \colon \A^{\Z} \to \B^{\Z}$ such that $GH=HF$. If $F$ and $G$ are RPCA and $F$ factors onto $G$, then it is easy to see that $\orb{F}$ factors onto $\orb{G}$ through the map that sends $\orb{x}$ to $\orb{H(x)}$, for all $x \in \A^{\Z}$. However, the notion of factoring for RPCA is stronger, since it also takes into account the transient times of the RPCA, \ie the number of steps for which the image of an ultimately rejected configuration is defined before it is rejected (which are not relevant in the corresponding 2D SFTs).

Let $F_0,\ldots,F_{n-1}$ be RPCA such that $\dom(F_i) \cap \dom(F_j) = \emptyset$, for all $i \neq j$. Then, $\bigsqcup_{i\in \co{0}{n}}F_i$ denotes the map with domain $\bigsqcup_{i \in \co{0}{n}}\dom(F_i)$ and that agrees with $F_i$ on $\dom(F_i)$, for all $i \in \co{0}{n}$. $\bigsqcup_{i\in \co{0}{n}}F_i$ is not always an RPCA, since there might be a configuration that is not in any $\dom(F_i)$ but that is locally everywhere in the domains (which are SFTs). However, and this will always be the case in this paper,  $\bigsqcup_{i\in \co{0}{n}}F_i$ is also an RPCA if $\dom(F_i)$ and $\dom(F_j)$ are over disjoint alphabets, for $i \neq j$. In this case, $\Omega_{\bigsqcup_{i\in \co{0}{n}}F_i}=\bigsqcup_{i \in \co{0}{n}} \Omega_{F_i}$ and $\orb{\bigsqcup_{i\in \co{0}{n}}F_i}=\bigsqcup_{i \in \co{0}{n}}\orb{F_i}$.

\subsection{Expansiveness}
The projective line $\Rb\defeq\R\sqcup\{\infty\}$ %endowed with the one-point compactification of the Euclidean topology of $\R$.
%With this topology $\Rb$ is topologically isomorphic to the unit circle endowed with the restriction of the euclidean topology, hence it is separable, metrizable and has a countable base.
is seen as the set of \dfn{slopes} to the \emph{vertical} direction.
Here, quite unconventionally, the horizontal direction is represented by $\infty$ and the vertical one by $0$.
The relevance of this choice will appear later, but in any case it does not affect any set-theoretical, topological or computable property because the inversion map over $\Rb$ is a computable homeomorphism. %These properties would even be the same by representing $\Rb$ as the quotient of the circle by the central symmetry, but this would give expressions a little more complicated.

The projective line $\Rb$ admits a natural effective topology if seen as the quotient of the circle by central symmetry: a subset is effectively closed if the corresponding subset of the circle is effectively closed as a subset of $[0,1]^2$. This topology is equivalent to the one-point compactification of the $\R$ and  renders $\Rb$ a compact, metric space.

%\begin{definition}
Let $X$ be a 2D subshift, $l\in\Rb$ a slope and $\lin{l}\subset\R^{2}$ the corresponding vectorial line.
We say that direction $l$ %(or line $\lin l$)
is \dfn{expansive} for $X$ if there exists a bounded shape $V\subset\R^2$ %orthogonal to $\lin l$
such that, for all $x,y \in X$,
\[x\restr{(\lin{l}+V)\cap \Z^2}=y\restr{(\lin{l}+V) \cap \Z^2}\impl x=y~.\]
%If $S$ is orthogonal to $l$, then its length $\length{S}$ is sometimes called a \dfn{radius} of expansiveness.
We denote by $\NE(X)$ the {set of \dfn{non-expansive} directions} (\ie the set of directions that are not expansive). %\defeq\sett{l}{l\text{ is expansive for }X} It scomplement $\E(X)$ is the \dfn{set of expansive directions}.%E(X)$ the set of expansive directions of $X$ and by $\NE(X)$ its complement, which is called the set of nonexpansive directions.
%\end{definition}
The terminology comes from the fact that if $l=p/q$ is rational (or infinite), then $l$ is expansive for $X$ if and only if the dynamical system $(X,\sigma^{(p,q)})$ is expansive, in the classical sense of expansive dynamical systems.%, see \cite{kurka}.

Expansive directions were first introduced by Boyle and Lind \cite{expsubd} in a more general setting.
The following fact is a particular case of \cite[Theorem~3.7]{expsubd}.
\begin{proposition}\label{p:atleastone}
Let $X$ be a 2D subshift.
Then, $\NE(X)$ is closed. In addition, $\NE(X)$ is empty if and only if $X$ is finite.
\end{proposition}

We say that $X$ is \dfn{extremely expansive} if $\card{\NE(X)}=1$, which is, according to Proposition~\ref{p:atleastone}, the most constrained non-trivial case.

In the case of SFTs (actually, of all effective subshifts), we have an additional restriction on the set of non-expansive directions that comes from computation theory, as is usually the case, see \cite{projsft, entrsft}.

A direction $l \in\Rb$ can be represented as the pair of coordinates of the intersection of the line $\lin{l}$ with the unit circle. This gives two (symmetric with respect to the origin) representations for each direction which are computably equivalent. Computability questions about expansive directions can then be transferred to computability questions about pairs of real numbers, which we already know how to deal with.

It can be noted that effectively closed subsets that do not contain $\{\infty\}$ are exactly the effectively closed subsets of $\R$.
The restriction map from $\Rb$ (with the above-defined effective topology) onto $\R$ is actually computable, and it can be noted that the pre-image of an effectively closed set by a computable function is effectively closed.

\begin{lemma}\label{lem:nonexpsftrestr}
Let $X$ be a 2D SFT. Then, $\NE(X)$ is effectively closed.
\end{lemma}
In particular, if an SFT $X$ has a unique direction of non-expansiveness, then this direction must be computable.
\begin{proof}

The statement follows from the following two facts: First, it is semi-decidable whether a direction is expansive, \ie there exists a TM that takes as input a (rational direction) and halts if the direction is expansive. This follows from \cite[Lemma 3.2]{expsubd}. Secondly, it is semi-decidable whether two \emph{expansive} directions belong in the same expansive component. (The expansive component of an expansive direction is the largest connected set that includes the direction and is included in the set of expansive directions. One can see that it is always an open interval.) This follows from \cite{nasu154}, as described in \cite[Appendix C]{opsd}.

Having these two facts in mind, it is not difficult to see that the following algorithm enumerates a sequence of intervals whose union is the complement of $\NE(X)$: For each rational direction, check whether it is expansive. Every time you find an expansive direction, check whether it is in the same component with one of the expansive directions that you have already found. Every time this is the case, output the whole interval of directions that is between them.

\end{proof}

A subshift $Y$ is called \dfn{extremely-expansively sofic} if there exists an extremely expansive SFT that factors onto $Y$. Since expansive directions are not preserved through block maps, an extremely-expansively sofic subshift need not be extremely expansive itself. In fact, as we will see, there exist extremely-expansively sofic subshifts that do not have any direction of expansiveness.

\begin{lemma}\label{lem:basicstuffaboutNE}
Let $X_0,X_1,\ldots$ be 2D subshifts over the same alphabet $\A$.
\begin{itemize}
%\item$\NE(X_0\times X_1)=\NE(X_0)\cup\NE(X_1)$.
\item If $X_0\subseteq X_1$, then $\NE(X_0)\subseteq\NE(X_1)$.
\item If $\NE(X_0)\cap\NE(X_1)=\emptyset$, then $X_0\cap X_1$ is a finite subshift.
\item If $\bigsqcup_wX_w$ is a closed disjoint (possibly uncountable) union, then $\NE(\bigsqcup_wX_w)=\bigcup_w\NE(X_w)$.
\end{itemize}\end{lemma}
%example for nondisjoint union, but not RPCA: two signals that, when crossing color the central cell, in 2 different colors.
\begin{proof}
The first claim follows immediately from the definitions.

For the proof of the second claim, we have that $\NE(X_0 \cap X_1) \subseteq \NE(X_0) \cap \NE(X_1) = \emptyset$ according to the first claim. Therefore, $\NE(X_0 \cap X_1) = \emptyset$, and since $X_0 \cap X_1$ is a subshift, Proposition~\ref{p:atleastone} gives that it is finite.

Finally, for the last claim, the inclusion $\NE(X_w) \subseteq \NE(\bigsqcup_wX_w)$ comes from the first point.

For the other inclusion, assume $l\in\NE(\bigsqcup_wX_w)$ %and let $\lin l\subset\R^2$ be the line with slope $l$ that goes through the origin. 
Then, there exist $x,y\in\bigsqcup_wX_w$ which coincide over an open half-plane $H_l \subseteq \R^2$ of slope $l$ and disagree somewhere outside it. The orbits of $x$ and $y$ under the shift action have a common limit point $z$. Then, $z$ is in the intersection $X_x \cap X_y$ of the subshifts that contain $x$ and $y$, respectively. By disjointness, we get that $X_x=X_y=X_{w'}$, for some $w'$, which means that $l \in \NE(X_{w'})$.

%Since the union $\bigsqcup_wX_w$ is closed, we have that $z \in \bigsqcup_wX_w$, so $z \in X_{w'}$, for some $w'$.
%Since the union is also disjoint and all $X_w$ are closed, this implies that $x,y \in X_{w'}$, too, which implies that $l \in \NE(X_{w'})$. Therefore, $\NE(\bigsqcup_wX_w) \subseteq \bigcup_w\NE(X_w)$.
\end{proof}

If $F$ is an RPCA, then we denote $\NE(F)\defeq\NE(\orb F)$.
It is straightforward that the horizontal direction (which according to our definition is $\infty$) is expansive for $F$. It is not much more complicated to see that, if the bi-radius is $1$, %all directions in $\oo{-1}1$ are.
%For every bi-radius $1$ RPCA $F$, we have that $(-1,1) \subseteq \E(F)$ so that
$\NE(F) \subseteq [-1,1]$ (directions around the horizontal are expansive). % and  $[\infty,-1) \cup (1,\infty] \subseteq \inv(E(F))$, or, in other words, $\inv(\NE(F)) \subseteq [-1,1]$. If $[-1,1]$ is equipped with the Euclidean distance, then by retraction through $\inv$, we obtain a distance for $[\infty,-1] \cup [1,\infty]$. For two directions $l,l' \in \Rb$, we define $d(l,l') \defeq |\inv(l')-\inv(l)|$.

Conversely, it can be shown that, up to a recoding, every 2D SFT for which the horizontal direction is expansive is equal to $\orb F$, for some RPCA $F$.
\chapter{Simulation}\label{s:simul}

\section{Simulation}

If $S,T\in\Ns$ and $Q\in\Z$, we say that RPCA $F:\az\pto\az$ \dfn{$(S,T,Q)$-simulates} RPCA $G:\bz\pto\bz$ if there is a partial continuous \dfn{decoding} surjection $\Phi:\az\pto\bz$ such that $\sigma\Phi=\Phi\sigma^S$, $G\Phi=\Phi\sigma^QF^T$, $G^{-1}\Phi=\Phi\sigma^{-Q}F^{-T}$ and the \dfn{simulating} subshift $\rock\Phi\defeq\bigsqcup_{\begin{subarray}c0\le t<T\\0\le s<S\end{subarray}}\sigma^sF^t(\dom(\Phi))$ is a disjoint union.% (some bi-dimensional \emph{Rokhlin tower}).
In other words, $1$ step of $G$ is encoded into $T$ steps of $F$, up to some shift by $Q$, and the intermediary steps used are not valid encodings.
We note $F\simu[S,T,Q,\Phi]G$, or when some parameters are clear from the context or not so important, $F\simu[S,T,\Phi]G$, $F\simu[S,T,Q]$, $F\simu[S,T]G$, or $F\simu G$ (each time this symbol will be used, $F$ and $G$ are meant to be RPCA).

We remind the reader that according to our notations, $\sigma\Phi=\Phi\sigma^S$ and $G\Phi=\Phi\sigma^QF^T$ and $G^{-1}\Phi=\Phi\sigma^{-Q}F^{-T}$ imply that the domains of the two partial functions are identical. This is in fact crucial for understanding the notion of simulation and it will be used extensively in the proofs and constructions to come. For example, this means that the equality $G\Phi=\Phi\sigma^QF^T$ does not immediately imply $G^{-1}\Phi=\Phi\sigma^{-Q}F^{-T}$, because the domains of $G^{-1}\Phi$ and $\Phi\sigma^{-Q}F^{-T}$ might be different (if we only had the equality $G\Phi=\Phi\sigma^QF^T$, it could happen that $x \in \dom(G^{-1}\Phi)$ but $x \notin F^T\sigma^Q(\A^\Z)$).
%$G^{-1}\Phi(x) \notin \Phi(\A^{\Z})$).

In fact, one can see that the couple of conditions $G\Phi=\Phi\sigma^QF^T$ and $G^{-1}\Phi=\Phi\sigma^{-Q}F^{-T}$ is equivalent to the triple of conditions $G\Phi=\Phi\sigma^QF^T$, $\dom(G\Phi)=\dom(\Phi\sigma^{-Q}F^{-T})$ and $\dom(G^{-1}\Phi) = \dom(\Phi F^{T} \sigma^Q)$.

$F$ \dfn{exactly simulates} $G$ if $\Phi$ is actually bijective.
In other words, there exists a well-defined \dfn{encoding} function $\Phi^{-1}:\bz\to\dom(\Phi)$.
$F$ \dfn{completely simulates} $G$ if, besides, $\Omega_F\subset\rock\Phi$.
In other words, every bi-infinite orbit of $F$ will eventually encode some orbit of $G$.
Actually, in our constructions we will even have the stronger $\dom(F^{t'})\subset\rock\Phi$, for some $t' \in \Z$.

%\clearpage

\begin{remark}\label{r:simsynchr}~\begin{enumerate}
	\item $%F^T\sigma^{Q}(\dom(\Phi))\subseteq
	\dom(\Phi)=\sigma^S(\dom(\Phi))$.
	\item $F\simu[S,T,DS]G$ if and only if $F\simu[S,T,0]\sigma^{D}G$.
	\item For any $s\in\co0S,t\in\co0T$, $\bulk[S]{\sigma^sF^t(\dom(\Phi))}$ is an SFT.
	\item Since the union $\rock\Phi$ is disjoint, there exists a shape $U\subfini\Z$ such that for any $x\in\rock\Phi$ %with $x\restr U=y\restr U$
	, $x\restr U$ determines the (unique) $s\in\co0S$ and $t\in\co0T$ such that $x\in\sigma^sF^t(\dom(\Phi))$.
\end{enumerate}\end{remark}

\begin{proof}
The first two claims follow immediately from the definitions.

For the third claim, notice that since $\Phi$ is continuous and $\sigma \Phi = \Phi \sigma^S$, this means that $\dom(F)$ is the domain of a PCA over $\bulk[S]{\A^{\Z}}$, so it is an SFT. Since $\sigma$ and $F$ are invertible maps and the property of being an SFT is preserved under invertible maps, we have that $\bulk[S]{\sigma^sF^t(\dom(\Phi))}$ is an SFT for all $s \in \co{0}{S}$ and $t \in \co{0}{t}$.

The last claim follows easily from the disjointness using a classical compactness argument.
\end{proof}

We can prove an analogue of Curtis-Lyndon-Hedlund theorem for decoding and encoding functions.
\begin{remark}\label{r:simhedlund}
	The decoding function $\Phi$ admits a {neighbourhood} $V\subfini\Z$ and a partial \dfn{bulked local rule} $\phi:\A^{V}\pto\B$ such that for all $x\in\az$, $\Phi(x)$ is defined if and only if $\phi(x\restr{iS+V})$ is defined for any $i\in\Z$, in which case the latter is equal to $\Phi(x)_i$.
	\\
	If the simulation is exact, the encoding function $\Phi^{-1}$ admits a {neighbourhood} $V\subfini\Z$ and a partial \dfn{unbulked local rule}, abusively noted $\phi^{-1}:\B^{V}\pto\A^S$ such that for all $y\in\B^\Z$, $\Phi^{-1}(y)$ is defined if and only if $\phi^{-1}(y\restr{i+V})$ is defined for any $i\in\Z$, in which case the latter is equal to $\Phi^{-1}(y)_{\co{iS}{(i+1)S}}$.
\end{remark}

Exact complete vertical (\ie $Q=0$) simulation is stronger than most notions found in the literature. In particular:
\begin{itemize}
	\item $\orb F$ simulates $\orb G$ in the sense of \cite{drs}.
    \item The $\Z^2$-action $(F,\sigma)$ over the limit set $\Omega_F$ (or the 2D SFT $\orb F$) is conjugate to a suspension of $\Omega_G$ in the sense of a homeomorphism \[\appl[\Psi]{\Omega_F}{\Omega_G\times\co0S\times\co0T}x{(\Phi F^{-t}\sigma^{-s}(x),s,t),~\text{where}~F^{-t}\sigma^{-s}(x) \in \dom(\Phi)}\]
        %such that $\Psi F(x)=(G^q(y),r,s)$ if $\Psi(x)=(y,qT+r-1,s)$ and $\Psi\sigma(x)=(\sigma^q(y),t,r)$ if $\Psi(x)=(y,t,qS+r-1)$ (see \cite{nexpdir});
    \item The $\Z^2$-action $(G,\sigma)$ over the limit set $\Omega_G$ (or the 2D SFT $\orb G$) is conjugate to the $\Z^2$-action $(F^T,\sigma^S)$ restricted to $\dom(\Phi) \cap \Omega_F$ (see \cite{gacs}); %\Pi_\age^{-1}(\{\dinf0\})$ and $\Pi_\addr^{-1}(\{\dinf(0\ldots S-1)}\})$;
    \item $G$ is a sub-automaton of a rescaling of $F$, so that $F$ simulates $G$ according to the definition of simulation given in \cite{ollingersimulation}. While it is not necessary to formally define  this notion of simulation, we can intuitively say that rescaling corresponds to the role of parameters $S$ and $T$ in our definition, while the sub-automaton condition corresponds to the decoding function $\Phi$. We notice, however, that Ollinger's definition is more general than ours, since it does  not require  $\rock\Phi$ to be a disjoint union, while the simulated can also be rescaled.
\end{itemize}

But the definition above also involves the transient part: every locally valid horizontal strip of height $t+1$ for $G$ gives a locally valid horizontal strip of height $Tt+1$ for $F$.

The following facts about our notion of simulation follow directly from the definition:
\begin{itemize}
\item Each kind of simulation is a conjugacy invariant.

\item If $F$ simulates $G$\resp{exactly}, then it simulates\resp{exactly} any of its subsystems (but clearly, completeness is not preserved). If $F$ factors onto $G$, then $F \simu[1,1,0] G$ completely. 

\item $F\times G\simu[1,1,0]F$ completely if $G$ does not have empty domain. Also $F\times G\simu[1,1,0]F$ exactly if $G$ includes a singleton subsystem (recall that $G$ is a PCA, so that it does not necessarily have periodic points). The simulation is simultaneously exact and complete if $G$ is a singleton system.

\item The surjectivity of $\Phi$ implies that only systems with empty domain can be simulated by systems with empty domain.
\end{itemize}
We will mainly focus on \dfn{non-trivial} simulations: this means that $S,T>1$ and $G$ does not have empty domain.

\begin{remark}\label{r:locvalid}
If $F\simu[S,T,Q,\Phi]G$ non-trivially, then for all $j \in \cc{0}{T}$,
%\begin{equation*}
%\Omega_F^{\ipart{T/2}}\supseteq 
$F^{j}(\dom(G\Phi))=\sigma^{-Q}F^{-(T-j)}(\dom(G^{-1}\Phi) \neq \emptyset$.
%\end{equation*}
 %, $k\in\Z$, $t$ in $\cc0{kT}$ or $\cc{kT}0$, then $F^{kT-t}$ is total over $F^t\Phi^{-1}(\dom(G^k))$.
\end{remark}

More specifically, a configuration \xpr{in the middle} of the work period, \ie when $j = \ipart{T/2}$ has at least $\ipart{T/2}$ forward and backward images, or, in other words, it belongs to $\Omega_F^{\ipart{T/2}}$.

The following lemma states that the limit sets correspond, in the case of complete simulation. It is a more mathematical and detailed version of the comment that we made earlier, that a valid strip horizontal of height $t+1$ in $G$ gives a valid horizontal strip of height $Tt+1$ in $F$ (provided that the strip is simulated).

\begin{lemma}\label{l:simlim}
Assume $F\simu[S,T,Q,\Phi]G$.
\begin{enumerate}
\item\label{i:simo} If $j\in\N\sqcup\{\infty\}$, then
\[\rock[j]\Phi\defeq\bigsqcup_{\begin{subarray}c0\le t<T\\0\le s<S\end{subarray}}\sigma^sF^t\Phi^{-1}(\Omega_G^j)
\] is a disjoint union and a subshift, included in $\Omega_F^{(j-1)T+1}$. In addition, $\rock[j]\Phi\supset\rock[j+1]\Phi$ and $\rock[\infty]\Phi=\bigcap_{j\in\N}\rock[j]\Phi$.
\item $\Omega_F\supset\rock[\infty]\Phi$.%\defeq\bigsqcup_{\begin{subarray}c0\le t<T\\0\le s<S\end{subarray}}\sigma^sF^t\Phi^{-1}(\Omega_G)$.
\item\label{i:limlim} If the simulation is complete, then $\Omega_F=\rock[\infty]\Phi$.
\end{enumerate}\end{lemma}
\begin{proof}~\begin{enumerate}
\item It is clear that $\rock[j]\Phi$ is a disjoint union and a subshift, each subset in the union being (syntactically) included in one in the expression of $\rock\Phi$.
Assume that $F\simu[S,T,Q,\Phi]G$ for some $Q\in\Z$.
Now,
\begin{eqnarray*}
\Phi^{-1}(\Omega_G^j)&=&
\dom(G^j\Phi)\cap\dom(G^{-j}\Phi)
\\&=&
\dom(\Phi\sigma^{jQ}F^{jT})\cap\dom(\Phi\sigma^{-jQ}F^{-jT})
\\&\subset&
\dom(F^{jT})\cap\dom(F^{-jT})=\Omega_F^{jT}~.
\end{eqnarray*}
Hence, for any $s\in\co0S$ and any $t\in\co0T$, $\sigma^sF^t\Phi^{-1}(\Omega_G^j)\subset\Omega_F^{jT-T+1}$. 
The other claims follow from the definitions.
\item It is obvious from the previous point that $\bigcap_{j\in\N}\Omega_F^{jT}\supset\bigcap_{j\in\N}\rock[j]\Phi=\rock[\infty]\Phi$.
\item Conversely, assume $x\in\Omega_F$, so that clearly $\forall k\in\Z,F^k(x)\in\Omega_F$. By completeness, there exist $y\in\dom(\Phi)$ and $s\in\co0S,t\in\co0T$ such that $\sigma^sF^t(y)=x$. Disjointness and a direct induction give that for all $k\in\Z$, $F^k(y)\in F^{k\bmod T}\sigma^{Q\lfloor k / T \rfloor}(\dom(\Phi))$. In particular, for all $j\in\Z$, $G^j\Phi(y)=\Phi F^{jT}\sigma^{jQ}(y)$ is defined. This gives that $\Phi(y)\in\Omega_G$, so $x \in \sigma^{-s}F^{-t}\Phi^{-1}(\Omega_G)= \sigma^{S-s}F^{T-t}\Phi^{-1}(\Omega_G)$.
\end{enumerate}\end{proof}

The following remark links the periodic points of the simulating and simulated systems. It is essential for proving aperiodicity of the subshifts that we construct. The same result appears in \cite{drs,twobytwo}, even though the argument essentially goes back to the kite-and-dart tile set of Penrose. We give a slightly more general version of the usual result also takes into consideration the shift by $Q$.

\begin{remark}\label{r:penrose}
 If $F\simu[S,T,Q]G$ completely, then $\orb F$ admits a configuration with period $(s-lQ,t)$ if and only if $\orb G$ admits a configuration with period $(k,l)$, where $s=kS$ and $t=lT$.
\end{remark}

We will only use the case $Q=0$, for which it is intuitively clear to see that it holds true. When $q \neq 0$, one has to have in mind that for every $T$ time steps of a configuration of $F$, the simulated configuration is shifted $Q$ steps to the left.

\section{Nested simulations}
In the sequel, we will be most interested in infinite sequences of simulations of the form: $F_0 \simu F_1 \simu F_2 \simu \ldots$. This looks like a formidable task, since every RPCA of the sequence must contain the information about an infinite number of configurations and update this information within a determined time, but, as the results of this section will imply, an infinite sequence of simulations gives RPCA with very useful properties. The construction of these sequences forms the basic part of our constructions and will be done in the following chapters.

If $\vec{S}=(S_i)_{ 0 \le i \le n-1}$ is a sequence of numbers, then $\vec{1S}$ is the sequence whose first element is equal to $1$ with the elements of $\vec{S}$ shifted by one after it. 
If $\vec{S}=(S_i)_{ 0 \le i \le n-1}$ and $\vec{T}=(T_i)_{ 0 \le i \le n-1}$ are finite sequences of non-zero numbers, then $\vec{1S/T}$ is the sequence $(S_{i-1}/T_i)_{0 \le i \le n-1}$, where $S_{-1}\defeq 1$. 
A short calculation shows that 
\begin{equation*}
\anib[\vec{1S/T}]{\vec{Q}}\prod T_i = \sum_{0 \le i \le n-1} \left(Q_i\prod_{0 <j<i} S_j \prod_{i < j \le n-1} T_j\right).
\end{equation*}
%the sequence $\prod_{0 <j<i} S_j Q_i \prod_{i < j \le n-1} T_j  )_{0 \le i \le n}$.

\begin{lemma}\label{l:simtrans}
	Simulation\resp{exact, complete, exact and complete} is a preorder.
\\	More precisely, if $F_0\simu[S_0,T_0,Q_0,\Phi_0]F_1\simu[S_1,T_1,Q_1,\Phi_1]\ldots\simu[S_{n-1},T_{n-1},Q_{n-1},\Phi_{n-1}]F_n$\resp{exactly, completely} for some $n\in\N$, then $F_0\simu[S,T,Q,\Phi]F_n$\resp{exactly, completely}, where 
\begin{equation*}
(S,T,Q,\Phi) = (\prod S_i,\prod T_i,{\anib[\vec{1S/T}]{\vec{Q}}\prod T_i},\Phi_{n-1}\cdots\Phi_0).
\end{equation*}
\end{lemma}

The products range from $0$ to $n-1$. If there were no shifts in the simulation (\ie if $Q_i=0$ for all $i$) the above statement would be more or less trivial. Even in the presence of shifts, the proof is essentially a simple verification.

\begin{proof}~\begin{itemize}
\item	Clearly $F\simu[1,1,0,\id]F$.
\item	Now suppose
    $F\simu[S,T,Q,\Phi]G\simu[S',T',Q',\Phi']H$.
    Then it is clear that $\sigma\Phi'\Phi=\Phi'\sigma^{S'}\Phi=\Phi'\Phi\sigma^{S'S}$ and $H\Phi'\Phi=\Phi'\sigma^{Q'}G^{T'}\Phi=\Phi'\Phi\sigma^{QT'+SQ'}F^{T'T}$.
    Moreover:
    \begin{eqnarray*}
    	\bigsqcup_{\begin{subarray}c0\le t<T\\0\le s<S\end{subarray}}\sigma^sF^t(\dom(\Phi))
    	&\supset&%\bigsqcup_{\begin{subarray}c0\le t<T\\0\le s<S\end{subarray}}\sigma^sF^t\Phi^{-1}(\rock{\Phi'})\\    	&=&
    	\bigsqcup_{\begin{subarray}c0\le t<T\\0\le s<S\end{subarray}}\sigma^sF^t\Phi^{-1}(\bigsqcup_{\begin{subarray}c0\le t'<T'\\0\le s'<S'\end{subarray}}\sigma^{s'}G^{t'}(\dom(\Phi')))\\
    	&=&\bigsqcup_{\begin{subarray}c0\le t<T\\0\le s<S\end{subarray}}\bigsqcup_{\begin{subarray}c0\le t'<T'\\0\le s'<S'\end{subarray}}\sigma^sF^t\Phi^{-1}(\sigma^{s'}G^{t'}(\dom(\Phi')))\\
		&=&\bigsqcup_{\begin{subarray}c0\le t<T\\0\le s<S\end{subarray}}\bigsqcup_{\begin{subarray}c0\le t'<T'\\0\le s'<S'\end{subarray}}F^{t+t'T}\sigma^{s+s'S+t'Q}(\dom(\Phi'\Phi))\\
		%&=&\bigsqcup_{0\le t<TT'}F^{t+t'T}(\bigsqcup_{0\le s<SS'}\sigma^{s+s'S+t'DT}(\dom(\Phi'\Phi)))\\
		%&=&\bigsqcup_{0\le t<TT'}F^{t+t'T}(\bigsqcup_{0\le s<SS'}\sigma^{s+s'S+t'DT\bmod SS'}(\dom(\Phi'\Phi)))\text{ (since }\sigma^{SS'}(\dom(\Phi'\Phi))=\dom(\Phi'\Phi))\\
		%&=&\bigsqcup_{0\le t<TT'}F^{t+t'T}(\bigsqcup_{0\le s<SS'}\sigma^{s+s'S}(\dom(\Phi'\Phi)))\\
%		&=&\bigsqcup_{\begin{subarray}c0\le t<T\\0\le s<S\end{subarray}}\bigsqcup_{\begin{subarray}c0\le t'<T'\\0\le s'<S'\end{subarray}}F^{t+t'T}\sigma^{s+s'S+t'D}(\dom(\Phi'\Phi))\\
%		&=&\bigsqcup_{\begin{subarray}c0\le t<T\\0\le t'<T'\end{subarray}}F^{t+t'T}(\bigsqcup_{\begin{subarray}c0\le s<S\\0\le s'<S'\end{subarray}}\sigma^{s+s'S+t'D\bmod SS'}(\dom(\Phi'\Phi)))\text{ (since }\sigma^{SS'}(\dom(\Phi'\Phi))=\dom(\Phi'\Phi))\\
%		&=&\bigsqcup_{\begin{subarray}c0\le t<T\\0\le t'<T'\end{subarray}}F^{t+t'T}(\bigsqcup_{\begin{subarray}c0\le s<S\\0\le s'<S'\end{subarray}}\sigma^{s+s'S}(\dom(\Phi'\Phi)))\\immediate
		&=&\bigsqcup_{\begin{subarray}c0\le t<TT'\\0\le s<SS'\end{subarray}}\sigma^sF^t(\dom(\Phi'\Phi))\eqdef\rock{\Phi'\Phi}~.
\end{eqnarray*}
This proves that $F\simu[SS',TT',QT'+SQ',\Phi'\Phi]H$.
\item If $\Phi$ and $\Phi'$ are bijections, then $\Phi'\Phi$ is also a bijection.
\item If both simulations are complete, then by Point \ref{i:limlim} of Lemma \ref{l:simlim},
\begin{eqnarray*}
\Omega_F&=&\bigsqcup_{\begin{subarray}c0\le t<T\\0\le s<S\end{subarray}}\sigma^sF^t\Phi^{-1}(\Omega_G)\\
&\subset&\bigsqcup_{\begin{subarray}c0\le t<T\\0\le s<S\end{subarray}}\sigma^sF^t\Phi^{-1}(\rock{\Phi'})\\
&=&\rock{\Phi'\Phi}.%\text{ as proven just above.}
\end{eqnarray*}
\item A direct induction gives the expected results.
%In particular, note that, for $n\in\N$,
%\[D_0+\frac{S_0}{T_0}\anib[(S_i/T_i)_{1\le i\le n}]{(D_i)_{1\le i\le n}}
%=\anib[\vec S/\vec T]{\vec D}
%~.\]

%    $F_0\simu[\prod_{i<n}S_i,\prod_{i<n}T_i,\Phi_{n}\ldots\Phi_0]F_{n}\simu[S_{n},T_{n},\Phi_{n}]F_{n+1}$.
%	It is clear that $\sigma\Phi_{n}\Phi_{n-1}\ldots\Phi_0=\Phi_n\sigma^{S_n}\Phi_{n-1}\ldots\Phi_0=\Phi_n\Phi_{n-1}\ldots\Phi_0\sigma^{\prod_{i\le n}S_i}$ and $F_n\Phi_n\Phi_{n-1}\ldots\Phi_0=\Phi_nF_{n-1}^{T_n}\Phi_{n-1}\ldots\Phi_0=\Phi_n\Phi_{n-1}\ldots\Phi_0F_0^{\prod_{i\le n}T_i}$.
\end{itemize}\end{proof}

Similarly to simulations, which involve a decomposition of the system in terms of how much is shifted the grid on which to read the encoding, a sequence of simulations involves a nested decomposition, which gives a full skeleton, inside each configuration, as expressed by the following lemma. Here, and in the following, we use gothic letters to denote sequences, but the corresponding normal letters to denote the elements of the sequences. Also, if $\seq S$ is an infinite sequence and $n \in \N$, then $\seq{S}_{\co{0}{n}}$ is the finite prefix of length $n$ of $\seq{S}$. Finally, if $(\Phi_i)_{i \in \N}$ is a sequence of decoding functions, then $\Phi_{\co{0}{n}}$ will be the decoding function $\Phi_{n-1} \cdots \Phi_0$. %Notice that we treat decoding functions different than sequences even though the notation is similar. Hopefully, this does not cause any confusion.

\begin{lemma}\label{l:nonvide}~\begin{enumerate}
\item\label{i:infsim} If $F_0\simu[S_0,T_0,\Phi_0]F_1\simu[S_1,T_1,\Phi_1]\ldots\simu[S_{n-1},T_{n-1},\Phi_{n-1}]F_n\simu[S_n,T_n,\Phi_n]\ldots$ and $j\in\N\sqcup\{\infty\}$, then
\begin{eqnarray*}
\rock[j]{\seq\Phi} &\defeq &\bigcap_{n\in\N}\rock[j]{\Phi_{\co0n}}\\ & = &\bigsqcup_{\begin{subarray}c\seq t\in\prod_{i\in\N}\co0{T_i}\\\seq s\in\prod_{i\in\N}\co0{S_i}\end{subarray}}\bigcap_{n\in\N}\sigma^{\anib[\seq S]{\seq{s}_{\co0n}}}F_0^{\anib[\seq T]{\seq{t}_{\co0n}}}\Phi_0^{-1}\cdots\Phi_{n-1}^{-1}(\Omega_{F_n}^j)
\end{eqnarray*}
is a disjoint union and a subshift.
\\
In addition, $\rock[j]{\seq\Phi}\supset\rock[j+1]{\seq\Phi}$ and $\rock[\infty]{\seq\Phi}=\bigcap_{j\in\N}\rock[j]{\seq\Phi}$.
\item\label{i:nonempty} If, besides, all simulations are nontrivial, then $\rock[2]{\seq\Phi}=\rock[\infty]{\seq\Phi}\subset\Omega_{F_0}$ is uncountable.

\item If the simulations (in the hypothesis of Point \ref{i:infsim}) are complete, then $\rock[\infty]{\seq\Phi}=\Omega_{F_0}$.

\item\label{i:skel} If the sequence $(\Phi_n)_{n \in \N}$ is computable, then the map $x \in \rock[\infty]\Phi \to (s_i,t_i)_{i \in \N}$, where $(s_i,t_i)_{i \in \N}$ is the (unique) sequence such that $x\in\bigcap_n\sigma^{\anib[\seq S]{\seq{s}_{\co0n}}}F_0^{\anib[\seq T]{\seq{t}_{\co0n}}}\dom(\Phi_{\co{0}{n}})$, is computable.
\end{enumerate}\end{lemma}
Point \ref{i:nonempty} implies nonemptiness of $\Omega_{F_0}$ and $\orb{F_0}$, and of any $\Omega_{F_n}$, since all those statements can be applied to the sequence starting from $n$.
Point \ref{i:skel} states that we can always recover the skeleton from a valid configuration. In particular the skeleton map is continuous.
\begin{proof}~\begin{enumerate}
\item By Lemma \ref{l:simtrans} and compactness, it is clear that $\rock[j]{\seq\Phi}$ is a subshift. The equality is rather easily checkable. We can see that the union is disjoint: if $(\seq s,\seq t)\ne(\seq{s'},\seq{t'})$, say $(s_m,t_m\ne s'_m,t'_m)$, then $\bigcap_{n\in\N}\sigma^{\anib[\seq S]{\seq{s}_{\co0n}}}F_0^{\anib[\seq T]{\seq{t}_{\co0n}}}\Phi_0^{-1}\cdots\Phi_{n-1}^{-1}(\Omega_{F_n}^j)$
is included in $\sigma^{\anib[\seq S]{\seq{s}_{\co0m}}}F_0^{\anib[\seq T]{\seq{t}_{\co0m}}}\Phi_0^{-1}\cdots\Phi_{m-1}^{-1}(\Omega_{F_m}^j)$, which is, according to Lemma~\ref{l:simtrans}, disjoint from
$\sigma^{\anib[\seq S]{\seq{s}'_{\co0m}}}F_0^{\anib[\seq T]{\seq{t}'_{\co0m}}}\Phi_0^{-1}\cdots\Phi_{m-1}^{-1}(\Omega_{F_m}^j)$ which includes \\$\bigcap_{n\in\N}\sigma^{\anib[\seq S]{\seq{s}'_{\co0n}}}F_0^{\anib[\seq T]{\seq{t}'_{\co0n}}}\Phi_0^{-1}\cdots\Phi_{n-1}^{-1}(\Omega_{F_n}^j)$
~.

\item
Since for any $n\in\N$ and $m\ge n$, $F_n\simu[S_n\cdots S_{m-1},T_n\cdots T_{m-1},\Phi_{m-1}\cdots\Phi_n]F_{m}$, then Point \ref{i:simo} of Lemma \ref{l:simlim} says that $\Omega_{F_n^{T_n\cdots T_{m-1}+1}}\supset\rock[2]{\Phi_{\co nm}}~.$
If the simulations are nontrivial, then $T_n\cdots T_{m-1}\to\infty$ when $n$ is fixed and $m \to \infty$, and
$\Omega_{F_n}\supset\bigcap_{m\in\N}\Omega_{F_n^{T_0\cdots T_{m-1}+1}}\supset\bigcap_{m\in\N}\rock[2]{\Phi_{\co nm}}=\rock[2]{\Phi_{\co n{\infty}}}~.$
Injecting this inclusion in the definition of $\rock[\infty]{\Phi_{\co0n}}$ gives that
%\begin{equation*}
$\rock[\infty]{\seq\Phi}\supset\bigcap_{m\ge 0}\rock[2]{\Phi_{\co0m}}\supset\rock[2]{\seq\Phi}.$

%\end{equation*}
The converse is trivially true, and Point \ref{i:limlim} of Lemma \ref{l:simlim} already tells us that $\rock[\infty]{\seq\Phi}\subset\Omega_{F_0}$.
\\
Moreover, since $F_n\simu[S_nS_{n+1},T_nT_{n+1}]F_{n+2}$ with $T_nT_{n+1}\ge 4$, then Remark \ref{r:locvalid} gives that $\Omega_{F_n}^2$ is non-empty.
Therefore, each of the uncountably many subsets in the disjoint union expressing $\rock[2]{\seq\Phi}$ is a closed non-empty intersection.
\item If $n\in\N$ is such that $F_0\simu[S_0\cdots S_{n-1},T_0\cdots T_{n-1},\Phi_{n-1}\cdots\Phi_0]F_n$ completely, then
by Point \ref{i:limlim} of Lemma \ref{l:simlim}, $\Omega_{F_0}=\rock[\infty]{\Phi_{\co0n}}$.
\item This follows from repeated application of Remark~5.4 and the fact that $\Phi_{\co{0}{n}}$ is a decoding function for all $n \in \N$.
\end{enumerate}\end{proof}

The following extends Lemma \ref{l:nonvide} (which can be recovered by $\B_i$ being singletons). In this case, every RPCA simulates a disjoint union of RPCA, each one of which simulates a disjoint union of RPCA and so on. In this way, we obtain an \xpr{infinite tree} of simulations. Along any branch of this tree, Lemma~\ref{l:nonvide} is true, but, more importantly, something similar is true even when we take all the (possibly uncountable) branches of this tree together.

\begin{lemma}\label{l:nonvides}~\begin{enumerate}
\item\label{i:uncsim} Let $(\B_n)_{n\in\N}$ be a sequence of finite alphabets, such that for any word $u\in\prod_{i<n}\B_i$ of length $n\in\N$, there exist $S_u,T_u, Q_u \in \N$, a decoding function $\Phi_u$ and a RPCA $F_u$ such that %for any $n\in\N$ and any $u\in\prod_{i<n}\B_i$,
$F_u\simu[S_u,T_u,\Phi_u]\bigsqcup_{b\in\B_n}F_{ub}$.

%Denote, for any $j\in\N\sqcup\{\infty\}$ and $z\in\prod_{i\in\N}\B_i$,

Let $\rocks[j]z{\seq\Phi}\defeq\bigcap_{n\in \N}\rock[j]{\Phi_{z_{\co0n}}}$ for all $j\in\N\sqcup\{\infty\}, z\in\prod_{i\in\N}\B_i$.

Then, for any $j\in\N \sqcup \{\infty\}$ and any closed $Y\subset \prod_{i\in\N}\B_i$,
%\begin{equation*}
$\rocks[j]Y{\seq\Phi}\defeq\bigsqcup_{z\in Y}\rocks[j]z{\seq\Phi}$
%\end{equation*} 
is a disjoint union and a subshift, and $\rocks[2]Y{\seq\Phi}=\rocks[\infty]Y{\seq\Phi}\subset \Omega_{F_\motvide}$.
\item Besides, the set $Z\defeq\set z{\prod_{i\in\N}\B_i}{\rocks[\infty]z{\seq\Phi}\ne\emptyset}$ corresponding to nested nontrivial, non-empty simulations is closed. %with language $\sett u{\forall n,\exists v%\in\prod_{\length u\le i<\length u+n}\B_i,\dom(F_{uv})\ne\emptyset}$.
If the simulations are complete, then $\rocks[2]Z{\seq\Phi}=\rocks[\infty]Z{\seq\Phi}=\Omega_{F_\motvide}$.
%\item If %$Z\subset\B^\Z$ is the closed set whose prefixes $u\in\lang(Z)$ correspond to nontrivial simulations, and $
%the simulations (in the hypothesis of Point \ref{i:infsim}) are nontrivial, then $\Sigma_{z}$ is uncountable if $z\in Z$, empty otherwise.
\end{enumerate}\end{lemma}

In the above statement, the notation $\Phi^z_{{\co0n}}$ stands for the composition $\Phi_{z_{\co{0}{n}}}\cdots\Phi_{z_0}\Phi_{\motvide}$, which is the decoding function from $F_{z_{\co{0}{n}}}$ onto $F_{\motvide}$.

\begin{proof}~\begin{enumerate}
\item
Point \ref{i:nonempty} of Lemma \ref{l:nonvide} gives that $\rocks[\infty]z{\seq\Phi}\ne\emptyset$  if $F_{z_{\co0n}}\simu F_{z_{\cc{0}{n+1}}}$ non trivially for any $n\in\N$, \ie all these RPCA have non-empty domain. The converse is obvious.
\\
By the same distributivity of decreasing intersections over unions as for Point \ref{i:infsim} of Lemma~\ref{l:nonvide}, it can be easily seen that
\[\rocks[j]{Y}{\seq\Phi}=\bigcap_{n\in\N}\bigsqcup_{u\in\lang_n(Y)}\bigsqcup_{\begin{subarray}c0\le t<\prod_{i<n}T_{u_{\co0i}}\\0\le s<\prod_{i<n}S_{u_{\co0i}}\end{subarray}}\sigma^sF_\motvide^t\Phi_\motvide^{-1}\Phi_{u_0}^{-1}\cdots\Phi_{u}^{-1}(\Omega_{F_u}^j)~,\]
which is a decreasing intersection of finite unions of subshifts, and we have $\rocks[2]z{\seq\Phi}=\rocks[\infty]z{\seq\Phi}\subset\Omega_{F_\motvide}$ for all $z\in Y$.
\item If $F_u\simu\bigsqcup_{a\in\B_n}F_{ua}$ completely, then Point \ref{i:limlim} of Lemma \ref{l:simlim} gives
\[\Omega_{F_u}=\bigsqcup_{\begin{subarray}c0\le t<T_u\\0\le s<S_u\end{subarray}}F_u^t\sigma^s\Phi_u^{-1}(\bigsqcup_{a\in\B_n}\Omega_{F_{ua}})~.\]
An immediate induction gives for any $n\in\N$,
\[\Omega_{F_\motvide}=\bigsqcup_{u\in\lang_{n+1}(Z)}\bigsqcup_{\begin{subarray}c0\le t<\prod_{i<n}T_{u_{\co0i}}\\0\le s<\prod_{i<n}S_{u_{\co0i}}\end{subarray}}\sigma^sF_\motvide^t\Phi_\motvide^{-1}\Phi_{u_0}^{-1}\cdots\Phi_{u_{\co0n}}^{-1}(\Omega_{F_u})~.\]
%which is included in $\bigsqcup_{u\in\lang_{n+1}(Z)}\bigsqcup_{\begin{subarray}c0\le t<\prod_{i<n}T_{u_{\co0i}}\\0\le s<\prod_{i<n}S_{u_{\co0i}}\end{subarray}}\sigma^sF_\motvide^t\Phi_\motvide^{-1}\cdots\Phi_{u_{\co0n}}^{-1}(\Omega_{F_u}^1)$.
Being true for any $n$, this gives the result.
\end{enumerate}\end{proof}
%It is known that all nonempty aperiodic subshifts are uncountable; we give here an explicit expression that corresponds to some kind of skeleton for the configurations in our system.

Lemmas~\ref{l:nonvide} and~\ref{l:nonvides} can be seen as extensions of Lemma~\ref{l:simlim} in the case of an infinite nested simulation. The following lemma can be seen as such an extension of Remark~\ref{r:penrose}.

\begin{lemma}\label{lem:aperiodichierarchy}
%The first remark is that, unless with trivial parameters, this cannot give any periodic configuration.
If $F_0\simu[S_0,T_0]F_1\simu[S_1,T_1]\ldots\simu[S_{n-1},T_{n-1}]F_n\simu[S_{n},T_{n}]\ldots$ completely, with $S_n,T_n>1$ for any $n\in\N$, then $\orb{F_0}$ is aperiodic.
\end{lemma}

In particular, either $\Omega_{F_n} = \emptyset$ ($= \orb{F_n}$) for all $n \in \N$ or, $\Omega_{F_n}$  (and $\orb{F_n}$) is aperiodic uncountable, for all $n \in \N$.

%Actually, from the fact that an aperiodic SFT has no configuration with either vertical or horizontal period, it would be enough to suppose that infinitely many of either the $S_n$ or the $T_n$ are strictly greater than $1$.
\begin{proof}
From Lemma~\ref{l:simtrans}, $F_0\simu[S_0\cdots S_{n-1},T_0\cdots T_{n-1}]F_n$ completely.
By Remark~\ref{r:penrose}, $\orb{F_0}$ cannot have any nontrival period less than $S_0\cdots S_{n-1}$ horizontally and less than $T_0\cdots T_{n-1}$ vertically.
If these two products go to infinity, we get that there cannot exist any periodic points.
%\\
%The second point comes from Lemma \ref{l:nonvide}, applied to the simulation sequences starting at any $n$.
\end{proof}

In fact, it follows from the proof that it is enough that one of the products $\prod_{i \in \N} S_i$ and $\prod_{i \in \N} T_i$ is infinite.
%By applying (possibly starting at any $n$) Lemma \ref{l:aperiodichierarchy} and Lemma \ref{l:nonvide}, we can recapitulate the following dichotomy.
%\begin{remark}\label{rem:nonemptyhierarchy}
%If $F_0\simu[S_0,T_0]F_1\simu[S_1,T_1]\ldots\simu[S_{n-1},T_{n-1}]F_n\simu[S_{n},T_{n}]\ldots$ completely, with $S_n,T_n>1$ for any $n\in\N$, then all limit sets $\Omega_{F_n}$ (and all SFT $\orb{F_n}$) are simultaneaously: empty if $\exists m\in\N,\dom(F_m)=\emptyset$; aperiodic uncountable otherwise.
%\end{remark}
%Note that if, on the other hand, $F_0, \ldots, F_m$ is a finite sequence of RPCA and for all $n\in\co0m$ there exist $S_n,T_n$ such that $F_n$ completely $(S_n,T_n)$-simulates $F_{n+1}$, and $\orb{F_m}$ is empty, then $\orb{F_n}$ is also empty for all $n\in\co0m$.
It is well known that a non-empty, aperiodic 2D SFT is uncountable. Lemma~\ref{l:nonvide} gives some additional information about how uncountability occurs in the case of an infinite nested simulation.

\section{Expansiveness and simulation}
The following lemmas highlight the relation between the notions of simulation and expansive directions. This subsection extends slightly Section~5 in \cite{nexpdir}.
%It is largely independent from the rest of the paper, so that it can be omitted on a first reading and read independently later on. For the largest part of the paper, we will only need and use Proposition~\ref{prop:hochman}.
The following lemmas correspond to Lemma~5.1 and Lemma~5.3 in \cite{nexpdir}, which examine how the so-called ``shape of prediction'' evolves.
It also motivates the choice of considering the horizontal direction as $\infty$, which will make many future expressions clearer.
\begin{lemma}\label{lem:relsimulexp}
Suppose $F\simu[S,T,Q]G$ exactly.
Then $\NE(F)\supseteq\frac1{T}\left(Q+S\NE(G)\right)$.
Moreover, if the simulation is complete, then $\NE(F)=\frac1{T}\left(Q+S\NE(G)\right)$.
\end{lemma}
In particular, $\NE(\sigma^{-Q}G)=\NE(G)+Q$ and $\NE(G^T)=\frac1T\NE(G)$.

\begin{proof}
Let us consider the matrix $M\defeq\left[\begin{array}{cc}S&Q\\0&T\end{array}\right]$ as acting over $\R^2$.
Consider a slope $l\in\Rb$, $\lin l\subset\R^2$ the corresponding vectorial line, $\lin l'\defeq M\lin l$ the vectorial line corresponding to slope $\frac STl+\frac QT$.
Roughly, $\lin l'$ for $F$ corresponds to $\lin l$ for $G$.

\begin{itemize}
\item Consider a finite shape $W'\subset\R^2$, $U$ and $f$ the neighbourhood and local rule of $F$, $V$ and $\phi^{-1}$ those of $\Phi^{-1}$, as defined in Remark \ref{r:simhedlund}. Without loss of generality, we can assume that $U = \cc{-uS}{uS}$, for some $u \in \N$.

Let $W\defeq M^{-1}W'+(T\cc{-u}{u} + V + \co{-Q}{0}) \times \{0\} + [-1,2[\times[0,1[$.
If $l\in\NE(G)$, then there exist configurations $x \neq y\in\Omega_G$ such that $\orb[x]G\restr{\lin l+W}=\orb[y]G\restr{\lin l+W}$.

Then, $\orb[\Phi^{-1}(x)]F \neq \orb[\Phi^{-1}(y)]F$, but we claim that 
\begin{equation*}
\orb[\Phi^{-1}(x)]F\restr{\lin l'+W'}=\orb[\Phi^{-1}(y)]F\restr{\lin l'+W'}.
\end{equation*}
Since $W'$ was an arbitrary finite shape, this implies that $\frac STl+\frac QT \in \NE(F)$, which proves that $\NE(F)\supseteq\frac1{T}\left(Q+S\NE(G)\right)$.

Let us proceed with the proof of the claim. Let $(p_1,p_2) \in \lin l'+W'$ and write $p_1\eqdef mS+r$, $p_2\eqdef nT+q$ and $n \defeq m'S+r'$, where $m,r,n,q,m',r' \in \Z$ and $0\le r,r'<S$ and $0\le q<T$. Intuitively, we can think that $(p_1,p_2)$ belongs to the encoding of the $m$'th letter of $\sigma^{-m'Q}G^n(x)$ and $G^n(y)$.

More precisely, a straightforward computation shows that 
\begin{equation*}
M^{-1}(p_1,p_2)=(m-Qm'+r/S-r'/{S}+q/T,n+q/T),
\end{equation*}
so that $(m-Qm',n) \in M^{-1}(p_1,p_2)+ [-1,2[ \times [0,1[$. This, in turn, implies that $(m-Qm',n) + (T\cc{-u}{u}+\co{-Q}{0}+V) \times \{0\}$ is included in $\lin l+ W$, so that
\begin{equation*}
\orb[x]{G}\restr{(m-Qm',n)+(T\cc{-u}{u}+\co{-Q}{0}+V)\times \{0\}} =
= \orb[y]{G}\restr{(m-Qm',n)+(T\cc{-u}{u}+\co{-Q}{0}+V)\times \{0\}}.
\end{equation*}

Using the facts that $V$ is the neighbourhood of $\phi^{-1}$ and that $\Phi^{-1}$ \xpr{blows-up} letters into blocks of size $S \times T$ with an additional shift of $Q$ for every vertical time step, we deduce that 
\begin{multline*}
\orb[\Phi^{-1}(x)]{F}\restr{(\co{(m-Qm')S}{(m-Qm'+1)S}+T\cc{-uS}{uS}+\co{-QS}{0}+nQ)\times \{nT\}} = \\
=\orb[\Phi^{-1}(y)]{F}\restr{(\co{(m-Qm')S}{(m-Qm'+1)S}+T\cc{-uS}{uS}+\co{-QS}{0}+nQ)\times \{nT\}}.
\end{multline*}

Notice that $nQ - Qm'S = r'Q$. Now, using the fact that $T\cc{-uS}{uS}$ is a neighbourhood for $f^q$ and $\co{-QS}{0}$ for $\sigma^{-r'Q}$, we obtain that 

\begin{multline*}
\sigma^{-r'Q}F^q\left(\orb[\Phi^{-1}(x)]{F}\right)\restr{\co{mS+r'Q}{(m+1)S+r'Q} \times \{nT\}} = \\
=\sigma^{-r'Q}F^q\left(\orb[\Phi^{-1}(y)]{F}\right)\restr{\co{mS+r'Q}{(m+1)S+r'Q} \times \{nT\}}.
\end{multline*}

The last equality implies that $\orb[\Phi^{-1}(x)]{F}\restr{(p_1,p_2)} = \orb[\Phi^{-1}(y)]{F}\restr{(p_1,p_2)}$, because

\begin{eqnarray*}
& & \sigma^{-r'Q}F^q\left(\orb[\Phi^{-1}(x)]{F}\right)\restr{\co{mS+r'Q}{(m+1)S+r'Q} \times \{nT\}}
\\&=&
\orb[\Phi^{-1}(x)]{F}\restr{\co{mS}{(m+1)S} \times \{nT+q\}}
\\&=&
\orb[\Phi^{-1}(x)]{F}\restr{\co{mS}{(m+1)S} \times \{p_2\}}
\end{eqnarray*}

and $p_1 \in \co{mS}{(m+1)S}$.

\item Consider a finite shape $W\subset\R^2$, $U$ the synchronizing shape as defined in Remark \ref{r:simsynchr}, $V$ and $\phi$ the neighbourhood and local rule of $\Phi$ as defined in Remark \ref{r:simhedlund}, and $W'\defeq MW+(V\cup U)\times\{0\}-\co0S\times\co0T$.

If $\frac STl+\frac QT\in\NE(F)$, then there exist configurations $x \neq y\in\Omega_F$ such that $\orb[x]F\restr{\lin l'+W'}=\orb[y]G\restr{\lin l'+W'}$.

By Remark \ref{r:simsynchr} and completeness of the simulation, there exist common $s\in\co0S$ and $t\in\co0T$ such that $x'\defeq\sigma^{-s}F^{-t}(x)$ and $y'\defeq\sigma^{-s}F^{-t}(y)$ are in $\dom(\Phi)$.
It follows easily from the definitions that $\orb[x']F\restr{\lin l'+MW+V\times\{0\}}= \orb[y']F\restr{\lin l'+MW+V\times\{0\}}$.% coincide over $\lin l'+MW+V\times\{0\}$.

By injectivity of $\Phi$, $\orb[\Phi(x')]G$ and $\orb[\Phi(y')]G$ are also distinct, but we claim that they coincide in $\lin l+W$. Since $W$ is an arbitrary finite shape, this implies that $l \in \NE(G)$, which proves that $\NE(F)\subseteq\frac1{T}\left(Q+S\NE(G)\right)$.

Let $(p_1,p_2) \in \lin l+W$. Then, $M(p_1,p_2)+V\times\{0\}\subset \lin{l'} +MW+V\times\{0\}$; it follows from this that
\begin{equation*}
 \orb[x']F\restr{(p_1S+p_2Q+V,p_2T)}= \orb[y']F\restr{(p_1S+p_2Q+V,p_2T)}.
\end{equation*}

In addition, we have that

\begin{eqnarray*}
\orb[\Phi(x')]G_{(p_1,p_2)}&=&
G^{p_2}\Phi(x')_{p_1}
\\&=&
\Phi\sigma^{p_2Q}F^{p_2T}(x')_{p_1}
\\&=&
\phi(\sigma^{p_2Q}F^{p_2T}(x')\restr{p_2S+V})
\\&=&
\phi(\orb[x']F\restr{(p_1S+V+p_2Q,p_2T)})
\end{eqnarray*}

The same holds for $y'$, and since, as we have noticed earlier, the final expression is the same for $x'$ and $y'$, we get that $\orb[\Phi(x')]G\restr{(p_1,p_2)} = \orb[\Phi(y')]G\restr{(p_1,p_2)}$, as claimed.
\end{itemize}\end{proof}

%For vectors $\vec S,\vec T\in\Ns^n$ and $\vec Q\in\Z^n$ with length $n\in\N$, we define $\Theta\defeq(\prod_{i<n}S_i/T_i)[-1,1]-\anib[\vec S/\vec T]{0\vec Q}$.
Lemmas \ref{l:simtrans} and \ref{lem:relsimulexp} can be combined to obtain expansive directions in nested simulations, which will be used extensively in Section~\ref{sec:expdir}.
\begin{lemma}\label{lem:iterrelsimulexp}
If $F_0\simu[S_0,T_0,D_0S_0]F_1\simu[S_1,T_1,D_1S_1]\ldots\simu[S_{n-1},T_{n-1},D_{n-1}S_{n-1}]F_{n}$ completely exactly, and all these RPCA have bi-radius $1$, then
\begin{equation*}
\NE(F_0)\subseteq\anib[\vec{1S/T}]{\vec{S}\vec{D}}+\left(\prod_{i<n}\frac{S_i}{T_i}\right)[-1,1]=\anib[\vec{S/T}]{\vec{D}}+\left(\prod_{i<n}\frac{S_i}{T_i}\right)[-1,1]~.
\end{equation*}
\end{lemma}
%The quotient of vector $\vec S$ by $\vec T$ is considered coordinatewise.
%\TODO[P]{maybe just after the def of anib: $0\vec D$ stands for the vector $(0,D_0,D_1,\ldots,D_{n-1})$.}
%Would we allow rational shifts, we could see in some sense that $F_0\simu[\prod_{i<n}S_i,\prod_{i<n}T_i]\sigma^{\anib[\vec S/\vec T]{0\vec D}}F_n$.
\begin{proof}
We already noted that the radius of a RPCA $F_n$ with bi-radius $1$ has $\NE(F_n)\subseteq[-1,1]$.
From Lemma \ref{l:simtrans}, we know that $F_0\simu[S,T,Q]F_n$ exactly completely, where $(S,T,Q) = (\prod S_i,\prod T_i,{\anib[\vec{1S/T}]{\vec S \vec D}\prod T_i})$ and from Lemma \ref{lem:relsimulexp}, we deduce that:
\[\NE(F_0)=\anib[\vec{1S/T}]{\vec S \vec D}+\left(\prod_{i<n}\frac{S_i}{T_i}\right)\NE(F_n)\subseteq \anib[\vec{1S/T}]{\vec S \vec D}+\left(\prod_{i<n}\frac{S_i}{T_i}\right)[-1,1]~.\]

Also, by definition we have that $\anib[\vec{1S/T}]{\vec S \vec D}=\anib[\vec{S/T}]{\vec D}$.
%Let us show this inductively on $n\in\N$.
%If $n=0$, the statement corresponds to the remark that RPCA with bi-radius $1$ have $\NE(F_0)\subseteq[-1,1]$.
%Now, for $n\in\N$, let us show the result for the sequence of simulations between $F_0$ and $F_{n+1}$, assuming it is true for the sequence of simulations between $F_1$ and $F_{n+1}$: $\NE(F_1)\subseteq(\prod_{1\le i\le n}S_i/T_i)[-1,1]+\anib[\vec S_{\cc1n}/\vec T_{\cc1n}]{0\vec D_{\cc1n}}$.
%From Lemma \ref{lem:relsimulexp}, $\NE(F_0)\subseteq\frac{S_0}{T_0}\NE(\sigma^{-D_0}F_1)$, which by Lemma \ref{lem:relshiftexp} is equal to
%\[\frac{S_0}{T_0}(\NE(F_1)+D_0)\subseteq\left(\prod_{0\le i\le n}\frac{S_i}{T_i}\right)[-1,1]+\anib[\vec S/\vec T]{0\vec D}~.\]
\end{proof}

%\begin{remark}\label{rem:widthnonexp}
%Notice that $\norm{\inv(\NE(F_0))} \le 2 \frac{S_0 \cdots S_{n-1}}{T_0 \cdots T_{n-1}}$. This will be really helpful in the proofs of Section~\ref{sec:expdir}.
%\end{remark}

In the limit case of an infinite nested simulation, we obtain the following proposition, which slightly extends Theorem~5.4 in \cite{nexpdir}.
\begin{proposition}\label{prop:hochman}
%If $F_0\simu[S_0,T_0,D_0S_0]F_1\simu[S_1,T_1,D_1S_1]\ldots\simu[S_{n-1},T_{n-1},D_{n-1}S_{n-1}]F_n\simu[S_n,T_n,D_nS_n]\ldots$ completely exactly, then,
If $F_i \simu[S_i,T_i,D_iS_i] F_{i+1}$ completely exactly, for all $i \in \N$, then
\begin{equation*}
\NE(F_0)\subseteq\anib[\seq S/\seq T]{\seq D}+\left(\inf_{n\in\N}\prod_{i<n}\frac{S_i}{T_i}\right)[-1,1]~.
\end{equation*}
In particular, if the simulations are non-trivial and $\prod_{i<n}S_i/T_i$ converges to $0$, then $\NE(F_0)=\{\anib[\seq S/\seq T]{\seq D}\}$.
\end{proposition}
%Note that, in the case when $\prod_{i<n}S_i/T_i$ converges to $0$, the interval is a singleton, and thanks to nonemptiness, we get an equality.
\begin{proof}
From Lemma \ref{lem:iterrelsimulexp}, we know that \begin{equation*}
\NE(F_0)\subseteq\bigcap_{n\in\N}\left(\prod_{i<n}\frac{S_i}{T_i}\right)[-1,1]+\anib[\seq{S}_{\co{0}{n}}/\seq{T}_{\co{0}{n}}]{\seq{D}_{\co{0}{n}}},
\end{equation*}
which gives the wanted inclusion, when $n$ goes to $\infty$.

For the second claim, if all the simulations are non-trivial, then from Lemma \ref{l:nonvide} we know that $\orb{F_0}$ is uncountable, hence by Proposition \ref{p:atleastone}, it has at least one non-expansive direction. In addition, by the first claim and the assumption $\prod_{i<n}S_i/T_i \to 0$, we know that $\NE(F_0) \subseteq \{\anib[\seq S/\seq T]{\seq D}\}$, and we must actually have equality.
\end{proof}

\section{Explicit simulation}\label{sub:simulconv}
In the previous sections of this chapters, we defined a notion of simulation and then proved some facts about this notion, which suggest that it is a good choice. However, we have not given any non-trivial example of simulation until now, nor have we explained how this could happen. For example, the decoding function $\Phi$ could be anything.

The simulation that we construct all have the same basic \xpr{form}. We call these simulation \dfn{explcit}, because the simulated configuration is explicitly written letter by letter in the simulating configuration. In order to make this  more precise, we need to give some more definitions and notations. 

%Let $s \in \co{0}{S}$ and $t \in \co{0}{T}$. Let $\gra{s}{t}{S}{T} \subseteq (\haine5^{**})^{\Z}$ denote the set of configurations such that $\bina{\pi_{\addr}}(c)= \dinf{(s\ldots(S-1)0\ldots(s-1))}$ and $\bina{\pi_{\age}}(c)=\dinf{t}$. In other words, the binary value of $\addr$ (after removing the initial $4$'s) increases by $1 \mod S$ from left to right with $s$ at the origin, while the binary value of $\age$ is constant and equal to $t$.

%Let $\gra{s}{t}{S}{T}$ denote the set of well-structured configurations with shift $(s,t)$.

Let us fix a some fields $\addr$, $\addr_{+1}$, $\age$ and $\age_{+1}$ (In fact, these are just distinct numbers that we use to project letters on). These are sometimes called \emph{coordinate} fields. For $s \in \co{0}{S}$ and $t \in \co{0}{T}$, let $\gra{s}{t}{S}{T} \defeq \per[s,S]{\addr,\addr_{+1}} \cap \emp[t]{\age,\age_{+1}}$. In $\gra{s}{t}{S}{T}$ the values of $\addr$ grow by $1$ modulo $S$ from left to right and the value of $\age$ is constant and equal to $t$, while the origin has $\addr$ $s$. This is the usual way to break up a configuration into blocks, with one small difference. Normally, we only need the fields $\addr$ and $\age$ to do this. However, since we are using PPA, we need to have some right- (or left-) moving copies of these fields in order to check the compatibility of these fields. Having this in mind, we define $\gra{s}{t}{S}{T}$ in the above way, since it will make notation a little lighter later on. The union $\bigsqcup_{\begin{subarray}c0\le t<T\\0\le s<S\end{subarray}}\gra{s}{t}{S}{T}$ is disjoint.

In addition, let $\grs{s}{S}\defeq \per[s,S]{\addr,\addr_{+1}}$. In $\grs{s}{S}$, we do not care about the value of $\age$ (or if it is even constant). Clearly, $\gra{s}{t}{S}{T} \subseteq \grs{s}{S}$. For $c \in \grs{s}{S}$ and $i \in \Z$, the pattern
\begin{equation*}
\col{i}{c}=c_{\co{-s+iS}{-s+(i+1)S}}
\end{equation*} 
is called a \dfn{colony} of $c$. Clearly, $(\col{i}{c})_{i \in \Z}=\bulk[S]{\sigma^{-s}(c)}$ and in $\col{i}{c}$, the value of $\addr$ (and $\addr_{+1}$) grows from $0$ to $S-1$ from left to right.

This is the natural way to break a configuration into colonies of size $S$. Now, we are going to use every colony to encode one letter of the simulated configuration. For this, we have to define the appropriate decoding function.

Let $\tilde{\phi} \colon (\haine5^{*})^* \pto \haine5^{**}$ be the following function, which is the basis of all the decoding functions that we will use: Let $w \in (\haine5^*)^*$ be a \emph{word} over the infinite alphabet $\haine5^*$ (we look at $w$ as a finite part of some $1$D configuration over $\haine5^*$). If $\hs{w}= \Chi{\vec{u}}3^{\length{w}-\length{\Chi{\vec{u}}}}$, where $\vec{u} \in \haine5^{**}$ (we look at $\vec{u}$ as a tuple of elements of $\haine5^*$), then we define $\tilde{\phi}(w)=\vec{u}$.

Notice that $\Chi{\vec{u}} \in \haine3^*$. In other words, $w$ is equal to $\Chi{\vec{u}}$ up to appending some $3$s at the end of $\Chi{\vec{u}}$ (this gives a word in $\haine4^*$) and then adding some $4$s in front of every \emph{letter} of $\Chi{\vec{u}}3^{\length{w}-\length{\Chi{\vec{u}}}}$ (which gives a word in $(\haine5^{*})^*$). Unless $w$ has this very specific form, $\tilde{\phi}(w)$ is not defined.

%$\tilde{\phi}$ is a function that takes as input words of $(\haine5^*)^*$ and outputs letters $\vec{u} \in \haine5^{**}$. We now extend it in various ways: First of all, we extend it to $\tilde{\phi}_{\field}$ which takes as input words of $(\haine5^{**})^*$, ignores all fields except $$ and outputs a letter $\vec{u} \in \haine5^{**}$. 

%after removing initial $4$'s from the letters of $w$, we obtain the encoding of $\vec{u}$ followed by some trailing $3$'s. Notice that here, we think of $w \in (\haine5^*)^*$ as a word over the (infinite) alphabet $\haine5^{*}$, while we think of $\vec{u} \in \haine5^{**}$ as a tuple of words of $\haine5^*$.

$\tilde{\phi}$ is well-defined because $\Chi{\cdot}$ is an injection and because $3$ does not appear as a letter of $\Chi{\vec{u}} \in \haine3^*$. A necessary condition so that $\tilde\phi(w)=\vec{u}$ is that $\length{w} \geq \length{\Chi{\vec{u}}}$. 

Let $\field$ be a new field and $\tilde{\phi}_{\field} \colon (\haine5^{**})^* \pto \haine5^{**}$ be defined as $\tilde{\phi} \pi_{\field}$. $\tilde{\phi}_{\field}$ can read words over letters with many fields by ignoring the other fields and using $\tilde{\phi}$ on $\field$. 

We can extend $\tilde{\phi}$ in a natural way to a map $\tilde{\Phi} \colon (\haine5^*)^{\Z} \pto (\haine5^{**})^{\Z}$ as follows: for all $c \in (\haine5^{*})^{\Z}$ and $i \in \Z$, $\tilde{\Phi}(c)_i=\tilde{\phi}(c\restr{\co{iS}{(i+1)S}})$. Similarly, $\tilde\phi_{\field}$ can be naturally extended to a map $\tilde{\Phi}_{\field} \colon (\haine5^{**})^{\Z} \pto (\haine5^{**})^{\Z}$.

The idea is that every configuration will be divided into colonies using the coordinate fields and then $\tilde{\phi}_{\field}$ will be used on every colonies so as to obtain a letter. Putting these letters together, we obtain the simulated configuration. 

Formally, a decoding function $\Phi$ will be equal to $\tilde{\Phi}_\field{\restr{\Sigma}}$, where $\Sigma \subseteq \grs{0}{S}$, for some $S$ that is large enough. If $b_i = \tilde{\phi}_\field(\col{i}{c})$, then $\Phi(c)= \tilde{\Phi}_\field(c)=(b_i)_{i \in \Z}$. We call $b_i$ the \dfn{simulated letter} of the $i$'th colony and the letters of $c$ are the \dfn{simulating letters}.%, or \dfn{letters of the simulating configuration}.

%$\tilde{\Phi}_{\field}$ is a function that takes as input configurations of $\haine5^{**}$ and outputs other configurations. It will be an essential part of all the decoding functions that we will define.

The decoding functions that we will use in our constructions will \emph{always} be of the form ${\tilde{\Phi}_{\field}}{\restr{\Sigma}}$, where $\Sigma \subseteq \grs{0}{S}$. For such functions, we immediately obtain two of the conditions of a decoding function of a simulation:

\begin{remark}\label{twoconditionsofsimulation}
Let us fix a field list $\C=[\addr,\addr_{+1},\age,\age_{+1},\info]$, $S \in \Ns$ and  vectors $\vec{k}, \vec{k'} \in \N^{*}$ such that the following inequalities hold:
\[\both{
k_\addr\ge\norm S\\
%k_\age\ge\norm T\\
k_\info\ge 1\\
S \geq \length{\Chi{\haine5^{\vec{k'}}}}
~,}\]
Let $\Sigma \defeq (\haine5^{\vec{k}})^{\Z}\cap \grs{0}{S}\cap \tilde{\Phi}_{\info}^{-1}((\haine5^{\vec{k'}})^{\Z})$. Then $\Phi \defeq {\tilde{\Phi}_{\info}}{\restr{\Sigma}} \colon (\haine5^{\vec{k}})^{\Z} \to (\haine5^{\vec{k'}})^{\Z}$ is surjective and $\Phi \sigma^S = \sigma \Phi$.

In addition, for every $b \in (\haine5^{**})^{\vec{\Z}}$, we are free to chose the values of the anonymous fields in any way we like in a pre-image.
\end{remark}

In the above remark, $\Sigma$ contains those configurations over $\haine5^{\vec{k}}$ that are well-structured (\ie divided into colonies with the origin having address $0$) and such that in the $i$'th colony we have the encoding of a letter of $\haine5^{\vec{k'}}$, for all $i \in \Z$.

\chapter{The programming language}\label{c:programming}
\section{Definitions and basic permutations}
In our constructions, we want to use permutations that are computed fast. It is not possible to formally state what fast means, but polynomially computable and, more generally, polynomially checkable permutations is fast enough. This is a common feature of all self-similar and hierarchical constructions and the reasons why it is needed are explained very thoroughly in \cite{gray}. For our purposes, it is enough to describe a pseudo-programming language, with which we will write \xpr{programs} that are interpreted as permutations $\alpha \colon \haine5^{**} \pto \haine5^{**}$. 

Let us start describing this programming language: It has four types, \dfn{terms} (that are denoted $t,t'\ldots$), \dfn{valuations} (that are denoted $v,v',\ldots$), \dfn{conditions} (that are denoted $c,c',\ldots$) and \dfn{permutations} (that are denoted $\alpha,\alpha',\ldots$). Each type is semantically interpreted as a different kind of mathematical object. Terms are interpreted as maps $t \colon \haine5^{**} \pto \haine5^{*}$. They represent some word information that can be extracted from a tuple. Valuations are interpreted as functions $v \colon \haine5^{**} \pto \N$. Valuations represent numerical information that can be extracted from tuples. Conditions are predicates over $\haine5^{**}$, or equivalently maps $q \colon \haine5^{**} \pto \{0,1\}$. Finally, permutations are, rather predictably, interpreted as (partial) permutations $\haine5^{**} \pto \haine5^{**}$ which will be used to define IPPA.%For every type, there exist some \dfn{basic} objects and some \dfn{inductive operations}. Applying the inductive operations to previously defined objects, we obtain new objects. The set of objects of every type is obtained as the transitive closure of this procedure.

%In fact, we do a global induction, which means that we apply the inductive operators of all types simultaneously. The reason that we need to do this is that, for example, there exist an inductive operation on terms that uses a previously constructed permutation, an inductive operation on permutations that uses a previously constructed condition and an inductive operator on conditions that uses a previously constructed term. More precisesly, we start with the sets of basic objects for every type, and we apply the corresponding inductive operators to them. This gives a new, larger set of objects for each type. We apply the inductive operators once more, and we obtain a still larger set of objects for every type, and so on.

Let us describe each type with more details. We are not going to try to give a formal definition of the programming language, since it is would be unnecessarily complicated. It would involve a global induction on the various types, starting from some basic objects and taking a closure under some inductive operations. Instead, we will simply list the objects that we are actually going to use in the rest of the thesis. The proofs that they are polynomially computable are often trivial and will be omitted in most cases.

%The syntax of the programming language is the following:

%\begin{eqnarray*}
%t&:&w~|~\pi_i~|~\hs{t}~|~%\double{t}~|~
%{t}\restr{v,v'}~|~tt'~|~t\circ\alpha\\
%v&:&\length{t}~|~\bina{t}~|~v+v'~|~v-v'~|~vv'~|~\ipart{v/v'}\\
%c&:&\halt pnt~v\geq v'~|%~v\equiv0\{v'\}~|
%~t=t'~|~c\wedge c'~|~\neg c\\
%\alpha&:&\beta[v,v';i]~|~\oplus[t,v;i]~|~\incr[v;i]~|~\alpha_{\U}[t;\info,\head_{-1},\head_{+1}]~|~\algorithmicif\ c\ \algorithmicthen\ \alpha~|~\alpha^{-1}~|~\alpha'\circ\alpha,
%\end{eqnarray*}
%where $w$ is a word from $\haine4^*$, $i,\info,\head_{-1},\head_{+1}$ denote mutually distinct field numbers (or labels), $\beta:\haine4\pto\haine4$ an alphabet partial permutation, and, inductively, $t,t'$ denote terms, $v,v'$ valuations, $c,c'$ conditions, $\alpha,\alpha'$ permutations, $p$ is the program of a TM and $n \in \N$

%We will now describe how the various objects of the programming language are interpreted semantically. We will use the same letter for an object of the programming language and its interpretation, hoping that this does not cause any unnecessary confusion.

\paragraph{Terms}
%Every term $t$ defines a partial function $t \colon \haine5^{**} \pto \haine5^*$ according to the following rules:
%For any $\vec u\in\haine5^{**}$:
\begin{itemize}
\item Every word $w \in \haine5^{*}$ is a term (understood as the constant function);
\item for all $i \in \N$, the projection $\pi_i$ of the $i$'th field is a term;
\item if $t$ is a term, then $\Chi{t}$ is also a term ($\Chi{t}(\vec{u})=\Chi{t(\vec{u})}$, for all $\vec{u}$ in $\haine5^{**}$);
\item if $v$ is a valuation and $t$ is a term, then $t\restr{v}$ is also a term,  where $t\restr{v}(\vec{u}) \defeq t(\vec{u})\restr{v(\vec{u})}$. In other words, $t\restr{v}$ uses $v$ as a pointer for $t$ and it gives the letter at the $v(\vec{u})$'th position of $t(\vec{u})$.
\end{itemize}

%\item $\hs{t}(\vec u)\defeq\hs{t(\vec u)}$;
%\item $\double{t}(\vec u)$ is the encoding $\double{t(\vec u)}$;
%\item ${t}\restr{v,v'}(\vec u)$ is the restriction ${t(\vec u)}\restr{\co{v(\vec u)}{v(\vec u)+v'(\vec u)}}$;
%\item $(tt')(\vec u)$ is the concatenation $t(\vec u)t'(\vec u)$;
%\item $(t \circ \alpha)(\vec u)\defeq t(\alpha(\vec u))$;
%\end{itemize}
%A term of the form $w$, where $w \in \haine5^*$ will essentialy be used when we check that a field is equal (up to $\hs{\cdot}$) to $\motvide$ or to the initial state $0$.

\paragraph{Valuations}

%Every valuation defines a partial function $v \colon \haine5^{**} \to \N$ according to the following rules:
%For any $\vec u\in\haine4^{**}$:
\begin{itemize}
\item Every natural $n \in \N$ is a valuation, understood as a constant function;
\item if $t$ is a term, then $\length{t}$ is a valuation;
\item For all vectors $\vec{k} \in \N^{*}$ and $i \in \N$, the function $l_{\vec{k},i}$ defined in Fact~\ref{f:encodings} is a valuation.
\item If $\seq{S} \colon \N \to \N$ is a sequence of numbers and $v$ a valuation, then $S_v$ (where $S_v(\vec{u}) \defeq S_{v(\vec{u})}$) is also a valuation. (In general, the complexity of this valuation depends on the complexity of $\seq{S}$ and it is not polynomially computable if $\seq{S}$ is not.)
\item Basic arithmetical operations (addition, subtraction, multiplication etc) of valuations are still valuations.
%\item $\length{t}(\vec u) \defeq \length{t(\vec u)}$;
%\item $\bina{t}(\vec u) \defeq \bina{\hs{t(\vec u)}}$;
%\item $(v+v')(\vec u)=v(\vec u)+v'(\vec u)$;
%\item $(v-v')(\vec u)=v(\vec u)-v'(\vec u)$ if $v'(\vec u)\le v(\vec u)$ (undefined otherwise);
%\item $(vv')(\vec u)=v(\vec u)\cdot v'\vec u)$;
%\item $\ipart{v/v'}(\vec u)$ is the quotient by Euclidean division of $v(\vec u)$ by $v'(\vec u)$.
\end{itemize}

In fact, we will need the following, more general version of the third bullet:

\begin{itemize}
\item For all valuations $v$, vector sequences $\vec{k} \colon \N \to \N^M$ and $i \in \N$, $l_{\vec{k}_v,i}$ (where $l_{\vec{k}_v,i}(\vec{u}) \defeq l_{\vec{k}_{v(\vec{u})},i}(\vec{u})$) is also a valuation. In this version, the vector whose structure $l_{\vec{k}_v,i}$ gives depends on the input letter. Of course, if $\vec{k}$ is not a polynomially computable sequence, then neither is $l_{\vec{k}_v,i}$.
\end{itemize}

A \dfn{vector valuation} is a collection $\vec{v}=(v_i)_{0 \le i \le M-1}$ of valuations, for some $M \in \N$. Vector valuations are used to obtain lengths of alphabets in a polynomially computable way.

\paragraph{Conditions}

%Here is an acceptable permutation that is just a little more complex and will be very useful: if $n$ is an acceptable valuation, $p$ is a program and $t$ a term, then $\halt pnt$ is the identity if $p$ stops within $n$ steps over term $t$, and is undefined otherwise.

%Every condition $q$ defines a unary predicate over $\haine5^{**}$ according to the following rules: (Formally, it defines a function $q \colon \haine5^{**} \pto \{0,1\}$, and we say that $\vec u$ satisfies $q$ if and only if $q(\vec u)=1$.)

%For all $\vec u \in \haine5^{**}$, $u$ satisfies:

\begin{itemize}
\item If $v_1, v_2$ are valuations, then $v_1 \geq v_2$ is a condition whose interpretation is clear;
\item if $t_1, t_2$ are terms, then $t_1 = t_2$ is a condition;
\item if $t,t_1$ are terms and $(Q_w)_{w \in \haine5^{*}}$ is a sequence of subsets of $\haine5^*$, then $t_1 \in Q_t$ is a condition. ($\vec{u}$ satisfies $t_1 \in Q_t$ if $t_1(\vec{u}) \in Q_{t(\vec{u})}$.) %Of course, $(Q_w)_{w \in \haine5^{*}}$ must be polynomially computable;
\item if $t$ is a term and $i_1, \ldots, i_n$ are fields, then $\emp[t]{i_1, \ldots, i_n}$ is a condition (that is true for $\vec{u}$ if and only if $\vec{u} \in  \emp[t(\vec{u})]{i_1, \ldots, i_n}$);
\item $\halt{p}{v}{t}$ is a condition, where  $\vec{u}$ satisfies $\halt{p}{v}{t}$ if and only if the TM defined by program $p$ does not stop within $v(\vec{u})$ steps over term $t(\vec{u})$;
\item boolean operations of conditions are also conditions.
%\item $\checkb(l)$ if and only if $u \in (\quatre^*)^l$.
%\item $v\ge v'$ if and only if $v(\vec u)\ge v'(\vec u)$;
%\item $v\equiv0\{v'\}$ if and only if $v(\vec u)$ is a multiple of $v'(\vec u)$; {This is not needed anymore. It can be expressed with the euklidean division.}{}
%\item $t=t'$ if and only if $t(\vec u)=t'(\vec u)$;
%\item ${\pi_i}\restr{t_1}={\pi_j}\restr{t_2}(w)$ is true iff ${\pi_i}_{\bina{t_1}(w)}={\pi_j}_{\bina{t_2}(w)}$;
%\item $c\wedge c'$ if and only if it satisfies both $c$ and  $c$;
%\item $\neg c$ if and only if it does not satisfy $c$.
\end{itemize}

\paragraph{Permutations}

%Every permutation $\alpha$ is interpreted as a partial permutation $\alpha \colon \haine5^{**} \pto \haine5^{**}$ according to the following rules:
%For $\vec u\in\haine5^{**}$:
\begin{itemize}
\item For every condition $q$, $\chekk[q]$ is a permutation. $\chekk[q](\vec{u})$ is equal to $\vec{u}$ if and only if $\vec{u}$ satisfies $q$ (and is undefined otherwise). This is an involution.
%\item Let $\vec{u'}$ be defined as follows: $u'_j \defeq u_j$ for any $j \neq i$, $ u'_{i,k} \defeq \beta(u_{i,k})$ for $k\in\co{v(\vec u)}{v(\vec u)+v'(\vec u)}$ and $u'_{i,k}=u_{i,k}$, otherwise;
%then $\beta[v,v';i](\vec u)\defeq\vec{u'}$ if $v(\vec u)=v(\vec{u'})$ and $v'(\vec u)=v'(\vec{u'})$ ($\beta[v,v';i](\vec u)$ is undefined if one of these equailities is not true);

%\item Let $\vec{u'}$ be defined as follows: $u'_j\defeq u_j$ for any $j\ne i$, $u'_{i,k}\defeq u_{i,k}+t(\vec u)_{k-v(\vec u)}\bmod5$ for $k\in\co{v(\vec u)}{v(\vec u)+\length{t(\vec u)}}$ and $u'_{i,k}=u_{i,k}$, otherwise ($\vec{u'}$ is undefined if $v(\vec u)+\length{t(\vec u)}>\length{u_i}$);
%then, $\oplus[t,v;i](\vec u)\defeq\vec{u'}$ if $v(\vec u)=v(\vec{u'})$ and $t(\vec u)=t(\vec{u'})$;
\item For every valuation $v$ and field $i \in \N$, $\incr[v,i]$ is a permutation defined in the following way: Let $\vec{u} \in \haine5^{**}$ and define $\vec{u'}$ in the following way: $u'_j\defeq u_j$ for all $j\ne i$, and $u'_i\defeq\sh[\length{u_i}]\gamma(u_i)$, where $\gamma(w)\defeq\anib{\bina w+1\bmod v(\vec u)}$ when $\bina w<v(\vec u)$ (undefined otherwise);
then $\incr[v;i](\vec u)\defeq\vec{u'}$ if $v(\vec u)=v(\vec{u'})$ (undefined otherwise).

Essentially $\incr[v;i]$ adds $1$ modulo $v(\vec{u})$ to the $i$'th field of $\vec{u}$. The additional complications are due to the fact that we want this rule to always be reversible (which would not necessarily be true if $v(\vec{u'})$ is not equal to $v(\vec{u})$) and length preserving (which is the reason that we use the strange $\gamma$ function).
\item  $\alpha_\U[t;\info,\head_{-1},\head_{+1}]$ is a permutation for every term $t$ and fields $\info$, $\head_{-1}$, $\head_{+1}$. We direct the reader to Section~\ref{Jarkko} for the definition of this permutation, since it uses a permutation that is defined and examined therein.

\item Let $t$ be a term and $i$ be a field such that $t$ \dfn{does not depend} on $i$. In other words, if $\vec{u},\vec{u'} \in \haine5^{**}$ and $\pi_j(\vec{u})=\pi_j(\vec{u'})$ for all $j \neq i$, then $t(\vec{u})=t(\vec{u'})$.

Then, $\rite[t;i]$ is a permutation defined as follows: Let $\vec{u} \in \haine5^{**}$. $\rite[t;i](\vec{u})$ is defined if and only if $\hs{u_i}= \motvide$. In this case, all fields remain the same except for $i$ which becomes equal to $\sh[\length{u_i}]{t(\vec{u})}$.

Essentially, we check that the field $i$ is empty and then write $t(\vec{u})$ on it, while preserving the lengths. The condition that $t$ does not depend on $i$ is essential to ensure reversibility.

$\rite[t;i]^{-1}$ first checks that the $i$'th field is equal to $t(\vec{u})$ and then empties it, while preserving the lengths. This is a way to reversibly erase some information from a letter, namely compare it with some other place of the letter where the same information is held.
\item For all fields $i,i'$, $\exch[i,i']$ is a permutation defined as follows: Let $\vec{u} \in \haine5^{**}$. $\exch[i,i'](\vec{u})$ is defined if and only if $\length{u_i}=\length{u_{i'}}$. In this case, all fields are unchanged except for $i$ and $i'$ whose values are exchanged.

This is a length-preserving involution.
\item For every condition $q$ and permutation $\alpha$, $\algorithmicif\ q\ \algorithmicthen\ \alpha$ is a permutation. On input $\vec{u}\in \haine5^{**}$, it applies $\alpha$ if condition $q(\vec{u})$ is satisfied and $q(\vec{u})=q(\alpha(\vec{u}))$. If $q(\vec{u})$ is satisfied and $q(\vec{u}) \neq q(\alpha(\vec{u}))$, then it is not defined on $\vec{u}$ (this ensures reversibility). Finally, if $q(\vec{u})$ is \emph{not} satisfied, it is equal to the identity.
\item The composition of permutations is also a permutation. In constructions, we will denote the composition $\alpha_2 \circ \alpha_1$ by writing $\alpha_2$ below $\alpha_1$.
\end{itemize}

In the definition, we check that the values of the valuations, terms and conditions that are given as parameters do not change. This is a technical point that ensures that they are interpreted as reversible functions.
In all our constructions, these conditions will easily be satisfied because the valuations, terms and conditions will either be constant or depend on fields that are not modified by the rule at hand.

If we were giving a complete, formal description of a language, then this would be the point where by a large, tedious induction we would prove that, given some natural conditions on the parameters, every permutation of the language is polynomially computable, or, more precisely, polynomially computable in its parameters (this means that its complexity is a polynomial of the complexity of its parameters) and that short programs exist for the permutations. Namely, the size of the program is $O(p_{t,v, \ldots})$, where $t,v$ etc. are the parameters of the permutation.

We can also prove that the size of a program of a permutation is approximately the same as the size of the program of its inverse. 
\section{Conventions about defining IPPA}\label{s:conv}

In the first part of this chapter, we gave a short exposition of the programming language that will be used in the rest of the thesis in order to define permutations of $\haine5^{**}$. However, in order to define a PPA, the number of fields and the directions of the fields also have to be fixed. 

Recall that we want to define PPA, \ie RPCA of the form $F = \sigma^{\vec\delta}\circ\alpha$, where $\vec\delta \in \{-1,0,+1\}^M$ is the shift vector and $\alpha$ is a partial permutation of $\A=\A_0 \times \ldots \A_{M-1}$, for some $M \in \N$. In our case, $F$ will always be the restriction of an IPPA, \ie $\A$ will be equal to $\haine5^{\vec{k}}$, for some $\vec{k} \in \N^M$ and $\alpha \defeq \beta\restr{\haine5^{\vec{k}}}$ will be the restriction of some (infinite) permutation $\beta$ defined in the programming language.

We will use the following conventions when constructing such PPA:

%This subsection is meant to help the reader understand what is going on in the next section.
%Thanks to our programming language, we will define many permutations in the following way:
\begin{itemize}
\item We first give a list of so-called \dfn{explicit} field labels. Such a list will often be noted in the form $\C\defeq[\field_{e},\ldots,\field'_{e'}]$, where $e,e' \in \{-1,0,+1\}$. The subscripts $e,\ldots,e'$ correspond to the \emph{directions} of the fields (if the direction is equal to $0$, then it will be omitted). The field list is a tuple of pairwise different natural projections, that are used by the permutation, together with their directions, that will be used by the shift. (The labels of the fields will make the permutations more understandable than the corresponding indices $i,i',\ldots$). The field list is not fixed, so in fact for every field list, we give a different permutation, even though they only differ in the enumeration of the fields.

The permutation is assumed to reject any element of $\haine5^{**}$ that does not involve all field numbers in the list, but note that it does not reject tuples that have more fields; the so-called \dfn{anonymous} fields, that are not in the list, are not modified by the permutation (but they might be used by some other PPA with which we compose). This allows us to define some simple PPA with few fields and then use them as \xpr{building blocks} in order to build more complicated ones in the following sense: the complicated PPA has more fields than the simple one, but, if it does not \xpr{touch} any of its fields, its behaviour on those fields is described by the corresponding behaviour of the building block.

If $\C$ and $\C'$ are two lists of field labels, then $\C \cup \C'$ is the list that contains the fields of $\C$ and $\C'$. Usually, the lists will be disjoint, so that we will use the notation $\C \sqcup \C'.$

\item After giving the field list, we describe an (infinite) permutation using the programming language defined in the first part of this chapter.

\item Then, we need to fix $M \in \N$ and $\vec{k} \in \N^M$. If we do not care about the existence of anonymous fields, then we always assume that $M$ is some number greater than or equal to the largest natural appearing in the field list $\C$. In this way, we ensure that the configurations will not be rejected simply because the program tries to access a field that is not there.

When we do not want anonymous fields to exist (for example, when we want to achieve exactness of a simulation), then we assume that the field list $\C$ is equal to $[0,\ldots,M-1]$ and we choose this $M$ for the number of fields.

In any case, after choosing $M$, we fix some vector $\vec{k}\in \N^M$ satisfying some appropriate conditions (which are case-specific).

\item Finally, we need to define the directions of the fields. However this has already been done in the definition of the field list with the use of the subscripts $e,e'$ etc. The directions of the anonymous fields can be anything. In fact, our statements will be true for \emph{all} directions of the anonymous fields, since we will not refer to them.
\end{itemize}

\chapter{The universal simulator}\label{construction}

In this chapter, our aim is to construct an RPCA (a family of RPCA in fact, depending on some parameters) that can simulate every other RPCA that satisfies some  conditions. This is done in Lemma~\ref{universal}. This RPCA is extremely helpful and it will be part of all our subsequent constructions. Since it is difficult to overstress the importance of this RPCA, we will give a step-by-step description of its construction with as many details as possible.

In Section~\ref{sec:structure}, we will embed a periodic rectangular grid in every configuration. This is a standard procedure in hierarchical constructions and it will allow us to partition every configuration into colonies and use the decoding function $\tilde\Phi$. In Section~\ref{Jarkko}, we will make a slight digression and show how we can simulate any TM with an RPCA in real-time. This is needed in order to preserve the expansiveness of the horizontal direction. Then, in Section~\ref{sec:singlepermutation}, we construct an RPCA to simulate an RPCA whose direction vectors are null (all its fields are still). There are some tricks involved in this phase, mainly having to do with deleting the previous simulated letter and synchronizing the computations. Then, in Section~\ref{sec:singleshift}, we construct an RPCA that can simulate any RPCA whose permutation is the identity \ie any shift. Finally, in Section~\ref{sec:universal}, we construct the universal IPPA $\unive$ that can simulate any RPCA, when it is restricted to the appropriate alphabet.

\section{Imposing a periodic structure}\label{sec:structure}

Let $\C[\coordi]=[\addr,\addr_{+1},\age,\age_{+1}]$. %(recall from Subsection \ref{s:conv} that this is simply a shorthand for a list of distinct natural parameters).

\begin{itemize}
\item $\age$ and $\addr$ are meant to localize the cell in its macrocell, and they correspond to the projections involved in the definition of explicit simulation in Section~\ref{sub:simulconv}.
\item $\age_{+1}$ and $\addr_{+1}$ are used to communicate with the neighbour cells, so that consistency between the $\age$ and $\addr$ fields is achieved.
%\item $\maxage$ and $\maxaddr$ will contain the maximal clock and address in a colony.
\end{itemize}

\begin{algo}{coordi}{\coordi}{v_\maxaddr, v_\maxage}
\STATE{$\chekk[\bina{\pi_{\addr_{+1}}}=\bina{\pi_\addr}$ \AND $\bina{\pi_{\age_{+1}}}=\bina{\pi_\age}]$} \label{al:coordi:chek} \COMMENT{Check left-neighbour information coherence.}
\STATE{$\incr[v_\maxaddr;\addr_{+1}]$}\label{al:coordi:incadd} \COMMENT{Increment $\addr_{+1}$ so that the right neighbour can check coherence.}
\STATE{$\incr[v_\maxage;\age]$}\label{al:coordi:incage} \COMMENT{Update \age.}
\STATE{$\incr[v_\maxage;\age_{+1}]$}\label{al:coordi:incage1} \COMMENT{Update $\age_{+1}$.}
\end{algo}

By the discussion of Chapter~\ref{c:programming}, we know that $\coordi[v_{\maxaddr},v_{\maxage};\C[\coordi]]$ is polynomially computable with respect to its parameters $v_{\maxaddr}$ and $v_{\maxage}$.

In practice, the two valuation parameters $v_{\maxaddr}$ and $v_{\maxage}$ will be constant over the alphabet of the PPA, in which case the behaviour will be described by the following:

\begin{lemma}\label{koo}
Let us fix a field list $\C[\coordi]\in\N^4$ and integers $S,T\in\Ns$.\\
Let $F$ be the IPPA defined by the permutation $\coordi[S,T;\C[\coordi]]$ and directions $\vec{\nu}_{\coordi}$ given by the label indices, and let $\vec k\in\N^*$ be a vector satisfying:
\[\both{
k_\addr,k_{\addr_{+1}}\ge\norm S\\
k_\age,k_{\age_{+1}}\ge\norm T
~.}\]
Let $c \in (\haine5^{\vec{k}})^{\Z}$. Then, $c \in F^{-2}((\haine5^{\vec{k}})^{\Z})$ if and only if there exist $s \in \co{0}{S}$ and $t \in \co{0}{T}$ %$0\le s<S$ and $0\le t<T$ such that for all $n\in\Z$, $\bina{\pi_{\addr}(c_n)}=s+n\bmod S$ and $\bina{\pi_{\age}(c_n)}=t$.
such that $c \in \gra{s}{t}{S}{T}$. In this case, $F(c) \in \gra{s}{t+1 \mod T}{S}{T}$.
%In addition, $\bina{\pi_{\addr}F(c)_n}=s+n\bmod S$ and $\bina{\pi_{\age}F(c)_n}=t+1\bmod T$, for all $n\in\Z$.
\end{lemma}

In the previous statement, $S$ and $T$ should be understood as the \dfn{width} and \dfn{height} of the macrocells. Notice, also, that the statement holds for all vectors $\vec{k} \in \N^*$ that satisfy the inequalities, which means that there can be other fields in the alphabet. This means that if we use $\coordi$ together with other rules that do not change the values of the fields in $\C[\coordi]$, the statement of the lemma will still be true.

The restrictions about the lengths of $\vec{k}$ ensure that fields are large enough that we can write the binary representation of $S$ and $T$ on them. %We will denote these inequalities by $\I(\coordi,\vec{k},S,T)$. This will make some of our statements shorter and more concise.

%We also note that since $v_{\maxaddr}=S$ is constant, $\incr[S,\addr](\vec{u})$, for example, does the following: it first deletes the initial $4$ from $\pi_{\addr}(\vec{u})$. Then, it interprets this as a binary number and adds $1$ modulo $S$ to it if its value is less than $S$. Then, it adds back the initial $4$, so that the length of $\addr$ becomes the same as it was in the beginning. In the proof, we will omit the $\hs{\cdot}$ to ease the notation.

\begin{proof}

We prove the stronger claim that if $F^2(c)$ exists, then there exist $0\le s<S$ and $0\le t<T$ such that for all $n \in \Z$, $\bina{\pi_{\addr}(c_n)}=\bina{\pi_{\addr_{+1}}(c_n)}=s+n \bmod S$ and $\bina{\pi_{\age}(c_n)}=\bina{\pi_{\age_{+1}}(c_n)}=t$.
%First of all, $\checkb(6)$ makes sure that the input letter has exactly $6$ fields.

Suppose that $\bina{\pi_{\addr}(c_n)} \neq \bina{\pi_{\addr_{+1}}(c_n)}$ or $\bina{\pi_{\age}(c_n)} \neq \bina{\pi_{\age_{+1}}(c_n)}$, for some $n \in \Z$. Then, line~\ref{al:coordi:chek} would not be defined at cell $n$, $F(c)$ would not exist, which is a contradiction. 

%v\equiv0\{v'\}

Suppose, then, that there exists $n \in \Z$ with $\bina{\pi_{\addr}(c_{n+1})} \neq \bina{\pi_{\addr}(c_n)} +1 \bmod S$. Line~\ref{al:coordi:incadd} and the fact that $\addr_{+1}$ is a right-going field imply that $\bina{\pi_{\addr_{+1}}F(c)_{n+1}} = \bina{\pi_{\addr}(c_n)} +1 \mod S$. Then, line~\ref{al:coordi:chek} is not defined at cell $n+1$ of $F(c)$ since $\bina{\pi_{\addr_{+1}}F(c)_{n+1}} = \bina{\pi_{\addr}(c_n)} +1 \mod S \neq \bina{\pi_{\addr}(c_{n+1})}$. Therefore $F^2(c)$ does not exist, which contradicts the hypothesis. Similarly, we can prove that $\bina{c_n.\age} = \bina{c_{n+1}.\age}$, for all $n \in \Z$. Thus, the stronger claim we made at the beginning of the proof is true.

If $\bina{\pi_{\addr}(c_0)}=s$ and $\bina{\pi_{\age}(c_0)}=t$, then the previous claim implies that for all $n \in \Z$, $\bina{\pi_{\addr}(c_n)}=s+n \bmod S$ and $\bina{\pi_{\age}(c_n)}=t$. Furthermore, since the value of $\addr$ is not changed by $F$ and the value of $\age$ is increased by $1 \mod T$ every time step by line~\ref{al:coordi:incage}, we have that $\bina{\pi_{\addr}F(c)_n}=s+n \bmod S$ and $\bina{\pi_{\age}F(c)_n}=t+1 \bmod T$, for all $n \in \Z$.

\end{proof}

In general, when using IPPA, we have to use a similar rule every time we want to impose some horizontal restriction on the configuration. Namely, we have to use an additional right-moving (or left-moving, it does not make a difference) field, and then we need $2$ steps in order to verify that the field is constant.

%The subsequent constructions will admit two fields $\addr$ and $\age$, over which its restrictions will behave as described in Lemma~\ref{koo}. This allows us to use the notation $\gr stF$. %and to talk about colonies, as explained in Subsection~\ref{sub:simulconv}.

All of the rules we construct will factor onto $\coordi[S,T;\C[\coordi]]$, for some $S,T \in \Ns$. The following remark will give the disjointness condition in the definition of simulation.

\begin{remark}\label{thirdconditionofsimulation}
Assume that $F\colon \az \pto \az$ factors onto $\coordi[S,T;\C[\coordi]]$ through the factor map $H$ and let $\gr stF \defeq H^{-1}(\gra{s}{t}{S}{T})$. Then, the union $\bigsqcup_{\begin{subarray}c0\le t<T\\0\le s<S\end{subarray}}\gr stF$ is disjoint and $F(\gr{s}{t}{F}) \subseteq \gr{s}{t+1 \mod T}{F}$.

Therefore, if $\Phi \colon \az \pto \bz$ satisfies that $\dom(\Phi) \subseteq \gr{0}{0}{F}$, for some $s,t$, then the union $\bigsqcup_{\begin{subarray}c0\le t<T\\0\le s<S\end{subarray}}F^t\sigma^s(\dom(\Phi))$ is disjoint.
\end{remark}

$\gr{s}{t}{F}$ implicitly depends on the factor map $H$. However, in applications, $H$ will be equal to $\pi_{\C[\coordi]}$ so that no ambiguity arises by omitting it.

\section{Simulating TM with IPPA}\label{Jarkko}

The IPPA $\coordi$ allows us to divide every configuration into colonies with a periodical clock. We want to use this space-time structure in order to do computations within the \dfn{work-periods} (the \xpr{time} between two subsequent steps where the clock is $0$). We are going to introduce the elements needed for this one by one, since, hopefully, it will make some of the ideas more clear. First, let us show how to simulate TMs in real time with PPA.

For all programs $p=p_{\M} \in \haine4^*$, we construct an IPPA that simulates $\M$ in real-time.
This subsection is inspired by \cite{morita}.%by \cite[prop 52]{ca4} (with a compression in the radius), but such a result dates back from \cite{morita}.

Let ${\C}_\U\defeq[\info,\head_{-1},\head_{+1}]$.

The key item to maintaining reversibility, is to keep track of the history of the computation. Some kind of \emph{archive} of each past step is shifted in the direction opposite to the head, in order for the head to always have space to write the new history.
Recall the definition of the function $\U$ from Subsection~\ref{s:turing}. Let $\gamma_\U[p] \colon (\haine4^*)^3 \pto (\haine4^*)^3$ be defined by the following transitions: $(a,h_{-1},h_{+1})$ is mapped to

\begin{itemize}
\item $(a',\Chi{a,q,\delta},\Chi{a,q,\delta})$, if $(h_\delta,h_{-\delta})=(q,\motvide)$ ($\head_{\delta}$ contains the TM head) and $\U(a,q,p)=(a',\motvide,+1)$. If $\head_{\delta}$ contains a head and the transition is an accepting one, then we write an encoding of the last transition on the $\head$ fields, modify $\info$ and the TM heads disappear. Here, the assumption that the TM head (which has disappeared) moves to the right is convenient to ensure injectivity.
\item $(a',h_{-1}',h_{+1}')$, where $(h_{\delta'}',h_{-\delta'}')=(q',\Chi{a,q,\delta})$ if $(h_\delta,h_{-\delta})=(q,\motvide)$ and  $\U(a,q,p)=(a',\motvide,\delta')$. If the transition is not an accepting one, then $\info$ is modified, the TM head is written on the appropriate $\head$ field and on the other $\head$ field we write an encoding of the transition and of the position of the head before the transition.
\item $(a,h_{-1},h_{+1})$ if $h_{-1},h_{+1}\notin Q_p\setminus\{\motvide\}$. If none of the $\head$ fields contains a TM head, then do nothing.
\end{itemize}

It is not difficult (by a tedious case enumeration) to see that $\gamma_\U[p]$ is a partial permutation and \emph{polynomially computable}, thanks in particular to the disjointness of $Q_p\subset\haine2^*$ and $\Chi{\haine4\times Q_p\times\{-1,1\}}\subset2\haine3^*$. Basically, $\gamma_\U[p]$ identifies the accepting state $\motvide$ with the absence of state (for which it just performs identity). In other cases, it prevents from having two (non-accepting) head states at the same cell; then it applies the transition rule and sends the new state to the correct direction (depending on $\delta$), while sending an archive of the last performed operation in the opposite direction.
At the moment that the accepting state appears, it just sends two (identical) archives in opposite directions (there is no new state to send).

We say that $c\in (\haine4^{**})^{\Z}$ \dfn{represents} $(z,q,j)\in \haine4^{\Z}\times Q\times\Z$ of the machine $\M$ corresponding to program $p$
if:
\begin{itemize}
\item For all $i \in \Z$, $\pi_{\info}(c_i)=z_i$;% and $x_i.\prog=p_\M$;
\item $(\pi_{\head_{-1}}(c_j),\pi_{\head_{+1}}(c_j)) \in \{q\} \times \{\motvide\} \cup \{\motvide\} \times \{q\}$;
\item For all $i\ne j$ and $\delta\in\{-1,+1\}$, $\pi_{\head_{\delta}}(c_i)\notin Q_p\setminus\{\motvide\}$
\item For all $i\ne j$ and $\delta\in\{-1,+1\}$, if $\pi_{\head_{\delta}}(c_i)\neq\motvide$ then $\delta$ has the sign of $j-i$.
\end{itemize}
Intuitively, this means that in $c$ there is at most one (non-accepting) head at position $j$, no head elsewhere, and nothing (represented by $\motvide$, like the accepting state) in $\head_{-1}$ on its right nor in $\head_{+1}$ on its left. The possible archives go away from the head position. We can thus see that the head will never move into a cell where there is an archive, so that one of the transitions of $\gamma_{\U}[p]$ will always be applicable.

Formally, we have the following lemma about the behaviour of $\gamma_\U[p]$.

\begin{lemma}\label{l:tmsim}
Let us fix a field list $\C[\U]\in\N^3$ and a program $p=p_{\M} \in \haine4^*$.\\
Consider the IPPA $F$ defined by permutation $\sh{\gamma_\U[p]}$ and directions $\vec{\nu}_{\U}$ given by the label indices.\\
Let $\vec k\in\N^*$ be a vector satisfying:
\[\both{
k_{\head_{-1}},k_{\head_{+1}}\ge \norm{\Chi{\haine4\times Q_p\times\{-1,+1\}}}\\
k_\info \ge 1
~.}\]
Let $c \in (\haine5^{\vec{k}})^{\Z}$ and suppose that $\hs{c}$ represents configuration $(z,q,j)\in \haine4^{\Z}\times Q\times\Z$ of $\M$. Then, for all $t \in \N$, $\hs{F^t(c)}$ represents  $\M^{t}(z,q,j)$.
\end{lemma}

As in Lemma~\ref{koo}, the inequalities about the lengths of $\vec{k}$ simply state that the fields are long enough. Using Lemma~\ref{sharpization}, we will omit the $\sh{\cdot}$ and $\hs{\cdot}$ from $\sh{\gamma_\U[p]}$ and $\hs{c}$ in the following proof, since they are only used to make $F$ have constant lengths.

\begin{proof}
We will prove the claim for $t=1$; the general claim then follows by induction. 

Suppose, first, that $\M(z,q,j)$ does not exist. This means that $\delta_\M(z_j,q)$ does not exist, or equivalently, that $\U(z_j,q,p)=\U(c_j.\info,q,p)$ does not exist. From the definition of $\gamma_\U[p]$ and the fact that $c$ represents $(z,q,j)$, we have that ${\gamma_{\U}[p]}(c_j)$ does not exist, which implies that $F(c)$ does not exist.

Suppose, then, that $\M(z,q,j)=(z',q',j')$ exists. This means that $\U(z_j,q,p)=(z'_j,q',j'-j)$, and $z'_i=z_i$ for any $i\ne j$.
By assumption, for any $i \in \Z$, $(\pi_{\info}(c_i),\pi_{\head_{-1}}(c_i),\pi_{\head_{+1}}(c_i)) = (z_i,h_{-1,i},h_{+1,i})$ for some $h_{-1,i},h_{+1,i} \in Q_p \cup \Chi{\haine4\times Q_p\times\{-1,+1\}} \cup \{\motvide\}$.

Moreover, for any $i\ne j$, and $\delta\in\{-1,+1\}$, $h_{\delta,i}\notin Q_p$, so the identity rule is applied. After applying the shifts, it gives that for any $i<j-1$, 
\begin{equation*}
(\pi_{\info}F(c_i),\pi_{\head_{-1}}F(c_i),\pi_{\head_{+1}}F(c_i))=(z_i,h_{-1,i+1},h_{+1,i-1}) = (z'_i,h_{-1,i+1},\motvide)
\end{equation*}
with $h_{-1,i+1}\notin Q_p$, and for any $i>j+1$, 
\begin{equation*}
(\pi_{\info}F(c_i),\pi_{\head_{-1}}F(c_i),\pi_{\head_{+1}}F(c_i))=(z_i,h_{-1,i+1},h_{+1,i-1})=(z'_i,\motvide,h_{+1,i-1})
\end{equation*}
with $h_{+1,i-1}\notin Q$.

Now, assume $(h_{-1,i},h_{+1,i})=(q,\motvide)$ and that $\U(z_j,q,p)=(z'_j,q',-1)$ (the other cases can be dealt with in a similar way).
Then the transition $(z_j,q,\motvide) \to (z'_j,q',\Chi{z_j,q,-1})$ is applied by $\gamma_\U[p]$. After the application of $\gamma_\U[p]$ and the shifts, we obtain 
\begin{eqnarray*}
(\pi_{\info}F(c_j),\pi_{\head_{-1}}F(c_j),\pi_{\head_{+1}}F_\U[p](c_j)) &=& (z'_j,\motvide,\motvide),\\ 
(\pi_{\info}F(c_{j-1}),\pi_{\head_{-1}}F(c_{j-1}),\pi_{\head_{+1}}F(c_{j-1})) &=& (z'_{j-1},q',\motvide) \text{ and,} \\ 
(\pi_{\info}F(c_{j+1}),\pi_{\head_{-1}}F(c_{j+1}),\pi_{\head_{+1}}F(c_{j+1})) &=& (z'_{j+1},\motvide,\Chi{z_j,q,-1}).
\end{eqnarray*}

All conditions are hence satisfied for $F(c)$ to represent $\M(z,q,j)$.
\end{proof}

Note that due to the parallel nature of IPPA, some configurations may involve several machine heads, and valid simulations may take place in parallel, provided that there is enough space between them so that the archive and the heads do not collide.
For this reason, we need to give a \xpr{finite version} of the previous lemma.

\begin{lemma}\label{Turing}
Let us fix a field list $\C[\U] \in\N^3$ and a program $p=p_{\M} \in \haine4^*$.\\
Consider the IPPA $F$ defined by permutation $\sh{\gamma_\U[p]}$ and directions $\vec{\nu}_{\U}$ given by the label indices.
%\begin{comment}
Let $\vec k\in\N^*$ be a vector satisfying:

\begin{align*}
&k_{\head_{-1}},k_{\head_{+1}}\ge \norm{\Chi{\haine4\times Q_p\times\{-1,+1\}}}\\
&k_\info \ge 1.
\end{align*}

Let $c \in (\haine5^{\vec{k}})^{\Z}$, $c'=\hs{c}$ and assume that there exists $n \in \N$ such that the set 
\begin{equation*}
J\defeq\set{j}{\Z}{{(\pi_{\head_{-1}}(c'_j),\pi_{\head_{+1}}(c'_j))}\ne(\motvide,\motvide)}
\end{equation*}
satisfies that for any $j\neq j'\in J$, we have $\abs{j'-j}>2n$, and that for all $j \in J$, $(c'_j.\head_{-1},c'_j.\head_{+1}) \in Q_p \times \{\motvide\} \cup \{\motvide\} \times Q_p$. For $j \in J$, let $q_j \defeq c'_j.\head_{-1}$ if $(c'_j.\head_{-1},c'_j.\head_{+1}) \in Q_p \times \{\motvide\}$ and $q_j \defeq c'_j.\head_{+1}$ if $(c'_j.\head_{-1},c'_j.\head_{+1}) \in \{\motvide\} \times Q_p$.

Then, $F^n(c)$ exists if and only if $(z^j,q'_j,j')\defeq\M^n(c'.\info,q_j,j)$ exists, for all $j\in J$.
In addition:
\begin{itemize}
\item $\pi_{\info}F^n(c)_{\cc{j-n}{j+n}}=z^j_{\cc{j-n}{j+n}}$, for all $j \in J$;
\item $\pi_{\info}F^n(c)_i=c_i$ if $i\notin J+\cc{-n}n$;
\item $(\pi_{\head_{-1}}F^n(c)_{j'},\pi_{\head_{+1}}F^n(c)_{j'}) \in \{(q'_j,\motvide)\} \cup \{(\motvide,q'_j\}$, for all $j \in J$;
\item $(\pi_{\head_{-1}}F^n(c)_i,\pi_{\head_{+1}}F^n(c)_i) \notin Q_p\times \{\motvide\} \cup \{\motvide\} \times Q_p$, if $i \notin \sett{j'}{j \in J}$
\end{itemize}
\end{lemma}

\begin{proof}
First note that the identity is always applied when the head is absent; in particular it is applied outside $J+\cc{-t}t$ at time $t\in\N$ (because $F$ has radius $1$) and initially the heads are only in the positions in $J$.

According to the assumptions, for all $j\in J$, %the configuration 
${c'}^j$ is obtained by turning all $(c_i.\head_{-1},c_i.\head_{+1})$ to $(\motvide,\motvide)$ except at position $j$ represents $(c'.\info,q_j,j)$. Thanks to Lemma \ref{l:tmsim}, for all $0 \le t \le n$, $F^t({c'}^j)$ exists if and only if $\M^n(c'.\info,q_j,j)$ exists.

In that case, since ${c'}^j$ coincides with $c'$ over interval $\cc{j-2n}{j+2n}$ and since the radius is $1$, a simple induction can show that $F^t({c'}^j)$ coincides with $F^t(c')$ over interval $\cc{j-2n+t}{j+2n-t}$.
Lemma~\ref{l:tmsim} hence gives the main claim.

Conversely, suppose that $F^t({c'}^j)$ is undefined for some $j\in J$ with $t\le n$ minimal. Then, $F^{t-1}({c'}^j)$ exists, and by Lemma \ref{l:tmsim} involves a unique (non-accepting) head, in some cell $j'\in\cc{j-t}{j+t}$.
Therefore, $\gamma_\U[p](F^{t-1}({c'}^j)_i$ is defined for any $i\ne j'$.
This means that $\gamma_\U[p](F^{t-1}({c'}^j)_{j'})$ is undefined; we have already noted that this is equal to $\gamma_\U[p](F^{t-1}(c')_{j'})$, which proves that $F^t(c')$ is undefined.
\end{proof}

Lemma~\ref{Turing} will be used in the following way: Every configuration will be divided into colonies by $\coordi$. Initially (when the clock is equal to $0$), inside every colony there will be exactly one TM head at the leftmost cell of the colony. These TM will perform some computation for a small amount of time compared to the the width of the colonies (the $S$ of Lemma~\ref{koo}) so that the heads will not meet. Lemma~\ref{Turing} will immediately imply that at the end of the computation, in every colony, $\info$ contains the output of the computation. Finally, the output of the computation will be copied onto some new field and then the computation will be run backwards (remember that $\gamma_{\U}[p]$ is a permutation).

We are now ready to give the details of the definition of the permutation $\alpha_\U[t;\info,\head_{-1},\head_{+1}]$: Let $\vec{u} \in \haine5^{**}$ and define $\vec{u'}$ in the following way:
\begin{equation*}
(u'.\info,u'.\head_{-1},u'.\head_{+1})=\sh{\gamma_{\U}[t(\vec{u})]}(u.\info,u.\head_{-1},u.\head_{+1}),
\end{equation*}
and $u'_j\defeq u_j$ for all $j\notin\{\info,\head_{-1},\head_{+1}\}$.
Then, 
\begin{equation*}
\alpha_\U[t;\info,\head_{-1},\head_{+1}](\vec u)\defeq\vec{u'},
\end{equation*}
if $t(\vec{u})=t(\vec{u'})$ (and it is undefined otherwise.).

Again, the definition gets a little more complicated due to the need to preserve the lengths, to have arbitrarily many fields and to ensure reversibility. When $t=\pi_{\prog}$, where $\prog$ is a new field, then the condition $t(\vec{u})=t(\vec{u'})$ is always satisfied.

\section{Computing the simulated permutation}\label{sec:singlepermutation}
Let $\C[\compute]= \C[\U]\sqcup[\newinfo]$.

\begin{itemize}
\item $\head_{-1},\head_{+1}$ are used by to simulate a TM with the rule of Subsection~\ref{Jarkko}.
\item The output of this computation is written on $\newinfo$ and then the computation is reversed (the Bennett trick, see~\cite{bennett}).
\end{itemize}

\begin{algo}{compute}{\compute}{v_\addr,v_\age,v_\schedule,t_\prog,t_\revprog}
\IF{$v_\age=0$ \AND $v_{\addr}=0$} 
%\STATE{$\chekk[\hs\newinfo=\motvide]$} \COMMENT{The output tape is initially empty.}\label{al:com:empstart}
%\STATE{$\chekk[\hs{\head_{-1}}=\hs{\head_{+1}}=\motvide]$} \label{al:com:empend} \COMMENT{All head fields are initially empty.}
%\IF{$v_\addr=0$}
\STATE $\rite[0;\head_{-1}]$ \label{al:com:wrhead}
\COMMENT{Write the machine initial state in the left head field.}
\ENDIF
%\ENDIF
\IF{$0\le v_\age<v_\schedule$}
\STATE $\sh{\gamma_{\U}[t_\prog;\info,\head_{-1},\head_{+1}]}$ \COMMENT{Run the machine in order to compute the permutation.} \label{al:com:posforcom}
\ELSIF{$v_\age=v_\schedule$}
\STATE{$\chekk[\pi_{\head_{-1}},\pi_{\head_{+1}}\notin Q_{t_\prog}\setminus\{\motvide\}]$} \COMMENT{Check that the computation halted.} \label{al:com:poscheckhalt}
\STATE{$\rite[\info;\newinfo]$} \COMMENT{Copy the output onto a different tape.} \label{al:com:posbennet}
\STATE{$\exch[\head_{-1},\head_{+1}]$} \COMMENT{The directions of the fields of $F_{\U}[p]^{-1}$ are opposite to those of $F_{\U}[p]$.} \label{al:com:posexch}
\ELSIF{$v_\schedule<v_\age\le2v_\schedule$}
\STATE{$\sh{\gamma_{\U}[t_\prog;\info,\head_{+1},\head_{-1}]^{-1}}$} \COMMENT{Unwind the computation in order to delete the archive.}\label{al:com:posbaccom}
\ENDIF
\IF{$2v_\schedule\le v_\age<3v_\schedule$}
\STATE $\sh{\gamma_{\U}[t_\revprog;\newinfo,\head_{-1},\head_{+1}]}$ \COMMENT{Compute the inverse of the permutation, in order to recover \info.} \label{al:com:negforcom}
\ELSIF{$v_\age=3v_\schedule$}
\STATE{$\chekk[\pi_{\head_{-1}},\pi_{\head_{+1}}\notin Q_{t_\revprog}\setminus\{\motvide\}]$} \COMMENT{Check that the computation halted.} \label{al:com:negcheckhalt}
\STATE{$\rite[\info;\newinfo]^{-1}$} \label{al:com:negbennet}\COMMENT{Empty $\newinfo$.}
\STATE{$\exch[\head_{-1},\head_{+1}]$} \label{al:com:negexch}\COMMENT{Reverse the directions again.}
\ELSIF{$3v_\schedule<v_\age\le4v_\schedule$}
\STATE{$\sh{\gamma_{\U}[t_\revprog;\info,\head_{+1},\head_{-1}]^{-1}}$} \label{al:com:negbaccom}\COMMENT{Unwind the second computation, too.}
\ENDIF
\IF{$v_\age=4v_\schedule$ \AND $v_\addr=0$}
\STATE $\rite[0;\head_{-1}]^{-1}$\label{al:com:delhead}
\COMMENT{Erase the machine initial state.}
\ENDIF
\end{algo}

$\compute[v_\addr,v_\age,v_\schedule,v_\prog,v_\revprog;\C[\compute]]$ is polynomially computable with respect to its parameters.

Note that, depending on the value of $v_\addr$, only a small amount of these permutation are applied.

In applications, the three parameters $v_\schedule,t_\prog,t_\revprog$ will be constant. $v_\schedule$ contains a natural number that controls how long the computation lasts. $t_\prog$ and $t_\revprog$ are interpreted as the program and the reverse program (\ie the program of the inverse IPPA) of the IPPA that we want to simulate. 

We are now able to simulate uniformly RPCA with \xpr{radius $0$} (null direction vector).
\begin{lemma}\label{behavior}
Let us fix a field list $\C[\coordi] \sqcup \C[\compute] \in\N^8$, vectors $\vec{k}, \vec{k'}\in\N^*$, integers $S, T \in \Ns, t_0, U\in\N$ and programs $p,p^{-1}\in\haine4^*$ of a partial permutation $\alpha:\haine5^{**}\pto\haine5^{**}$ and its inverse $\alpha^{-1}$, respectively and let $G$ the IPPA corresponding to permutation $\alpha$ and null direction vector. 

Consider the IPPA $F$ with directions $\vec{\nu}_{\coordi \cup \compute}$, and permutation 
\begin{equation*}
\coordi[S,T]\circ\compute[\bina{\pi_\addr},\bina{\pi_\age}-t_0,U,p,p^{-1}]~,
\end{equation*} 

and assume that the following inequalities hold:
\[\both{
U\ge\max\{t_p({\haine5^{\vec{k'}}}),t_{p^{-1}}({\haine5^{\vec{k'}}})\}\\
S\ge\max\{2t_0,\norm{\Chi{\haine5^{\vec{k'}}}}\}\\
T\ge4U+t_0\\
k_\addr,k_{\addr_{+1}}\ge\norm S\\
k_\age,k_{\age_{+1}}\ge\norm T\\
k_{\head_{-1}},k_{\head_{+1}}\ge\length{\Chi{\haine4\times (Q_p \cup Q_{p^{-1}})\times\{-1,+1\}}}\\
k_\info,k_\newinfo\ge1.
}\]

Then, $F\restr{(\haine5^{\vec k})^\Z}\simu[S,T,0,\Phi]G\restr{(\haine5^{\vec k'})^\Z}$, where $\Phi\defeq\tilde{\Phi}_{\info}{\restr\Sigma}$ and
%\begin{equation*}
$\Sigma\defeq(\haine5^{\vec{k}})^{\Z}\cap\gra{0}{0}{S}{T}\cap \emp[\motvide]{\newinfo,\head_{-1},\head_{+1}}\cap\tilde{\Phi}_{\info}^{-1}((\haine5^{\vec{k'}})^\Z).$
%\end{equation*}
\end{lemma}

The number $t_0$ should be understood as a delay before which to apply this rule, and $U$ as the maximal time that we allow to the (forward and backward) computation.

\begin{proof}
Remarks~\ref{twoconditionsofsimulation} and \ref{thirdconditionofsimulation} imply that $\Phi$ is surjective, $\Phi \sigma^S = \sigma \Phi$ and that $\dom(\Phi)=\bigsqcup_{\begin{subarray}c0\le t<T\\0\le s<S\end{subarray}}F^t\sigma^s(\dom(\Phi))$ is a disjoint union. 

Therefore, in order to prove the simulation we only have to prove that $G \Phi = \Phi F^T$. This is equivalent (since we are talking about partial functions) to the facts that if $\Phi(c)=b$, then $F^T(c)$ exists if and only if $G(b)$ exists, and in that case $\Phi F^T(c)=G(b)$. 

We are actually going to prove the following stronger: 
\begin{fact}\label{fact:permutation}
If $c \in (\haine5^{\vec{k}})^{\Z}\cap\gra{0}{t_0}{S}{T}\cap\emp[\motvide]{\newinfo,\head_{-1},\head_{+1}}\cap\tilde{\Phi}_{\info}^{-1}(b)$, then $F^{4U}(c)$ exists if and only if $G(b)$ exists, and in that case $F^{4U}(c) \in \gra{0}{t_0+4U}{S}{T}\cap\emp[\motvide]{\newinfo,\head_{-1},\head_{+1}}\cap\tilde{\Phi}_{\info}^{-1}(G(b))$. 
\end{fact}
Since the only rule applied outside $t_0 \le \age \le t_0 + 4U$ is $\coordi$, Fact~\ref{fact:permutation} implies that if $\Phi(c)=b$, then $\Phi F^T(c)=G(b)$, which concludes the proof of the lemma.

For the rest of the proof, let $c \in (\haine5^{\vec{k}})^{\Z}\cap\gra{0}{t_0}{S}{T}\cap\emp[\motvide]{\newinfo,\head_{-1},\head_{+1}}\cap\tilde{\Phi}_{\info}^{-1}(b)$.

Suppose, first, that $F^{4U}(c)$ exists.
\begin{itemize}
\item $\age=t_0$: Initially, $c \in \emp[\motvide]{\newinfo,\head_{-1},\head_{+1}}$. Line~\ref{al:com:wrhead} writes the initial head of the TM on $\head_{-1}$ of the leftmost cell of every colony.
\item $t_0 \le \age < t_0+U$: Only the permutation of line~\ref{al:com:posforcom} is applied. Together with the directions of $\vec{\nu}_{\compute}$, this implies that we apply $U$ steps of the IPPA $F_{\U}[p]$ on configuration \\
$(\pi_{\info}(c),\pi_{\head_{-1}}(c),\pi_{\head_{+1}}(c))$. This configuration has a starting TM head at the leftmost cell of every colony, and since $c \in \Phi^{-1}(b)$, we have that in the $i$'th colony the input of the TM is $\Chi{b_i}$.

\item $t_0=t_0+U$: Line~\ref{al:com:poscheckhalt} checks that in no place of the tape does there appear a head of the TM defined by $p$. This means that all the TM have accepted within the first $U$  steps and since $p$ is the program of $\alpha$, we take that $\alpha(b_i)$ exists, for all $i \in \Z$, or equivalently that $G(b)$ exists. %(Recall that $c \in \Phi^{-1}(b)$, which means that for all $n$, the input to the TM in the $n$'th colony is $\Chi{b_n}$.)
\end{itemize}

Therefore, if $F^{4U}(c)$ exists, then $G(b)$ exists.%$c \in \emp{\newinfo,\head_{-1},\head_{+1}} \text{ and } \alpha(b) \text{ exists }$.

For the other direction, suppose that $G(b)$ exists, or, equivalently, that $\alpha(b_i)$ exists, for all $i \in \N$. %$c \in \emp{\newinfo,\head_{-1},\head_{+1}}$ and that $\alpha(b)$ exists, or, equivalently, that $\alpha(b_n)$ exists, for all $n \in \N$.

\begin{itemize}
\item $\age=t_0$: By assumption, $c \in \emp[\motvide]{\newinfo,\head_{-1},\head_{+1}}$.%therefore $c$ is not rejected by the checks of lines~\ref{al:com:empstart}-\ref{al:com:empend}. 
Line~\ref{al:com:wrhead} writes the initial state of the TM on the $\head_{-1}$ of the leftmost cell of every colony.

\item $t_0 \le \age < t_0+U$: Only the permutation of line~\ref{al:com:posforcom} is applied. Together with the directions of $\vec{\nu}_{\compute}$, this implies that at each step, we apply the IPPA $F_{\U}[p]$ on the configuration \\
$(\pi_{\info}(c),\pi_{\head_{-1}}(c),\pi_{\head_{+1}}(c))$. There is a TM head at the leftmost position of every colony and the input in the $i$'th colony is equal to $\Chi{b_i}$. In other words, $\tilde{\phi}_{\info}(\col{i}{c})=b_i$, for all $i \in \Z$.

\item $\age=t_0+U$ Since $\alpha(b_i)$ is defined for all $i \in \Z$, $U \geq t_p(\haine5^{\vec{k'}})$ and $S \geq 2U$ , we can see that the conditions of Corollary~\ref{Turing} are satisfied with $n=U$. This means that the computation of the TM in every colony has accepted and that the output of the computation is written on the $\info$ of every colony. In other words, $\tilde{\phi}_{\info}(\col{i}{F^{U}(c)})=\alpha(b_i)$, for all $i \in \Z$.

The check of line~\ref{al:com:poscheckhalt} is true, since by assumption all of the TM have accepted before $\age=U$, and when a TM halts its head disappears. Line~\ref{al:com:posbennet} copies the contents of $\info$ on $\newinfo$.  Therefore, after the application of line~\ref{al:com:posbennet}, we have that $\tilde{\phi}_{\newinfo}(\col{i}{F^{U}(c)})=\alpha(b_i)$, for all $i \in \Z$.

Finally, line~\ref{al:com:posexch} swaps the fields $\head_{-1}$ and $\head_{+1}$. This can be thought of as \xpr{reversing} the directions of these fields. We do this because we want to reverse the computation done by $F_{\U}[p]$ and in order to achieve this, it is not enough to apply $\sh{\gamma_{\U}[p]^{-1}}$, but we also need to use directions $-\vec{\nu}_{\compute}$.

\item $t_0+U+1 \le \age \le t_0+2U$: Only the permutation of line~\ref{al:com:posbaccom} is applied. Together with the fact that the shift directions have been reversed and that the fields $\head_{-1}$, $\head_{+1}$ have also been exchanged in the rules of lines \ref{al:com:posforcom} and \ref{al:com:posbaccom}, this implies that the IPPA $(F_{\U}[p])^{-1}$ is applied for $U$ time steps on the configuration $(F_{\U}[p])^U(\pi_{\info}(c),\pi_{\head_{-1}}(c),\pi_{\head_{+1}}(c))$. 

Therefore, 
%\begin{equ*}
$(\pi_{\info}F^{2U}(c),\pi_{\head_{-1}}F^{2U}(c),\pi_{\head_{+1}}F^{2R}(c))=(\pi_{\info}(c),\pi_{\head_{-1}}(c),\pi_{\head_{+1}}(c)).$
%\end{multline*} 

In other words, the computation has been \xpr{run backwards} until the beginning, but the output of the computation is on $\newinfo$. This is the trick used by Bennet in \cite{bennett} to simulate arbitrary TM with reversible ones.

At this point, $F^{2U}(c) \in \emp[\motvide]{\head_{+1}}$, $\head_{-1}$ is empty except at the left-most cell of every colony, where it contains the initial state $0$, and finally, $\phi_{\info}(\col{i}{F^{2U}(c)})=b_i$ and $\phi_{\newinfo}(\col{i}{F^{2U}(c)})=\alpha(b_i)$, for all $i \in \Z$.

\item $t_0+2U \le \age < t_0+3U$: Only the permutation of line~\ref{al:com:negforcom} is applied. Together with the directions of $\vec{\nu}_{\compute}$, this implies that at each step, we apply the IPPA $F_{\U}[p^{-1}]$ on the configuration \\
$(\pi_{\newinfo}(c),\pi_{\head_{-1}}(c),\pi_{\head_{+1}}(c))$. Notice that we use $\newinfo$ as the TM tape and we use the program $p^{-1}=t_{\revprog}$.

\item $\age= t_0 +3U$: Since $\alpha^{-1}(\alpha(b_i))$ is defined for all $i \in \Z$, $U > t_{p^{-1}}(\haine5^{\vec{k'}})$ and $S \geq 2U$ , the conditions of Corollary~\ref{Turing} are satisfied with $n=U$. This implies that 
$\phi_{\newinfo}(\col{i}{F^{3U}(c)})=b_i$, for all $i \in \Z$.

The check of line~\ref{al:com:negcheckhalt} is true, since all of the TM have accepted before $\age=3U$. Line~\ref{al:com:negbennet} copies the contents of $\info$ on $\newinfo$. Since, at this point these fields are equal in every cell, this is equivalent to emptying fields $\newinfo$ (in a reversible way, though). Therefore, after applying this permutation, $F^{3U}(c) \in \emp[\motvide]{\newinfo}$.

We still have to empty the $\head$ fields, too. For this, we have to run the computation backwards. Line~\ref{al:com:negexch} swaps the fields $\head_{-1}$ and $\head_{+1}$, \xpr{reversing} the directions of these fields.

\item $t_0+3U+1 \le \age \le t_0+4U$: Only the permutation of line~\ref{al:com:negbaccom} is applied. Together with the fact that the shift directions have been reversed and that the head fields inside the rules are also exchanged, this implies that the IPPA $F_{\U}[p^{-1}]^{-1}$ is applied for $U$ time steps on the configuration $F_{\U}[p^{-1}]^{3U}(\pi_{\info}(c),\pi_{\head_{-1}}(c),\pi_{\head_{+1}}(c))$.

Therefore, 
%\begin{multline*}
$(\pi_{\info}F^{4U}(c),\pi_{\head_{-1}}F^{4U}(c),\pi_{\head_{+1}}F^{4U}(c))=(\pi_{\info}F^{2U}(c),\pi_{\head_{-1}}F^{2U}(c),\pi_{\head_{+1}}F^{2U}(c)). $
%\end{multline*}

Notice that now we are using $\info$ as the tape of the TM, while during the forward computation we used $\newinfo$. This is not a problem, though, because the two fields were equal at the end of the forward computation at step $3U$.

At this point, we have that $\phi_{\info}(\col{i}{F^{4U}(c)})=\alpha(b_i)$, for all $i \in \Z$. Also, in the $\head$ fields, there exists the initial state $0$ of the TM on $\head_{-1}$ of the leftmost cell of every colony, while the rest of them are empty.

\item Finally, line~\ref{al:com:delhead} deletes the initial state, and we get that $F^{4U}(c) \in \emp[\motvide]{\newinfo,\head_{-1},\head_{+1}}$.
\end{itemize}

Therefore, we have proved that if $G(b)$ exists, then $F^{4U}(c)$ exists and $F^{4U}(c) \in (\haine5^{\vec k})^\Z \cap \gra{0}{t_0+4U}{S}{T}\cap\emp[\motvide]{\head_{-1},\head_{+1},\newinfo} \cap \tilde{\Phi}_{\info}^{-1}(\alpha(b))$, which finishes the proof of the lemma.
\end{proof}

In a nutshell, this is how the construction works. First, use the program $p$ to compute $\alpha$. At the end of this phase, $\info$ contains $\alpha(b)$ (in the colonies). Copy $\alpha(b)$ onto $\newinfo$ and in the second phase, run the computation backwards so as to erase all auxiliary information written by the TM during the computation. At the end of the second phase, $\info$ contains $b$ and $\newinfo$ contains $\alpha(b)$. In the third and fourth phases of the construction, perform the reverse of what was done in the first two phases, while exchanging the roles of $\newinfo$ and $\info$. First, use $p^{-1}$ with tape field $\newinfo$ so as to compute $\alpha^{-1}(\alpha(b))=b$, then copy $\info$ onto $\newinfo$ (thus emptying $\newinfo$) and then perform the computation backwards. At the end, $\newinfo$ is again empty and $\info$ contains $\alpha(b)$ and everything was done in a reversible way.

Notice for all $b \in (\haine5^{\vec{k'}})$ and all $c \in \Phi^{-1}(b)$, the values of the fields in $\C[\coordi] \sqcup \C[\compute]$ of $c$ are uniquely determined. This implies that if there are no anonymous fields, or if the values of the anonymous fields were determined by the fields of $\C[\coordi] \sqcup \C[\compute]$, then the simulation is also exact.

\section{Shifting}\label{sec:singleshift}
Let $\C[\shift]\defeq[\info,\mail_{-1},\mail_{+1}]$.

$\mail_{+1}$ and $\mail_{-1}$ are used to exchange the information of $\info$ between colonies.

In the following algorithm, $M \in \Ns$ has to be thought of as the number of fields in the simulated alphabet, $\vec{\nu} \in \{-1,0,+1\}^M$ as the vector of directions of the simulated IPPA, and $\vec{k'} \colon \haine5^{**} \pto  \N^M$ is a vector valuation that gives the lengths of the alphabet of the simulated IPPA.
%valuation, that is a map from $\N$ into $\N^M$%, that can also be understood as a family of $M$ valuations, for which the values depend only on the field lengths of the alphabet, not the letter itself:
%for all vector $\vec{k}\in\N^*$, there exists a single $\vec{k'}(\vec k)\in\N^M$ such that $\vec{k'}(\haine5^{\vec k})=\{\vec{k'}(\vec{k})\}$.
$\vec{k'}$ represents the field lengths of the simulated letters, whose information is then \xpr{known} to all the letters of the simulating PPCA.

\begin{algo}{shift}{\shift}{M,\vec{\nu},\vec{k'},v_\addr,v_\age,v_\maxaddr}
\IF{$v_\age=0$ \OR $v_\age=v_\maxaddr$}
\FOR{$0\le i<M$}
	\IF{$l_{\vec{k'},i}\le v_\addr<l_{\vec{k'},i+1}$} \label{al:shi:encoding}
		\STATE{$\exch[\info,\mail_{\vec{\nu}_i}]$} \label{al:shi:movetape} \COMMENT{Letters are moved to the corresponding moving fields, and back after $v_\maxaddr$ steps.}
	\ENDIF
\ENDFOR
\ENDIF
\end{algo}

This is polynomially computable in the parameters.

\begin{lemma}\label{parshift}
Let us fix a field list $\C[\coordi]\sqcup\C[\shift]\in\N^7$, an integer $M\in\Ns$, a direction vector $\vec{\nu} \in \{-1,0,+1\}^M$, a vector $\vec{k'} \in \N^M$ , a vector $\vec{k} \in \N^{*}$ and integers $S,T\in \Ns$, $t_0\in\N$.

Consider the IPPA $F$ defined by directions $\vec{\nu}_{\coordi \sqcup \compute}$, given by the label indices, and permutation 
\begin{equation*}
\coordi[S,T] \circ \shift[M,\vec{\nu}_{\coordi \sqcup \compute},\vec{k'},\bina{\pi_\addr},\bina{\pi_\age}-t_0,S],
\end{equation*}

and assume that the following inequalities hold:

\[\both{
S\geq\norm{\Chi{\haine5^{\vec{k'}}}}\\
T\geq t_0+S\\
k_\addr,k_{\addr_{+1}}\geq\norm S\\
k_\age,k_{\age_{+1}}\geq\norm T\\
k_\info,k_{\mail_{-1}},k_{\mail_{+1}}\geq 1
~.}\]

Then $F\restr{(\haine5^{\vec k})^\Z}\simu[S,T,0,\Phi]\sigma^{-\vec\nu}\restr{(\haine5^{\vec{k'}})^{\Z}}$, where $\Phi\defeq\tilde{\Phi}\restr\Sigma$ and $\Sigma\defeq(\haine5^{\vec{k}})^{\Z}\cap\gra{0}{0}{S}{T}\cap \emp[\motvide]{\mail_{-1},\mail_{+1}}\cap\tilde{\Phi}^{-1}((\haine5^{\vec{k'}})^\Z)$.
\end{lemma}

Recall that, by convention, the directions of the fields are the opposite of the shift that is actually applied.

\begin{proof}
Again, by the definition of $\Phi$ and Remarks~\ref{twoconditionsofsimulation} and~\ref{thirdconditionofsimulation}, we know that $\Phi$ is surjective, $\Phi \sigma^S = \sigma \Phi$ and that $\dom(\Phi)\defeq\bigsqcup_{\begin{subarray}c0\le t<T\\0\le s<S\end{subarray}}F^t\sigma^s(\dom(\Phi))$ is a disjoint union.

Therefore, we only have to show that $\sigma^{-\vec{\nu}} \Phi = \Phi F^T$, which is equivalent to showing, since $\sigma^{-\vec{\nu}}(b)$ is defined for all $b \in (\haine5^{\vec{k'}})^{\Z}$, that if $\Phi(c)=b$, then $\Phi F^T(c)= \sigma^{-\vec{\nu}}(b)$.

As in the proof of Lemma~\ref{behavior}, we are going to prove the following stronger
\begin{fact}\label{fact:shift}
If $c \in (\haine5^{\vec{k}})^{\Z}\cap\gra{0}{t_0}{S}{T}\cap\emp[\motvide]{\mail_{-1},\mail_{+1}}\cap\tilde{\Phi}_{\info}^{-1}(b)$, then $F^S(c) \in \gra{0}{t_0+S}{S}{T} \cap \emp[\motvide]{\mail_{-1}=\mail_{+1}} \cap \tilde{\Phi}^{-1}(\sigma^{-\vec{\nu}}(b))$.
\end{fact}

\begin{itemize}
\item $\age=t_0$: By assumption, $c \in \emp[\motvide]{\mail_{-1},\mail_{+1}}$. 

Lines~\ref{al:shi:encoding} and \ref{al:shi:movetape} copy the encodings of the fields of $b$ on the correct $\mail$, in the following sense:

Since $c \in \Phi^{-1}(b)$, $\pi_{\info}(\col{i}{c})\restr{\co{l_{\vec{k'},j}}{l_{\vec{k'},j+1}}}=\Chi{\pi_j(b_i)}$, for all $i \in \Z$ and $0 \le j < M$. Line~\ref{al:shi:movetape} moves the letter in $\info$ to $\mail_{\vec{\nu}_j}$ when $l_{\vec{v},j}\le \addr<l_{\vec{v},j+1}$, or in other words, moves the encoding of the $j$'th field of $b_j$ onto $\mail_{+1}$ if $j$ is a right-moving field and to $\mail_{-1}$ if $j$ is a left-moving field.

This means that after the application of line \ref{al:shi:movetape}, we have that \\
$\pi_{\mail_{\vec{\nu}_j}}(\col{i}{c})\restr{\co{l_{\vec{k'},j}}{l_{\vec{k'},j+1} }}=\Chi{\pi_j(b_i)}$, while $\pi_{\mail_{-\vec{\nu}_j}}(\col{i}{c})\restr{\co{l_{\vec{k'},j}}{l_{\vec{k'},j+1}}}=\pi_{\info}(\col{i}{c})\restr{\co{l_{\vec{k'},j}}{l_{\vec{k'},j+1}}}=\motvide$, for all $i \in \Z$ and $0 \le j < M$.

\item $t_0 \le \age < t_0+S$: No permutation is applied during these time steps. Only the $\mail$ fields are shifted to the corresponding direction.

\item $t_0 = t_0+S$: Every symbol that was part of the encoding of the $j$'th field of $b$ has travelled $S$ steps to the direction indicated by $\vec{\nu}_j$. This means that before the application of line~\ref{al:shi:movetape}, we have that $\pi_{\mail_{\vec{\nu}_i}}(\col{i}{F^S(c)})\restr{\co{l_{\vec{k'},j}}{l_{\vec{k'},j+1}}}=\Chi{\pi_i(b_{i-\vec{\nu}_i})}$, while \\
$\pi_{\mail_{-\vec{\nu}_j}}(\col{n}{F^S(c)})\restr{\co{l_{\vec{k'},j}}{l_{\vec{k'},j+1}}}=\pi_{\info}(\col{i}{F^S(c)})\restr{\co{l_{\vec{k'},j}}{l_{\vec{k'},j+1}}}=\motvide$, for all $i \in \Z$ and $0 \le j < M$.

Line~\ref{al:shi:movetape} moves the letter from the $\mail$ fields back to $\info$. Therefore, after the application of line~\ref{al:shi:movetape}, we have that $F^S(c) \in \emp[\motvide]{\mail_{-1},\mail_{+1}}$ and $\Phi(F^S(c))=\sigma^{-\vec{\nu}}(b)$, which concludes the proof of the Lemma.
\end{itemize}
\end{proof}

In this proof, it is of great importance that all the letters of $b$ have the same lengths, because this implies that the $j$'th field of every letter of $b$ is encoded at the same positions inside every colony. In fact, the reason that we deal only with alphabets of constant lengths is that this shifting procedure can work so easily.

\section{Simulating any fixed rule}\label{sec:universal}

In this subsection, we will use Lemma~\ref{behavior} and \ref{parshift} to construct an IPPA that can simulate non-trivially any PCA, when restricted to an appropriate finite subalphabet.

Let $\C[\unive]\defeq \C[\compute] \cup \C[\shift] \in \N^{6}$ ($\C[\compute]$ and $\C[\shift]$ share the field $\info$).

%\begin{algo}{shift}{\shift}{}{M,\vec{\nu},\vec{k'},v_\addr,v_\age,v_\maxaddr}{\C[\shift]}

\begin{algo}{unive}{\unive}{M,\vec{\nu},\vec{k'},v_\addr,v_\age,v_\maxaddr,v_\schedule,t_\prog,t_\revprog}
\STATE{$\compute[v_\addr,v_\age,v_\schedule,t_\prog,t_\revprog]$} \label{al:uni:comp}
\STATE{$\shift[M,\vec\nu,v_\addr,v_\age-4v_\schedule,v_\maxaddr,\vec{k'}]$} \label{al:uni:shift}
\end{algo}

This is easily seen to be polynomially computable in the parameters.

Notice that $\shift$ and $\compute$ are used at \xpr{different time steps}, \ie at different values of $v_\age$. $\compute$ starts being used when $v_\age=0$ and, by definition of $v_\compute$, is equal to the identity when $v_\age \geq 4v_\schedule$, while $\shift$ starts being used when $v_\age=4v_\schedule$ (it has a delay of $4v_\schedule$). Formally, this means that for every value of $v_\age$, at most one of the rules $\compute$ and $\shift$ is not equal to the identity.
%\end{remark}

The following lemma is the fruit of all our efforts until now. It provides an IPPA that can simulate any PCA when restricted to a sufficiently large alphabet.

\begin{lemma}\label{universal}
Let us fix a field list $\C[\unive] \sqcup \C[\coordi] \in\N^{10}$, an integer $M\in\Ns$, programs $p,p^{-1}\in\haine2^*$ of a partial permutation $\alpha:\haine5^{**}\pto\haine5^{**}$ %of PPCA $G\defeq\sigma^{-\vec\nu}\alpha$, 
and its inverse $\alpha^{-1}$, respectively, a direction vector $\vec{\nu}\in \{-1,0,+1\}^M$,  a vector $\vec{k'} \in \N^M$, a vector $\vec{k} \in \N^{*}$ and integers $S,T,t_0,U\in\N$.

Let $G=\sigma^{-\nu}\alpha$ and $F$ be the IPPA defined by directions $\vec{\nu}_{\coordi \sqcup \unive}$, given by the label indices, and permutation

\begin{equation*}
\coordi[S,T]\circ \unive[M,\vec{\nu}_{\unive},\vec{k'},\bina{\pi_\addr},\bina{\pi_\age}-t_0,S,T,U,p,p^{-1}]~,
\end{equation*}

and assume that the following inequalities hold:

\[\both{
U\ge\max\{t_p({\haine5^{\vec{k'}}}),t_{p^{-1}}({\haine5^{\vec{k'}}})\}\\
S\ge\max\{2U,\norm{\Chi{\haine5^{\vec{k'}}}}\}\\
T\ge4U+S+t_0\\
k_\addr,k_{\addr_{+1}}\ge\norm S\\
k_\age,k_{\age_{+1}}\ge\norm T\\
k_{\head_{-1}},k_{\head_{+1}}\ge\length{\Chi{\haine4\times (Q_p \cup Q_{p^{-1}})\times\{-1,+1\}}}\\
k_\info,k_\newinfo,k_{\mail_{-1}},k_{\mail_{+1}}\ge1.
}\]

Then, $F\restr{(\haine5^{\vec k})^\Z}\simu[S,T,0,\Phi]G\restr{(\haine5^{\vec{k'}})^{\Z}}$ completely, where $\Phi\defeq\tilde{\Phi}_{\info}{\restr{\Sigma}}$ and 
\begin{equation*}
\Sigma\defeq(\haine5^{\vec{k}})^{\Z}\cap\gra{0}{t_0}{S}{T} \cap \emp[\motvide]{\head_{-1},\head_{+1},\mail_{-1},\mail_{+1},\newinfo}\cap\tilde{\Phi}_{\info}^{-1}((\haine5^{\vec{k'}})^\Z).
\end{equation*}
\end{lemma}

\begin{proof}

Notice that line~\ref{al:uni:comp} together with $\coordi$ make up the permutation whose behavior is described in Lemma~\ref{behavior}, while line~\ref{al:uni:shift} and $\coordi$ make up the permutation of Lemma~\ref{parshift}. Also, as already noted, for any value of $\age$, at most one permutation of lines~\ref{al:uni:comp} and~\ref{al:uni:shift} is not equal to the identity.

Like in the proofs of Lemmas~\ref{behavior} and~\ref{parshift}, we can easily see that we only have to show that if $\Phi(c)=b$, then $\Phi F^T(c) = G(b)$, and that this follows from the following stronger fact:

\begin{fact}\label{fact:shiftandpermutation}
If $c \in (\haine5^{\vec{k}})^{\Z}\cap\gra{0}{t_0}{S}{T}\cap\emp[\motvide]{\newinfo,\head_{-1},\head_{+1},\mail_{-1},\mail_{+1}}\cap\tilde{\Phi}_{\info}^{-1}(b)^\Z$, then $F^{4U+S}(c)$ exists if and only if $G(b)$ exists, and in that case $F^{4U+S}(c) \in \gra{0}{t_0+4U+S}{S}{T}\cap\emp[\motvide]{\newinfo,\head_{-1},\head_{+1}}\cap\tilde{\Phi}_{\info}^{-1}(G(b))$.
\end{fact}

%Indeed, if $c \in \grb0{0}{\head_{-1}=\head_{+1}=\mail_{-1}=\mail_{+1}=\motvide}\cap\tilde{\Phi^{-1}}(b)$, then until $\age$ becomes $t_0$ no field of $\C[\unive]\cup \C[\shift]$ changes. Therefore, $F^{t_0}(c) \in \grb0{t_0}{\head_{-1}=\head_{+1}=\mail_{-1}=\mail_{+1}=\motvide}\cap\tilde{\Phi^{-1}}(b)$. 
Indeed, according to Fact~\ref{fact:permutation}, we have that $F^{4U}(c)$ exists if and only if $\alpha(b)$ exits, or, equivalently, if and only if $G(b)$ exists, and in this case 
\begin{equation*}
F^{4U}(c) \in \gra{0}{t_0+4U}{S}{T}\cap\emp[\motvide]{\head_{-1},\head_{+1},\mail_{-1},\mail_{+1},\newinfo}\\ \cap \tilde{\Phi}_{\info}^{-1}(\alpha(b)).\end{equation*}

Fact~\ref{fact:shift} implies that
%\begin{multline*}
$F^{4U+S}(c) \in \gra{0}{t_0+4U+S}{S}{T}\cap \emp[\motvide]{\head_{-1},\head_{+1},\mail_{-1},\mail_{+1},\newinfo}\cap \tilde{\Phi}_{\info}^{-1}(\sigma^{-\nu}\alpha(b)),$
%\end{multline*} 
which concludes the proof of the lemma, since $G=\sigma^{-\nu}\alpha(b)$.
\end{proof}

$\unive$ is a rule (family of rules in fact, since they depend on parameters) that can simulate any PPA. Every IPPA $F$ that we will construct later will factor onto $\unive$. They might have some additional fields for which we apply a different rule, and this rule might even take into consideration the fields of $\C[\unive]$, but none of these other rules is going to \emph{change} the fields of $\C[\unive]$. Therefore, by projecting onto $\C[\unive]$, we will immediately obtain that $F$ simulates $G$, even though the simulation might not be exact. 

However, if $\vec{k}$ does not have any anonymous fields, then the simulation is exact, since all the fields of $\C[\coordi] \sqcup \C[\unive]$ are uniquely determined by $\Phi$.

\subsection{Satisfying the inequalities}

Lemma~\ref{universal} is true only under the assumption that the following set of inequalities are satisfied:

\[\both{
U\geq\max\{t_p({\haine5^{\vec{k'}}}),t_{p^{-1}}({\haine5^{\vec{k'}}})\}\\
S\geq\max\{2U,\norm{\Chi{\haine5^{\vec{k'}}}}\}\\
T\ge4U+S+t_0\\
k_\addr,k_{\addr_{+1}}\geq\norm S\\
k_\age,k_{\age_{+1}}\geq\norm T\\
k_{\head_{-1}},k_{\head_{+1}}\geq \length{\Chi{\haine4\times (Q_p \cup Q_{p^{-1}})\times\{-1,+1\}}}\\ %+ \length{\Chi{\haine4\times Q_{p^{-1}}\times\{-1,+1\}}}\\
k_\info,k_\newinfo,k_{\mail_{-1}},k_{\mail_{+1}}\geq 1
~.}\]

We will denote this set of inequalities by $\I(\vec{k},\vec{k'},S,T,U,t_0,p,p^{-1})$. When $t_0=0$, \ie when the computation starts at $\age=0$, we will omit it and write $\I(\vec{k},\vec{k'},S,T,U,p)$, instead. Let us explain again intuitively why each inequality is needed:
\begin{itemize}
\item $k_\addr,k_{\addr_{+1}}\ge\norm S$: The fields $k_\addr$ and $k_{\addr{+1}}$ have to be large enough so that we can write the binary representation of $S$ in them.
\item $k_\age,k_{\age_{+1}}\ge\norm T$: The fields $k_\age$ and $k_{\age{+1}}$ have to be large enough so that we can write the binary representation of $T$ in them.
\item $U\ge\max\{t_p({\haine5^{\vec{k'}}}),t_{p^{-1}}({\haine5^{\vec{k'}}})\}$: We have to run the TM long enough so that the computation of $p$ onto the encoded letters halts.
\item $S\ge\max\{2U,\norm{\Chi{\haine5^{\vec{k'}}}}\}$: The colonies have to be wide enough so that we can encode the letters of $\haine5^{\vec{k'}}$ in them. In addition, they have to be wide enough so that the heads of the computation do not \xpr{collide}.
\item $T\ge4U+S+t_0$: $T$ has to be large enough so that the computation and the shifting are done before the next working period starts.
\item $k_{\head_{-1}},k_{\head_{+1}}\ge\length{\Chi{\haine4\times (Q_p\cup Q_{p^{-1}})\times\{-1,+1\}}}$: %+\length{\Chi{\haine4\times Q_{p^{-1}}\times\{-1,+1\}}}$: 
The head fields have to be large enough so that states of $\gamma_{\U}[p]$ and $\gamma_{\U}[p^{-1}]$ can be written on them.
\item $k_\info,k_\newinfo,k_{\mail_{-1}},k_{\mail_{+1}}\ge1$. Empty fields are of no use, in general.
\end{itemize}

\begin{remark}
If $p,p^{-1}\in \haine2^*, \vec{k'} \in \N^M$ and $t_0 \in \N$ are fixed, then we can choose $\vec{k},U,S$ and $T$ such that the inequalities of Lemma~\ref{universal} is satisfied.
\end{remark}

\begin{proof}
Since $\vec{k'}, p$ and $t_0$ are fixed, given $U,S,T$ we can choose $\vec{k} \defeq \vec{k}_{U,S,T}$ such that all of the inequalities except for the three first are satisfied as equalities. Then, given $S,U$ we can choose $T \defeq T_{S,U}$ such that the third inequality is satisfied as an equality. Similarly, given $U$ we can choose $S \defeq S_U$ such that the second inequality is satisfied. Finally, $U$ can be chosen independently from the rest of the parameters, since it only depends on $p$ and $\vec{k'}$, which are fixed.
\end{proof}

In later constructions, the choice of $U$ will not be so straightforward, as $\vec{k'}$ will depend on $\vec{k}$. In this case, we first fix $G$ and then look for the suitable values of the parameters. The situation becomes trickier when the simulated RPCA depends on the choice of the parameters, as will be the case in the following chapters. Then, we have to be careful not to fall into a circular argument.

\chapter{Infinite hierarchies}\label{s:hierarchy}

For every PPA $G$ and sufficiently large $S,T$, Lemma~\ref{universal} shows that is is possible to construct a PPA $F$ that $(S,T,0)$-simulates $G$. In addition, the simulation can be made exact. We also want  to make it complete. The most direct way is to restrict $F$ to $\tilde{\Phi}_{\info}^{-1}(\dom{G})$, which, by definition makes the simulation complete. However, this is not good because it is a radius-$S$ SFT condition and, if we wanted to have an infinite nested simulation, we would have to impose an infinite number of such restrictions, so that the subshift we would obtain would not be an SFT.

The idea, which is the basic idea behind all hierarchical constructions, is that if the simulated alphabet  is determined in an easy way by the simulating alphabet, then it is possible to design a simple IPPA that ensures that the simulating configuration is in $\tilde{\Phi}_{\info}^{-1}(\dom{G})$.

\section{Son-father checks}
The first thing is to check that the simulated letter, which is written in an encoded form bit by bit in $\info$, has the correct structure, \ie it is the encoding of a letter with the correct number of fields and lengths.

$\vec{k'} \in \haine5^{**} \pto \N^M$ is a vector valuation, that %, in polynomial time,
gives the lengths of the simulated alphabet as a function of the lengths of the simulating alphabet. In applications, it will be easily computable from every letter of the simulating alphabet, or, in other words, the information about the structure of the simulated letter will be known to all of the letters of the simulating IPPA.
%This is achieved by the following permutation. $\vec{v}$ is a function from $\haine5^{**}$ to $\N^{M}$. It takes as input any letter of $\haine5^{**}$ and outputs a vector of lengths. In applications, it will either be constant by definition, or we will restrict to alphabets where it is constant. 
%The restriction of $\vec{v}$ to $\haine5^{\vec{k}}$ will be constant, for all $\vec{k} \in \N^{*}$, which 
%This means that the lengths of the simulated letter only depend on the lengths of the simulating letters, and, a fortiori, that the information about the lengths of the simulated alphabet is \xpr{common} to all the letters of the simulating PPCA.

\begin{algo}{chekka}{\chekka}{M,v_\addr,t_\info,\vec{k'}}
\IF{$v_\addr\ge l_{\vec{k'},M}$}
\STATE{$\chekk(t_\info=3)$} \COMMENT{On the right side of the encoding, \info\ is $3$.}
\ENDIF
\FOR{$0\le i<M$}
\IF{$v_\addr=l_{\vec{k'},i}$}
\STATE{$\chekk[t_\info=2]$} \COMMENT{Field separators are at the expected positions.}
\ELSIF{$l_{\vec{k'},i}<v_\addr<l_{\vec{k'},i+1}$}
\STATE{$\chekk[t_\info\in\haine2]$} \COMMENT{Proper field encodings are binary.}
\ENDIF
\ENDFOR
\end{algo}

This permutation is polynomially computable in its parameters. (In every case that we use it, it will be easily checkable that the parameters are polynomially computable.)
%Let $\grs{s}{S}$ denote the set of configurations $c \in \haine5^{**}$ such that $\bina{\pi_{\addr}}(\hs{c})= \dinf{(s\ldots(S-1)0\ldots(s-1))}$. Of course, $\gra{s}{t}{S}{T} \subseteq \grs{s}{S}$, for all $T \in \N$ and $t \in \co{0}{T}$. We use $\grs{s}{S}$ when the $\age$ is not relevant.

The following lemma follows simply  by inspection of the definition of $\Chi{\cdot}$ and ${\chekka}[M,v_\addr,t_\info,\vec{k'}]$:

\begin{lemma}\label{lem:fixalpha}
Let us fix a field list $[\addr,\info]\in\N^2$, an integer $M\in\Ns$, a vector $\vec{k'} \in \N^M$, $S\in\Ns$ and a vector $\vec{k} \in \N^*$.
%Let $\vec{k} \in \N^*$ be a vector satisfying:

Let $F$ be the IPPA defined by a null direction vector and permutation $\chekka[M,\bina{\pi_{\addr}},\pi_{\info},\vec{k'}]$, and assume that the following inequalities hold:
\[\both{
S \geq \norm{\Chi{\haine5^{\vec{k'}}}}\\
k_\addr\ge\norm S\\
k_\info\ge1
~.}\]

Let $c \in (\haine5^{\vec k})^{\Z} \cap \grs{s}{S} \cap \tilde{\Phi}_{\info}^{-1}(b)$, where $s \in \co{0}{S}$ and $b \in (\haine5^{**})^{\Z}$. Then, $F(c)$ exists if and only if $b \in (\haine5^{\vec{k'}})^{\Z}$ and in this case $F(c)=c$.
\end{lemma}

In other words, if a configuration is split into colonies using the $\addr$ field and every colony has the encoding of some letter on its $\info$ tape, then $\chekka[M,\bina{\pi_{\addr}},\pi_{\info},\vec{k'}]$ ensures that this encoded letter belongs in $\haine5^{\vec{k'}}$. 

We can also check that some field $i$ in the simulated letter has a prescribed prefix (given by a term $t$). 

\begin{algo}{hier}{\hier}{M,v_\addr,t_\info,\vec{k'},i,t}
\IF{$l_{\vec{k'},i}<v_\addr\le l_{\vec{k'},i}+\length{\Chi{t}}$}
\STATE{$\chekk[t_\info=\Chi{t}\restr{v_\addr-l_{\vec{k'},i}}]$}
\ENDIF
\end{algo}
%if we want to allow prefixes, first line is \IF{$l_{\vec{v},i}<v_\addr\le l_{\vec{v},i}+3\length t$}

\begin{lemma}\label{lem:fixfield}
Let us fix a field list $[\addr,\info]\in\N^2$, an integer $M\in\Ns$, a field $i\in\co0M$, a covector $\vec{k'}\in \N^M$, $S \in \N$, a term $t \colon \haine5^{**} \to \haine5^{*}$ and a vector $\vec{k} \in \N^*$.

Let $F$ be the IPPA defined by a null direction vector and permutation $\hier[M,\bina{\pi_\addr},\info,\vec{k'},i,t]$, and assume that the following inequalities hold:

\[\both{
S \geq \norm{\Chi{\haine5^{\vec{k'}}}}\\
k_\addr\ge\norm S\\
k_\info\ge1
~.}\]

Let $c \in (\haine5^{\vec k})^{\Z} \cap \grs{s}{S} \cap \phi^{-1}(b)$, where $0 \le s <S$ and $b \in (\haine5^{\vec{k'}})^{\Z}$ and assume that $t(c_n)=t(c_{n'})\defeq t_c$ for all $n,n' \in \Z$.\\
Then, $F(c)$ exists if and only if ${\pi_i(b_j)}_{\co{0}{\norm{t_c}}}=t_c$, for all $j \in \Z$ and in this case $F(c)=c$.
\end{lemma}

We implicitly assume that if $l > \norm{w}$, where $w \in \A^{*}$, then $w_l=\motvide$.

%\begin{proof}
%It follows immediately the definition of $\Chi{\cdot}$ and $\hier[M,\bina{\pi_\addr},\info,\vec{k'},i,t]$.
%\end{proof}

In other words, if all letters of $c$ have the \xpr{same idea} about what $\pi_i(b_j)$ should be, then, they can check in one step that this indeed happens. In practice, $t$ will usually be equal to $\pi_{\field}$, where $\field$ is a horizontally constant field, so that the condition $t(c_n)=t(c_{n'})$ will be true. In this case, we just check that $\pi_\field(c_n)$ is a prefix of $\pi_{\field}(b_j)$, for all $j,n \in \Z$. %the condition $t(c_n)=t(c_{n'})\defeq t_c^j$ will usually be satisfied in the following way: $t$ will be equal to $\pi_{\field}$, where $\field$ is horizontally constant in every configuration that is not ultimately rejected.

Lemmas~\ref{lem:fixalpha} and \ref{lem:fixfield} correspond to what in \cite{drs} is achieved by mentioning that ``the TM knows at which place the information of every field is held''. For many people, this argument is one of the most confusing things in that construction. This is the reason why we have tried to explain this point as clearly as possible and show exactly how the cells of the simulating IPPA can collectively check that some constant information of the simulating alphabet is the same in the simulated alphabet. In fact, we use a general term $t$ in Lemma~\ref{lem:fixfield}, which essentially allows us to impose any (polynomially computable) condition on the simulated alphabet.

%We do this in an explicit way, that allows fast decoding and comparison to a given term, and by keeping the length of the encoding small (constant times larger than the length of the encoded letter). Once this has been achieved, the specific details of the encoding function $\Chi{\cdot}$ are not really important.

\section{Self-simulation}\label{sself}

We are now ready to construct a self-simulating RPCA. This is the simplest and first example of nested simulation. We just check that the simulated letter has the same lengths as the simulating ones and that some \xpr{hierarchical} fields (which contain the values $p,p^{-1},U,S,T$ that are fixed in Lemma~\ref{universal}) have the same value in the simulated letter as in the simulating ones (where their values is already fixed).

Let $\C[\selfs]\defeq\C[\unive]\sqcup [\maxaddr,\maxage,\schedule,\prog,\revprog] \in \N^{15}$. 

$\maxaddr$ and $\maxage$ are used to obtain the values of $v_{\maxaddr}$ and $v_\maxage$. Similarly, $\prog,\revprog$ and $\schedule$ are used to obtain the values $t_{\prog},t_{\revprog}$ and $v_{\schedule}$. All of these fields will be horizontally constant.

\begin{algo}{selfs}{\selfs}{M,\vec\nu}
\IF{$\bina{\pi_\age}=0$}
\STATE{$\chekk[\emp[\motvide]{\head_{-1},\head_{+1},\mail_{-1},\mail_{+1},\newinfo}]$}\label{al:self:initial}
\STATE{$\chekka[M,\bina{\pi_\addr},\pi_\info,(\length{\pi_j})_{j<M}]$} \COMMENT{Check that the lengths of the simulated letter are the same}\label{al:self:alph}
\FOR{$i\in\{\maxaddr,\maxage,\schedule,\prog,\revprog\}$} 
\STATE{$\hier[M,\bina{\pi_\addr},\info,(\length{\pi_j})_{j<M},i,\pi_i]$} \COMMENT{Check that the hierarchical fields of the simulated letter are the same}\label{al:self:hiera}
\ENDFOR
\ENDIF
\STATE{$\unive[M$,$\vec{\nu}$,$(\length{\pi_j})_{j<M}$,$\bina{\pi_\addr}$,$\bina{\pi_\age}$, $\bina{\pi_\maxaddr}$, $\bina{\pi_\maxage}$, $\bina{\pi_\schedule}$, $\pi_\prog$, $\pi_\revprog]$}\COMMENT{The alphabet is as expected; we can simulate.}
\STATE{$\coordi[\bina{\pi_\maxaddr},\bina{\pi_\maxage}]$}
\label{al:self:unive} 
\end{algo} 

In the next lemma, we do not want to have any anonymous fields, but only those fields that are used in $\selfs$. There are $15$ fields in $\C[\selfs]$, so we take the field list $[0,\ldots,14]$, which means that we assign a number of $\co{0}{15}$ to every field in $\C[\selfs]$ in some random (but fixed) way. Once we have done this, the corresponding vector of directions is also well-defined.

\begin{lemma}\label{self}
Let us fix the field list $\C[\selfs]\defeq[0,\ldots,14]$, the corresponding direction vector $\vec{\nu}_{\selfs}$, integers $S,T,U\in\Ns$ and vector $\vec k\in\N^{15}$. Let $F$ be the IPPA with directions $\vec{\nu}_{\selfs}$ and permutation $\selfs[15,\vec{\nu}_{\selfs}]$ and $p,p^{-1}$ be the programs for this permutation and its inverse, respectively.

Let $F_{\vec{k},S,T,U}$ be the restriction of $F$ to the subalphabet 
\begin{equation*}
\A_{\vec k,S,T,U}\defeq\haine5^{\vec{k}}\cap\emp[S]{\maxaddr}\cap\emp[T]{\maxage}\cap\emp[U]{\schedule}\cap\emp[p]{\prog}\cap\emp[p^{-1}]{\revprog},
\end{equation*}

and assume that the following inequalities are satisfied:

\[\both{
\I(\vec{k},\vec{k},S,T,U,p,p^{-1})\\
k_\prog \geq\norm p\\
k_\revprog \geq\norm{p^{-1}}\\
k_\maxaddr \geq \norm{S}\\
k_{\maxage} \geq \norm{T}\\
k_{\schedule} \geq \norm{U}
~.}\]

Then, $F_{\vec{k},S,T,U}\simu[S,T,0,\Phi]F_{\vec{k},S,T,U}$ completely exactly, where $\Phi\defeq\tilde{\Phi}_{\info}{\restr{\Sigma}}$ and \begin{equation*}
\Sigma\defeq\A_{\vec k,S,T,U}^{\Z}\cap\gra{0}{0}{S}{T}\cap\emp[\motvide]{\head_{-1},\head_{+1},\mail_{-1},\mail_{+1},\newinfo}\cap\tilde{\Phi}^{-1}_{\info}(\A_{\vec k,U,S,T,p}^\Z).
\end{equation*}
\end{lemma}

It is important to notice that $F$ is a fixed rule that does not depend on $\vec{k},S,T,U$. Therefore, its program $p$ is a \emph{fixed} word which we can \xpr{feed} to itself by restricting the alphabet to $\emp[p]{\prog}$. This is the basic idea of self-simulation. Notice also that the fields for which we do a hierarchical check are exactly those that are fixed in the definition of $\A_{\vec k,S,T,U}$. We need to ensure that these fields have the correct value in the simulated letter. The correct way to do this is to check that the value of the simulated letter is in a good relation with the values in the letters of the simulating IPPA. Here, the relation is simply equality. Later it will be something more complicated.

We will try to give as many details as possible in the following proof because it will serve as a prototype for the rest of the hierarchical simulations.

\begin{proof}

We have to show three things: First of all, that $F_{\vec{k},S,T,U}$ $(S,T,0)$-simulates $F_{\vec{k},S,T,U}$ with decoding function $\Phi$ (simulation), second,  that $\Phi$ is injective (exactness) and, finally,  that $\Omega_{F_{\vec{k},S,T,U,p}} \subseteq \rock{\Phi}$ (completeness).

For the simulation part, let $b \in \A_{\vec k,S,T,U}^{\Z}$ and $c \in \A_{\vec k,S,T,U}^{\Z}\cap\gra0{0}{S}{T}\cap \emp[\motvide]{\head_{-1},\head_{+1},\mail_{-1},\mail_{+1},\newinfo} \cap \tilde{\Phi}^{-1}(b)$. By definition, $c$ is not rejected by the checks of lines~\ref{al:self:initial},\ref{al:self:alph} and~\ref{al:self:hiera}. 

Indeed, line~\ref{al:self:initial} checks that the fields $\head_{-1}$, $\head_{+1}$, $\mail_{-1}$, $\mail_{+1}$, $\newinfo$ are empty, which is true since 
\begin{equation*}
c \in \gra{0}{0}{S}{T}\cap\emp[\motvide]{\head_{-1},\head_{+1},\mail_{-1},\mail_{+1},\newinfo}.
\end{equation*}
Line~\ref{al:self:alph} checks that the lengths of the fields of $b$ and $c$ are the same, while the checks of line~\ref{al:self:hiera} check that $b$ and $c$ have the same values in the fields $\maxaddr, \maxage, \prog$  and $\revprog$, which are true by definition.

Since $c$ is not rejected by these checks and $F_{\vec{k},S,T,U}$ is a subrule of 
\begin{equation*}
\coordi[S,T]\circ \unive_{15,\vec{\nu}_{\selfs}}[(\length{\pi_j})_{j<M},\bina{\pi_\addr},\bina{\pi_\age},S,T,U,p,p^{-1}]
\end{equation*} 
and, by assumption, the inequalities of Lemma~\ref{universal} are satisfied, and $p$ is the program of $\selfs[15,\vec{\nu}_{\selfs}]$, Lemma~\ref{universal} gives that $F_{\vec{k},S,T,U}$ $(S,T,0)$-simulates $F_{\vec{k},S,T,U}$ with decoding function $\Phi$.

For the exactness part, we have already noted various times that the values of the fields in $\C[\unive]$ are uniquely determined for all $c \in \Phi^{-1}(b)$. For the hierarchical fields (\ie $\maxaddr, \maxage, \schedule, \prog, \revprog$) the values are fixed for all $c \in \A_{\vec k,S,T,U}^{\Z}$. In addition, there do not exist any anonymous fields (since we chose $M=15$). Therefore, $\Phi$ is injective and the simulation is exact.

In order to show that the simulation is complete, we will first show that if $c \in \gra{0}{0}{S}{T} \cap F_{\vec{k},S,T,U}^{-T}(\A_{\vec k,S,T,U}^{\Z})$, then $c \in \Phi^{-1}(\A_{\vec k,S,T,U}^{\Z})$. 

Indeed, line~\ref{al:self:initial} checks that $c \in \gra{0}{0}{S}{T}\cap\emp[\motvide]{\head_{-1},\head_{+1},\mail_{-1},\mail_{+1},\newinfo}$ (in the sense that if this is not true, then $F_{\vec{k},S,T,U,p}$ would not be defined, so there would be a contradiction). According to Lemma~\ref{lem:fixalpha}, line~\ref{al:self:alph} checks that for every colony $\col{i}{c}$, $\pi_{\info}(\col{i}{c})$ has the structure of the encoding of a letter in $\haine5^{\vec{k}}$. (We cannot immediately say that $\pi_{\info}(\col{i}{c})$ is the encoding of a letter in $\haine5^{\vec{k}}$ because there are some triplets that are not used by $\Chi{\cdot}$. So for example, if the first three letters in the $\info$ tape are $111$, then $\pi_{\info}(\col{i}{c})$ is not the encoding of a letter in $\haine5^{\vec{k}}$, even though the $2$'s and $3$'s are in the correct positions.) In addition, since $F_{\vec{k},S,T,U}^{T}$ exists and the inequalities $\I(\vec{k},\vec{k},S,T,U,p,p^{-1})$ are satisfied, this means that the computation of $p$ on input $\pi_\info(\col{i}{c})$ halts, therefore for all $i \in \Z$, $\tilde{\phi}_{\info}(\col{i}{c})=b_i \in \haine5^{**}$. Lemma~\ref{lem:fixalpha} now implies that $\tilde{\phi}_{\info}(\col{i}{c})=b_i \in \haine5^{\vec{k}}$.
Finally, line~\ref{al:self:hiera} checks that  $\tilde{\phi}_{\info}(b_i) \in \emp[S]{\maxaddr}\cap\emp[T]{\maxage}\cap\emp[U]{\schedule}\cap\emp[p]{\prog}\cap\emp[p^{-1}]{\revprog}$ by checking that the hierarchical fields of $b_i$ have the same values as the corresponding fields of the letters of $c$ (notice that the hierarchical fields are constant for the letters of $c$, so that Lemma~\ref{lem:fixfield} applies). Summarizing, we have that $b \in \A_{\vec k,S,T,U}^{\Z}$, so that $c \in \A_{\vec k,S,T,U}^{\Z}\cap\gra{0}{0}{S}{T} \cap \emp[\motvide]{\head_{-1},\head_{+1},\mail_{-1},\mail_{+1},\newinfo}\cap\tilde{\Phi}^{-1}_{\info}(\A_{\vec k,S,T,U}^\Z)$.

Finally, if $c \in F_{\vec{k},S,T,U}^{-2T}(\A_{\vec k,S,T,U}^{\Z})$ then $F^t\sigma^s(c) \in \gra{0}{0}{S}{T}$, for some $s\in \co{0}{S}$ and $t \in \co{0}{T}$. Therefore, $F_{\vec{k},S,T,U}^t\sigma^s(c) \in \gra{0}{0}{S}{T} \cap F_{\vec{k},S,T,U}^{-T}(\A_{\vec k,S,T,U}^{\Z})$ so that $F_{\vec{k},S,T,U}^t\sigma^s(c) \in \Phi^{-1}(\A_{\vec k,S,T,U}^{\Z})$. This implies that \\
$\Omega_{F_{\vec{k},S,T,U}} \subseteq F_{\vec{k},S,T,U}^{-2T}(\A_{\vec k,S,T,U}^{\Z}) \subseteq \bigsqcup_{\begin{subarray}c0\le t<T\\0\le s<S\end{subarray}}F^t\sigma^s(\dom(\Phi))$, which means that the simulation is also complete.
\end{proof}

\subsection{Satisfying the inequalities}

%In this subsubsection, we continue using the same assumptions and notation as in Lemma~\ref{self}. 
It is not as straightforward to see that the inequalities $\I(\vec{k},\vec{k},S,T,U,p,p^{-1})$ can be satisfied as it was for Lemma~\ref{universal}, because in this case $\vec{k'}$ is equal to $\vec{k}$, which means that we cannot fix $\vec{k'}$ and then choose $\vec{k},S,T,U$ sufficiently big.

\begin{remark}\label{rem:inequselfsimi}
We can find $\vec{k},S,T,U$ such that the inequalities of Lemma~\ref{self} are satisfied. In addition, for all $\epsilon > 0$, $S / T$ can be made larger than $1 -\epsilon$. (Intuitively, the macro-tiles can be made as close to a square as we want.)
\end{remark}

\begin{proof}
We have to satisfy the following inequalities:
\[\both{
U\ge\max\{t_p({\haine5^{\vec{k}}}),t_{p^{-1}}({\haine5^{\vec{k}}})\}\\
S\ge\max\{2U,\norm{\Chi{\haine5^{\vec{k}}}}\}\\
T\ge 4U+S\\
k_\addr,k_{\addr_{+1}}\ge\norm S\\
k_\age,k_{\age_{+1}}\ge\norm T\\
k_{\head_{-1}},k_{\head_{+1}}\ge\length{\Chi{\haine4\times (Q_p \cup Q_{p^{-1}})\times\{-1,+1\}}}\\
k_\info,k_\newinfo,k_{\mail_{-1}},k_{\mail_{+1}}\ge1\\
k_\prog=\norm p\\
k_\revprog=\norm{p^{-1}}\\
k_\maxaddr = \norm{S}\\
k_{\maxage} = \norm{T}\\
k_{\schedule} = \norm{U}
~.}\]

For all $S,T,U$, let us choose $\vec{k}\defeq\vec{k}_{S,T,U}$ such that all of the inequalities except the first three are equalities. Then, $\norm{\Chi{\haine5^{\vec{k}}}}\le P_1(\log{S},\log{T},\log{U})$ and $\max\{t_p({\haine5^{\vec{k}}}),t_{p^{-1}}({\haine5^{\vec{k}}})\} \le P_2(\log{S},\log{T},\log{U})$, for some polynomials $P_1,P_2$. These follow by definition of $\Chi{\cdot}$ and $\vec{k}_{S,T,U}$ and the fact that the program $p$ is fixed and has polynomial complexity.

Therefore, it is enough to find $S,T,U$ that satisfy the following inequalities:
\[\both{
U\ge P_2(\log{S},\log{T},\log{U})\\
S\ge\max\{2U,P_1(\log{S},\log{T},\log{U})\}\\
T\ge 4U+S
~.}\]

%There is a polynomial $P_1$ (depending only on $P$ and $P'$) such that $P,P' \le P_1$ Therefore, it is enough to satisfy the following simpler inequalities:

%\[\both{
%U\ge P_1(\log{S},\log{T},\log{U})\\
%S\ge\max(2U,P_1(\log{S},\log{T},\log{U}))\\
%T\ge 4U+S
%~.}\]

%There are many choices that can do this. We will now describe two of them, with a qualitative difference between them.

For all $S,U$, let us choose $T \defeq T_{S,U}=S+4U$. Also, for all $S,S_0,r$, let us choose $U \defeq U_{S,S_0,r}=\log^r(S+S_0)$. Then, the third inequality is satisfied and the other two are written as follows:

\[\bothrl{
\log^r(S+S_0)\ge &P_2(\log{S},\log(S+4\log^r(S+S_0)),\log(\log^r(S+S_0))\\
S\ge&\max\{2\log^r(S+S_0),\\
&P_1(\log{S},\log(S+4\log^r(S+S_0)),\log(\log^r(S+S_0)))\}.}\]

There exist \emph{fixed} $r,S_0 \in \N$ such that the first inequality is satisfied for all $S$, because $P_2$ is a fixed polynomial (hence its degree is a fixed number). Let us choose such $r,S_0$. Then, if $S$ is sufficiently large we also have that the second inequality is also satisfied (since $r,S_0$ are now fixed), because the right hand side grows only polylogarithmically in $S$, which finishes the proof of the claim.

For the second claim, $S/T=\frac{S}{S+\log^r(S+S_0)}$, which can be made larger than $1-\epsilon$ by choosing $S$ sufficiently large.
\end{proof}

\begin{corollary}
There exists an RPCA $G$ such that $\orb{G}$ is non-empty, aperiodic and $\NE(G)=\{0\}$.
\end{corollary}

\begin{proof}
Let $G \defeq F_{\vec{k},S,T,U} \colon \A_{\vec k,U,S,T}^{\Z} \to \A_{\vec k,U,S,T}^{\Z}$, for some parameters that satisfy $\I(\vec{k},\vec{k},S,T,U,p,p^{-1})$. This is possible, according to Remark~\ref{rem:inequselfsimi}. By definition, we have that $S < T$.

It is not difficult to see that $G^{-1}(\A_{\vec k,U,S,T}^{\Z})$ is nonempty. Then, Lemmas~\ref{self}, \ref{lem:aperiodichierarchy}, \ref{l:nonvide} and Proposition~\ref{prop:hochman} imply that $\orb{G}$ is non-empty, uncountable, aperiodic and $\NE(G)=\{0\}$.
\end{proof}

This finishes the construction of an extremely-expansive, aperiodic 2D SFT. Once we achieved self-simulation, then extreme expansiveness follows immediately from Proposition~\ref{prop:hochman}.

\section{Hierarchical simulation}
We now want to construct more general nested hierarchical simulations, where the parameters of the simulation might vary in every simulation level. This structure is more flexible than a simple self-simulating RPCA, and it will be more useful in the various applications.

Let us fix the field list $\C[\hsim] \defeq \C[\unive] \sqcup [\lev,\prog,\revprog]$ and let $\vec{\nu}_{\hsim}$ be the corresponding vector of directions.

\begin{itemize}
\item $\prog$ and $\revprog$ are used as in the previous section.
\item $\lev$ is used to obtain the values of $v_{\maxaddr}, v_{\maxage}$ and $v_{\schedule}$, not through a direct projection, as in the previous case, but in a polynomially computable way.
\end{itemize}

%In this construction, the fields $\maxaddr,\maxage$ and $\schedule$ are not used. Instead, we obtain the values of the valuations $v_{\maxaddr}, v_{\maxage}$ and $v_{\schedule}$ as functions of a parameter $n$ that is stored in $\lev$. 
In the following, let $\seq S,\seq T, \seq U \in \N^{\N}$ be sequences of integers and $(\vec{k} \colon \N \to \N^M)_{n \in \N}$ is a \emph{sequence} of vectors depending on $n$ (It can give rise to a vector valuation by using $\bina{\pi_\field}$ as the index of the sequence).

\begin{algo}{hsim}{\hsim}{M,\vec\nu,\vec{k},\seq S,\seq T,\seq U}
%\STATE{$\chekk[c]$}\label{al:hier:check}
\IF{$\bina{\pi_\age}=0$}
\STATE{$\chekk[\emp[\motvide]{\head_{-1},\head_{+1},\mail_{-1},\mail_{+1},\newinfo}]$}\label{al:hier:empty}
\STATE{$\chekka[M,\bina{\pi_\addr},\pi_\info,\vec{k}_{\bina{\pi_\lev}+1}]$} \COMMENT{Check that the lengths of the simulated letter are correct} \label{al:hier:alph}
\STATE{$\hier[M,\bina{\pi_\addr},\pi_\info,\vec{k}_{\bina{\pi_\lev}+1},\prog,%\sh[\vec{v}_{\bina{\pi_\lev}+1,\prog}]{\pi_\prog}
\pi_{\prog}]$} \COMMENT{$\prog$ of the simulated letter is the same}\label{al:hier:prog}
\STATE{$\hier[M,\bina{\pi_\addr},\pi_\info,\vec{k}_{\bina{\pi_\lev}+1},\revprog,%\sh[\vec{v}_{\bina{\pi_\lev}+1,\revprog}]{\pi_\revprog}
\pi_\revprog]$} \COMMENT{$\revprog$ is also the same}\label{al:hier:revprog}
\STATE{$\hier[M,\bina{\pi_\addr},\pi_\info,\vec{k}_{\bina{\pi_\lev}+1},\lev,%\sh[\vec{v}_{\bina{\pi_\lev}+1,\lev}]{\anib{\bina{\pi_\lev}+1}}
\anib{\bina{\pi_\lev}+1}]$} \COMMENT{$\bina{\lev}$ of the simulated letter increases by $1$}\label{al:hier:lev}
\ENDIF
\STATE{$\unive[M$, $\vec{\nu}$, $\vec{k}_{\bina{\pi_\lev}}$, $\bina{\pi_\addr}$ ,$\bina{\pi_\age}$, $\seq{S}_{\bina{\pi_\lev}}$, $\seq{T}_{\bina{\pi_\lev}}$, $\seq{U}_{\bina{\pi_\lev}}$, $\pi_\prog$, $\pi_\revprog]$} \COMMENT{Simulate}\label{al:hier:unive}
\STATE{$\coordi[\seq{S}_{\bina{\pi_\lev}},\seq{T}_{\bina{\pi_\lev}}]$}
\end{algo}

%This permutation is also computable in its parameters.

We will now construct a nested simulation of RPCA where the simulation parameters are different at every level.

\begin{lemma}\label{lem:nestsimul}

Let $\seq U,\seq S,\seq T$ be polynomially checkable sequences of integers. Let us fix the field list $\C[\hsim]\defeq [0,\ldots,12]$, the corresponding \emph{fixed} direction vector $\vec{\nu}_{\hsim}$ and a polynomially checkable sequence of $13$-uples $\seq{\vec{k}}\in(\N^{13})^\N$. 

Let $F$ be the IPPA with directions $\vec{\nu}_{\hsim}$ and permutation $\hsim[13,\vec{\nu}_{\hsim},\vec{k},\seq S,\seq T, \seq U;\C[\hsim]]$ and $p,p^{-1}$ be the programs for this permutation and its inverse, respectively.

For all $n \in \N$, let $F_n$ be the restriction of $F$ to the subalphabet

\begin{equation*}
\A_{n}\defeq \haine5^{\vec{k}_n}\cap\emp[n]{\lev}\cap\emp[p]{\prog}\cap\emp[p^{-1}]{\revprog},
\end{equation*}

and assume that the following inequalities hold for all $n \in \N$:

\[\both{
\I(\vec{k}_n,\vec{k}_{n+1},S_n,T_n,U_n,p,p^{-1})\\
k_{n,\prog}\geq\norm p\\
k_{n,\revprog}\geq\norm{p^{-1}}\\
k_{n,\lev} \geq \norm{n}
~.}\]

Then, $F_n\simu[S_n,T_n,0,\Phi_n]F_{n+1}$ completely exactly, where $\Phi_n\defeq\tilde{\Phi}_{\info}{\restr{\Sigma_n}}$ and $\Sigma_n \defeq \A_{n}^{\Z} \cap \gra{0}{0}{S_n}{T_n}\cap \emp[\motvide]{\head_{-1},\head_{+1},\mail_{-1},\mail_{+1},\newinfo} \cap \tilde{\Phi}^{-1}(\A_{n+1}^{\Z})$
\end{lemma}

The proof is very similar to the proof of Lemma~\ref{self}. There exists some differences, though. For example, we do not have the fields $\maxaddr, \maxage$ and $\schedule$. These fields are computed with the aid of field $\lev$, so we perform a hierarchical check for $\lev$. Apart from that, the proof follows the same pattern.

Another difference, which will be important when we prove that the inequalities can be satisfied is that the program is not fixed once we fix $M$ and $\vec{\nu}$, as in the self-similar case, but depends on $\vec{k},\seq{S},\seq{T}$ and $\seq{U}$. Therefore, its complexity also depends on these parameters. More precisely, $t_p(\A_{n}) = P(\length{\A_n}+t_{\vec{k}}(n)+t_{\seq{S}}(n)+t_{\seq{T}}(n)+t_{\seq{U}}(n))$, for some polynomial $P$ that does not depend on the parameters. This is due to the fact that the program consists in a bounded number (independent of $n$) of polynomially computable functions and a bounded number of calls to the parameters. Similarly, $\length{p}= O(\length{p_{\vec{k}}}+\length{p_{\seq{S}}}+\length{p_{\seq{T}}}+
\length{p_{\seq{U}}})$. (The same things hold for $p^{-1}$.)

\begin{proof}
Let us fix $n \in \N$.
We have to show three things: that $F_n$ $(S_n,T_n,0)$-simulates $F_{n+1}$ with decoding function $\Phi_n$ (simulation), that $\Phi_n$ is injective (exactness) and that $\Omega_{F_{n}} \subseteq \dom(\Phi_n)$ (completeness).

For the simulation part, let $b \in \A_{n+1}^{\Z}$ and $c \in \Phi^{-1}(b) \in \A_{n}^{\Z}\cap\gra{0}{0}{S_n}{T_n} \cap \emp[\motvide]{\head_{-1},\head_{+1},\mail_{-1},\mail_{+1},\newinfo}$. By definition, $c$ is not rejected by the checks of lines~\ref{al:hier:empty},\ref{al:hier:alph},\ref{al:hier:prog}, \ref{al:hier:revprog} and \ref{al:hier:lev}, . 

Since $c$ is not rejected by these checks and $F_{n}$ factors onto
\begin{align*}
&\unive[13,\vec{\nu}_{\unive},(\length{\pi_j})_{j<M},\bina{\pi_\addr},\bina{\pi_\age},S_n,T_n,U_n,p,p^{-1}]\\
&\coordi[S_n,T_n]
\end{align*} 
and, by assumption, the inequalities of Lemma~\ref{universal} are satisfied by $\vec{k}_n$ and $\vec{k}_{n+1}$, and $p$ is the program of $\hsim[13,\vec{\nu}_{\hsim},\vec{k},\seq S,\seq T, \seq U]$, Lemma~\ref{universal} gives that $F_{n}$ $(S_n,T_n,0)$-simulates $F_{n+1}$ with decoding function $\Phi_n$.

For the exactness part, we have already noted various times that the values of the fields in $\C[\unive]$ are uniquely determined for all $c \in \Phi^{-1}(b)$. For the hierarchical fields (\ie $\lev, \prog, \revprog$) the values are fixed for all $c \in \A_{n}^{\Z}$. In addition, there do not exist any anonymous fields (since we chose $M=13$). Therefore, $\Phi_n$ is injective and the simulation is exact.

For the completeness part, we will only show that if $c \in \gra{0}{0}{S_n}{T_n} \cap F_{n}^{-T}(\A_{n}^{\Z})$, then $c \in \Phi^{-1}(\A_{n+1}^{\Z})$. Having shown this, it is easy to conclude that the simulation is complete using the same argument as in the proof of Lemma~\ref{self}

Indeed, line~\ref{al:hier:empty} checks that $c \in \gra{0}{0}{S_n}{T_n}\cap \emp[\motvide]{\head_{-1},\head_{+1},\mail_{-1},\mail_{+1},\newinfo}$. According to Lemma~\ref{lem:fixalpha}, line~\ref{al:hier:alph} checks that for every colony $\col{i}{c}$, $\pi_{\info}(\col{i}{c})$ has the structure of the encoding of a letter in $\haine5^{\vec{k}_{n+1}}$. In addition, since $F_{n}^{T_n}(c)$ exists and the equations $\I(\vec{k}_{n},\vec{k}_{n+1},S_n,T_n,U_n,p,p^{-1})$ are satisfied, this means that the computation of $p$ on input $\pi_\info(\col{i}{c})$ halts, therefore for all $i \in \Z$, $\tilde{\phi}_{\info}(\col{i}{c})=b_i \in \haine5^{**}$. Lemma~\ref{lem:fixalpha} now implies that $\tilde{\phi}_{\info}(\col{i}{c})=b_i \in \haine5^{\vec{k}_{n+1}}$.
Finally, lines~\ref{al:hier:lev},\ref{al:hier:prog} and \ref{al:hier:revprog} check that  $\tilde{\phi}_{\info}(b_i) \in \emp[n+1]{\lev}\cap\emp[p]{\prog}\cap\emp[p^{-1}]{\revprog}$. 

Summarizing, we have that $b \in \A_{n+1}^{\Z}$, so that $c \in \A_{n}^{\Z}\cap\gra0{0}{S_n}{T_n} \cap \emp[\motvide]{\head_{-1},\head_{+1},\mail_{-1},\mail_{+1},\newinfo}\cap\tilde{\Phi}^{-1}_{\info}(\A_{n+1}^\Z) = \Sigma_n$, or equivalently that $c \in \Phi^{-1}(\A_{n+1}^{\Z})$.
\end{proof}

\subsection{Satisfying the inequalities}

Let us now show that the inequalities of Lemma~\ref{lem:nestsimul} can be satisfied:

\begin{remark}\label{rem:inequhiera}
We can find $\vec{k} \in (\N^{13})^{\N}$ and $\seq{S},\seq{T},\seq{U} \in \Ns^{\N}$ such that the inequalities of Lemma~\ref{lem:nestsimul} are satisfied. In addition, $\prod_{i<n} S_i/T_i$ can be made both $0$ and $\neq 0$.
\end{remark}

We have to deal with two problems, which were not present in the previous cases: First, there is an infinite set of inequalities, since there is also an infinite set of RPCA, and they must all be satisfied simultaneously. Second, the size of the program and the complexity of the permutations depends on the choice of the parameters $\seq S$ and $\seq T$.

\begin{proof}
We have to satisfy the following inequalities, for all $n \in \N$
\[\both{
U_n\ge\max\{t_p({\haine5^{\seq{\vec{k}}_{n+1}}}),t_{p^{-1}}({\haine5^{\seq{\vec{k}}_{n+1}}})\}\\
S_n\ge\max\{2{U}_n,\norm{\Chi{\haine5^{\vec{k}_{n+1}}}}\}\\
T_n\ge 4{U}_n+S_n\\
k_{n,\prog}\geq\norm p\\
k_{n,\revprog}\geq\norm{p^{-1}}\\
k_{n,\head_{-1}},k_{n,\head_{+1}}\ge\length{\Chi{\haine4\times (Q_p \cup Q_{p^{-1}})\times\{-1,+1\}}}\\
k_{n,\addr},k_{n,\addr_{+1}}\ge\norm{S_n}\\
k_{n,\age},k_{n,\age_{+1}}\ge\norm{T_n}\\
k_{n,\info},k_{n,\newinfo},k_{n,\mail_{-1}},k_{n,\mail_{+1}}\ge1\\
k_{n,\lev}=\norm{n}
~.}\]

For all $n \in \N$ and $\seq{S},\seq{T}$ and $\seq{U}$,  let us choose $\vec{k}_n\defeq\vec{k}_{n,\seq{S},\seq{T},\seq{U}}$ such that the last four inequalities are satisfied as equalities. Then, we can see that %the sequence $\vec{k}$ is polynomially checkable,
\begin{equation*} 
\norm{\Chi{\haine5^{\vec{k}_n}}}\le P_1(\log{S_n},\log{T_n},\log{n},k_{n,\head_{-1}},k_{n,\head_{+1}},k_{n,\revprog},k_{n,\prog}),
\end{equation*}
for some polynomial $P_1$. 

We claim that $\length{p} \le c(\length{p_{\vec{k}}}+\length{p_{\seq{S}}}+\length{p_{\seq{T}}}+\length{p_{\seq{U}}})$, for some constant $c$. (The same holds for $p^{-1}$ and we can assume that the constant $c$ is the same.) This is because, as we have already noticed, the program of $p$ uses a fixed number of polynomial operations and a bounded number of calls to the parameters $\length{p_{\vec{k}}}$,  $\length{p_{\seq{S}}}$, $\length{p_{\seq{T}}}$, $\length{p_{\seq{U}}}$.

For the same reason, we have that
\begin{multline*}
\max\{t_p({\haine5^{\vec{k}}}),t_{p^{-1}}({\haine5^{\vec{k}}})\} \le P_2(\log{S_n},\log{T_n},\log{n},k_{n,\head_{-1}},\\
k_{n,\head_{+1}},k_{n,\revprog},k_{n,\prog},t_{\vec{k}}(n),t_{\seq S}(n),
,t_{\seq T}(n),t_{\seq U}(n)),
\end{multline*}
for some \emph{fixed} polynomial $P_2$ that does not depend on the parameter sequences.

Therefore, it is enough to find sequences $\vec{k},\seq{S},\seq{T},\seq{U}$ that satisfy the following inequalities, for all $n \in \N$:
\[\bothrl{
k_{n,\prog}, k_{n,\revprog}\geq & c(\length{p_{\vec{k}}}+\length{p_{\seq{S}}}+\length{p_{\seq{T}}}+
\length{p_{\seq{U}}})\\
%\geq & c(\length{p_{\vec{k}}}+\length{p_{\seq{S}}}+\length{p_{\seq{T}}}+
%\length{p_{\seq{U}}})\\
k_{n,\head_{-1}},k_{n,\head_{+1}}\geq & \length{\Chi{\haine4\times (Q_p \cup Q_{p^{-1}})\times\{-1,+1\}}}\\
U_n\ge & P_2(\log{S_{n+1}},\log{T_{n+1}},\log(n+1),\\
&k_{n+1,\head_{-1}}, k_{n+1,\head_{+1}},k_{n+1,\revprog},k_{n+1,\prog},\\ &t_{\vec{k}}(n+1),t_{\seq S}(n+1),t_{\seq T}(n+1),t_{\seq U}(n+1))\\
S_n \ge & \max\{2U_n,P_1(\log{S_{n+1}},\log{T_{n+1}},\log(n+1),\\
& \length{p_{\vec{k}}},\length{p_{\seq{S}}},\length{p_{\seq{T}}},\length{p_{\seq{U}}})\}\\
T_n \ge & 4U_n+S_n
~.}\]

Recall that in the above inequalities, $Q_p$ and $Q_{p^{-1}}$ depend on the choice of parameter sequences.

We will show two ways to do this. The first one does not give an extremely-expansive SFT, because $\prod_{i<n}S_i/T_i$ does not converge to $0$, while the second one does.

\begin{enumerate}
\item For all sequences $\seq S$ and $\seq U$, let us choose $T_n \defeq T_{n,\seq S,\seq U}=S_n+4U_n$. Also, for all $n_0,r$ and $Q \geq 2$, let us choose $U_{n,n_0,r}\defeq U_n=(n+n_0)^r$, $S_{n,n_0}\defeq S_n=Q^{n+n_0}$, $k_{n,\prog}=k_{n,\revprog}=n_0Qr$ and $k_{n,\head_{-1}}=k_{n,\head_{+1}}=n_0$ for all $n \in \N$.

Then, the last inequality is satisfied by definition. In addition, for all $n_0,r,Q$, we have that $\length{p_{\seq{S}}} \le \norm{c_1n_0Q}$, $\length{p_{\seq{U}}} \le \norm{c_2rn_0}$ and $\length{p_{\seq{T}}}$, $ \length{p_{\vec{k}}} \le \norm{c_3rn_0Q}$, for some \emph{constants} $c_1,c_2,c_3$. 
This is true because the sequence $(n+n_0)^r$ is uniformly (polynomially) computable in $n,n_0,r$, which means that there exists an algorithm that takes as input $(n,n_0,r)$ and outputs $(n+n_0)^r$, for \emph{all} values of $n,n_0$ and $r$. If we use the program for this algorithm together with a description of $n_0$ and $r$, then we obtain a program of length bounded by $\norm{c_2rn_0}$ for the sequence $((n+n_0)^r)_{n \in \N}$. A similar argument holds for the sequence $Q^{n+n_0}$.

Since this algorithm works for \emph{all} choices of $n_0,Q,r$, it means that $Q_p$ and $Q_{p^{-1}}$ are actually \emph{fixed}.

In addition, all of the algorithms are polynomially computable, which means that  
\begin{equation*}
t_{\seq S}(n), t_{\seq T}(n), t_{\vec{k}}(n) \le P_3(\log{Q^{n+n_0}}), t_{\seq U}(n) \le P_4(\log(n+n_0)^r),
\end{equation*} 
for some \emph{fixed} polynomials $P_3,P_4$.

Therefore, substituting these in the inequalities above and doing some regrouping of the terms in parentheses (that is omitted), the inequalities that need to be satisfied are written as follows:

\[\bothrl{
n_0Qr \geq & c'\log(n_0Qr)\\
n_0 \geq & c'\\
(n+n_0)^r\ge & P_5(\log{Q^{n+n_0+1}},\log(n+n_0+1)^r,\log(n+1))\\
Q^{n+n_0}\ge & \max\{2(n+n_0+1)^r, P_6(\log{Q^{n+n_0+1}},\log(n+1))\}
~,}\]
for some polynomials $P_5,P_6$ and constant $c'$ that do not depend on $r,n_0$ or $Q$.

Since $c'$ is fixed, the first two inequalities are true for all but a finite number of triples $n_0,Q,r$. Without loss of generality, we assume that it is always true.
We can choose $n_Q$ and $r$ such that the second inequality is true for all $n \in \N$ and all $n_0 \geq n_Q$, because the right hand of the inequality is bounded by a fixed polynomial of $(n+n_0)$ and $r$, while the left-hand side grows like $n^r$. With fixed $r$, we can also find $n'_Q$ such that the third inequality is satisfied for all $n \in \N$ and all $n_0 \geq n'_Q$, because the left-hand side grows exponentially in $(n+n_0)$ and the right hand only polynomially (since $r$ is fixed). By choosing $n_0=\max\{n_Q,n'_Q\}$ we can satisfy both inequalities for all $n$ at the same time. %Obviously, they are polynomially checkable.

Note that $\prod_{i \in \N}S_i/T_i = \prod_{i \in \N}(1+(n+n_0)^r/Q^{n+n_0}) \neq 0$. Therefore, if we choose the sequences like this, we do not obtain a unique direction of non-expansiveness, but rather a cone of non-expansive directions.

\item For all $n_0 \in \N$ and $Q\geq 2$, let us choose $S_{n,n_0}\defeq S_n=Q^{n+n_0}$, $T_{n,n_0}\defeq T_n=2S_n$ and $U_{n,n_0}\defeq U_n=\frac{S_n}{2Q}$, $k_{n,\prog}=k_{n,\revprog}=n_0Q$ and $k_{n,\head_{-1}}=k_{n,\head_{+1}}=n_0$ for all $n \in \N$. We can use a similar argumentation as in previous case to show that it is enough to satisfy the following inequalities: %have that $t_{\seq S} = O((n+n_0)\log{Q})=t_{\seq U}=t_{\seq T}$, for all $Q$. Therefore, as before, the inequalities are written in the following way: 

\[\both{
n_0Q \geq \norm{Qn_0}\\
\frac{Q^{n+n_0}}{4}\ge P_3(\log{Q^{n+n_0+1}},\log(n+1))\\
Q^{n+n_0}\ge\max\{\frac{Q^{n+n_0+1}}{2Q},P_4(\log{Q^{n+n_0+1}},\log(n+1))\}
~,}\]

for some polynomials $P_3,P_4$ and constant $c$ that do not depend on $n_0$ and $Q$. 

Obviously, for all $Q$ these inequalities are satisfied when $n_0$ is sufficiently large.

In this case, $\prod_{i \in \N}S_i/T_i= \prod_{i \in \N}S_i/2S_i = 0$, therefore the corresponding SFT is extremely expansive.
\end{enumerate}
\end{proof}

For both cases, we have a lot of freedom in choosing the sequences. In the previous proof, we just described two of the possible ways which are enough for the results we want to obtain and help in presenting the basic ideas of the proof that is needed in any possible case.

\section{Universality}
Let $\C[\other]=[\otherinfo_{-1},\otherinfo,\otherinfo_{+1}]$ and let us fix the field list $\C[\intru]=\C[\hsim]\sqcup\C[\other]$ and the corresponding direction vector $\vec{\nu}_{\intru}$.

For any $n$, consider an RPCA $G_n$ with permutation $\alpha_n \colon (\haine5^{l_n})^3 \to (\haine5^{l_n})^3$ over $\C[\other]$. This is not a strict restriction in itself: all RPCA can be represented in this way, up to a simple alphabet renaming and use of Remark~\ref{sharpization}% (for example by considering $\sh[\vec{w}]{\alpha_n}$, for some large enough $\vec{w}$).

If the sequence of permutations $(\alpha_n)_{n \in \N}$ is polynomially computable, we can build a PPA that simulates $G_n$, for all $n \in \N$. 

\begin{algo}{intru}{\intru}{M,\vec{\nu},\vec{k},\seq S,\seq T,\seq U, \seq\alpha}
\STATE{$\alpha_{\bina{\lev}}[\C[\other]]$} \COMMENT{$G_n$ on the $\C[\other] $ fields.}\label{al:intru:gn}
\IF{$\bina{\pi_\age}=0$}
\STATE{$\chekk[\emp[\motvide]{\head_{-1},\head_{+1},\mail_{-1},\mail_{+1},\newinfo}]$}\label{al:univ:empty}
\STATE{$\chekka[M,\bina{\pi_\addr},\pi_\info,\vec{k}_{\bina{\pi_\lev}+1}]$}  \label{al:intru:alph}
\STATE{$\hier[M,\bina{\pi_\addr},\pi_\info,\vec{k}_{\bina{\pi_\lev}+1},\prog,%\sh[\vec{v}_{\bina{\pi_\lev}+1,\prog}]{\pi_\prog}
\pi_{\prog}]$}\label{al:intru:prog}
\STATE{$\hier[M,\bina{\pi_\addr},\pi_\info,\vec{k}_{\bina{\pi_\lev}+1},\revprog,%\sh[\vec{v}_{\bina{\pi_\lev}+1,\revprog}]{\pi_\revprog}
\pi_\revprog]$} \label{al:intru:revprog}
\STATE{$\hier[M,\bina{\pi_\addr},\pi_\info,\vec{k}_{\bina{\pi_\lev}+1},\lev,%\sh[\vec{v}_{\bina{\pi_\lev}+1,\lev}]{\anib{\bina{\pi_\lev}+1}}
\anib{\bina{\pi_\lev}+1}]$} \label{al:intru:lev}
\ENDIF
\STATE{$\unive[M$, $\vec{\nu}$, $\vec{k}_{\bina{\pi_\lev}}$, $\bina{\pi_\addr}$, $\bina{\pi_\age}$, $\seq{S}_{\bina{\pi_\lev}}$, $\seq{T}_{\bina{\pi_\lev}}$, $\seq{U}_{\bina{\pi_\lev}}$, $\pi_\prog$, $\pi_\revprog]$} \COMMENT{Simulate}\label{al:intru:unive}
\STATE{$\coordi[\seq{S}_{\bina{\pi_\lev}},\seq{T}_{\bina{\pi_\lev}};\C[\coordi]]$}
\end{algo}

%This permutation is polynomially computable in its parameters.

The only difference of this rule with $\hsim$ is that it has 3 additional fields (which implies that $\vec{k}$ will be chosen in $(\N^{16})^{\N}$) and that we apply $\alpha_{\bina{\lev}}$ onto the field list $\C[\other]$ \emph{independently} from what we do on $\C[\hsim]$.

\begin{lemma}\label{lem:univppa}

Let $\seq U,\seq S,\seq T$ be polynomially checkable sequences of integers and $\alpha$ a polynomially computable sequence of permutations.
Let us fix the field list $\C[\intru]\defeq [0,\ldots,15]$, the corresponding \emph{fixed} direction vector $\vec{\nu}_{\intru}$ and a polynomially checkable sequence of $M$-uples $\seq{\vec{k}}\in(\N^{15})^\N$. Let $F$ be the IPPA with directions $\vec{\nu}_{\intru}$ and permutation $\intru[15,\vec{\nu}_{\intru},\vec{k},\seq S,\seq T, \seq U;\C[\intru]]$ and $p,p^{-1}$ be the programs for this permutation and its inverse, respectively.

For all $n \in \N$, let $F_n$ be the restriction of $F$ to the subalphabet

\begin{equation*}
\A_{n}\defeq \haine5^{\vec{k}_n}\cap\emp[n]{\lev}\cap\emp[p]{\prog}\cap\emp[p^{-1}]{\revprog},
\end{equation*}

and assume that the following inequalities hold:

\[\both{
\I(\vec{k}_n,\vec{k}_{n+1},S_n,T_n,U_n,p)\\
k_{n,\prog}\geq\length p\\
k_{n,\revprog}\geq\length{p^{-1}}\\
k_{n,\lev} \geq \norm{n}\\
k_{n,\otherinfo}=k_{n,\otherinfo_{-1}}=k_{n,\otherinfo_{+1}}\geq l_n
~.}\]

If $\Omega_{G_n} \neq \emptyset$, then $F_n$ completely $(S_n,T_n,0)$-simulates $F_{n+1}$ with decoding function $\Phi_n=\tilde{\Phi}_{\info}{\restr{\Sigma_n}}$, where 
%\begin{equation*}
$\Sigma_n \defeq \A_{n}^{\Z} \cap \gra{0}{0}{S}{T}\cap \emp[\motvide]{\head_{-1},\head_{+1},\mail_{-1},\mail_{+1},\newinfo} \cap \Phi^{-1}(\A_{n+1}^{\Z}).$
%\end{equation*}

The simulation is exact if and only if $\Omega_{G_n}$ is a singleton. In addition, if $c \in F_n^{-1}(\A_{n}^{\Z})$, then $G_n \pi_{\C[\other]}(c)=\pi_{\C[\other]}F_n(c)$.
\end{lemma}

As mentioned before, the proof is very similar to the proof of Lemma~\ref{lem:nestsimul}. Therefore, we are going to omit most of the details and only stress those points where there is a difference.

\begin{proof}

Let us fix $n \in \N$.
For the first claim, we have to show that $F_n$ $(S_n,T_n,0)$-simulates $F_{n+1}$ with decoding function $\Phi_n$ (simulation), that $\Phi_n$ is an injection if and only if $\Omega_{G_n}$ is a singleton and that $\Omega_{F_{n}} \subseteq \dom(\Phi_n)$ (completeness).

For the simulation part, let $b \in \A_{n+1,p}^{\Z}$. We can find $c \in \A_{n}^{\Z}$ that simulates $b$: we choose $c \in \Phi^{-1}(b) \in \A_{n}^{\Z}\cap\gra{0}{0}{S}{T}\cap \emp[\motvide]{\head_{-1},\head_{+1},\mail_{-1},\mail_{+1},\newinfo}$ such that $\pi_{\C[\other]}(c) \in \Omega_{G_n}$ (this is possible by the assumption that $\Omega_{G_n} \neq \emptyset$). Then, it is easy to see that $c$ simulates $b$, because $\C[\other]$ is only \xpr{touched} by $\alpha_{\bina{\lev}}\defeq \alpha_n$, $p$ is the program of $\intru[17,\vec{\nu}_{\intru},\vec{k},\seq S,\seq T, \seq U;\C[\intru]]$ and $G_n^{T}(\pi_{\C[\other]}(c))$ exists.

For the exactness part, as usual $\pi_{\C[\hsim]}(c)$ is uniquely determined by $b$. However, $\pi_{\C[\other]}(c)$ can be chosen independently from $b$ to be any element of $\Omega_{G_n}$, so that the simulation is exact if and only if $\Omega_{G_n}$ is a singleton.

Finally, for the completeness part, an argument almost identical to the argument in the proof of Lemma~\ref{lem:nestsimul} shows that if $c \in \gra{0}{0}{S}{T} \cap F_{n}^{-T}(\A_{n}^{\Z})$, then $c \in \Phi^{-1}(\A_{n+1}^{\Z})$. As we know, this is enough to show that the simulation is complete.

The second claim, that if $c \in F_n^{-1}(\A_{n}^{\Z})$, then $G_n \pi_{\C[\other]}(c)=\pi_{\C[\other]}F_n(c)$ is straightforward from the definition of $F_n$, since the only rule that \xpr{touches} the fields $\C[\other]$ is $G_n$.
\end{proof}

\begin{remark}\label{rem:univenonempt}
\begin{enumerate}
\item $\Omega_{F_0} \neq \emptyset$ if and only if $\Omega_{G_n} \neq \emptyset$, for all $n \in \N$.
\item If $\Omega_{F_0} \neq \emptyset$, then $F_0$ completely simulates $G_n$ for all $n \in \N$.
\end{enumerate}
\end{remark}

\begin{proof}
\begin{enumerate}
\item If $\Omega_{G_n} = \emptyset$ for some $n \in \N$, then $\Omega_{F_n} = \emptyset$, so that since $F_0$ simulates $F_n$ (by transitivity of simulation), we obtain that $\Omega_{F_0}=\emptyset$ by Lemma~\ref{lem:aperiodichierarchy}.

If, on the other hand, $\Omega_{G_n} \neq \emptyset$ for all $n \in \N$, then Lemma~\ref{l:nonvide} gives that $\Omega_{F_0} \neq \emptyset$.
\item If $\Omega_{F_0} \neq \emptyset$, then the second claim of Lemma~\ref{lem:univppa} implies that $F_n$ factors onto $G_n$. Since $F_0$ simulates $F_n$, for all $n \in \N$, we obtain that $F_0$ simulates $G_n$, for all $n \in \N$.
\end{enumerate}
\end{proof}

\begin{remark}\label{rem:univextrexp}
Even if $\prod_{i \in \N} S_i/T_i =0$, $F_0$ is not necessarily extremely expansive, since we might have non-expansive directions coming from the $G_n$ part. However, %it can be shown that if $\NE(G_n)=\{\infty\}$, for all $n \in \N$, then $\NE(F_n)=\{\infty\}$, too.
in the special case that $\Omega_{G_n}$ is a singleton for all $n \in \N$, then all the simulations are exact and it is straightforward to see that $\NE(F_0) = \{0\}$, because Proposition~\ref{prop:hochman} applies.
\end{remark}

\subsection{Satisfying the inequalities}

%For all $n \in \N$ and $\seq{S},\seq{T}$ and $\seq{U}$,  let us choose $\vec{k}_n\defeq\vec{k}_{n,\seq{S},\seq{T},\seq{U}}$ such that all of the inequalities except the first five are equalities. Then, we can see that %the sequence $\vec{k}$ is polynomially checkable,
%\begin{equation*} 
%\norm{\Chi{\haine5^{\vec{k}_n}}}\le P_1(\log{S_n},\log{T_n},\log{{U}_n},\log{n},p_{\vec{k}},p_{\seq{S}},p_{\seq{T}},p_{\seq{U}}),
%\end{equation*}
%for some polynomial $P_1$. This is easy to see, keeping in mind that $\length{p} \le c(\length{p_{\vec{k}}}+\length{p_{\seq{S}}}}+\length{p_{\seq{T}}}}+\length{p_{\seq{U}}}})$, for some constant $c$. A similar argument works for $p^{-1}$ and we can assume that the constant $c$ is the same.

%In addition, we claim that 
%\begin{equation*}
%\max\{t_p({\haine5^{\vec{k}}}),t_{p^{-1}}({\haine5^{\vec{k}}})\} \le P_2(\log{S_n},\log{T_n},\log{{U}_n},\log{n}+t_{\vec{k}}(n)+t_{\seq S}(n)+t_{\seq T}(n) +t_{\seq U}(n)),
%\end{equation*}
%for some \emph{fixed} polynomial $P_2$ that does not depend on the parameter sequences.

\begin{remark}\label{rem:inequunivppa}
We can find $\vec{k} \in (\N^{16})^{\N}$ and $\seq{S},\seq{T},\seq{U} \in \Ns^{\N}$ such that the inequalities of Lemma~\ref{lem:univppa} are satisfied and $\prod_{i \in \N}S_i/T_i=0$.
\end{remark}

We only state the case $\prod_{i \in \N}S_i/T_i=0$ (even though we can make it $\neq 0$) too, because it is what we will need and use in the applications.

\begin{proof}
Let us write explicitly the inequalities that we need to satisfy:

\[\both{
U_n\ge\max\{t_p({\haine5^{\seq{\vec{k}}_{n+1}}}),t_{p^{-1}}({\haine5^{\seq{\vec{k}}_{n+1}}})\}\\
S_n\ge\max\{2{U}_n,\length{\Chi{\haine5^{\vec{k}_{n+1}}}}\}\\
T_n\ge 4{U}_n+S_n\\
k_{n,\prog}\geq\length p\\
k_{n,\revprog}\geq\length{p^{-1}}\\
k_{n,\head_{-1}},k_{n,\head_{+1}}\ge\length{\Chi{\haine4\times (Q_p \cup Q_{p^{-1}})\times\{-1,+1\}}}\\

k_{n,\addr},k_{n,\addr_{+1}}\ge\norm{S_n}\\
k_{n,\age},k_{n,\age_{+1}}\ge\norm{T_n}\\
k_{n,\info},k_{n,\newinfo},k_{n,\mail_{-1}},k_{n,\mail_{+1}}\ge1\\
k_{n,\lev} = \norm{n}\\
k_{n,\otherinfo}=k_{n,\otherinfo_{-1}}=k_{n,\otherinfo_{+1}}= l_n
~.}\]

For all $\seq{S},\seq{T}$ and $\seq{U}$ ,  let us choose $\vec{k}_n\defeq\vec{k}_{n,\seq{S},\seq{T},\seq{U}}$ such that the last five inequalities are satisfied as equalities. Then, we have the crucial inequality
\begin{equation*} 
\norm{\Chi{\haine5^{\vec{k}_n}}}\le P_1(\log{S_n},\log{T_n},\log{n},l_n,k_{n,\head_{-1}},k_{n,\head_{+1}},k_{n,\revprog},k_{n,\prog}),
\end{equation*}
where $l_n$ is the size of the alphabet of $\alpha_n$.

The other crucial inequality of the proof of Lemma~\ref{rem:inequhiera} also holds without any essential changes:

\begin{multline*}
\max\{t_p({\haine5^{\vec{k}}}),t_{p^{-1}}({\haine5^{\vec{k}}})\} \le P_2(\log{S_n},\log{T_n},\log{n},l_n,k_{n,\head_{-1}},\\
k_{n,\head_{+1}},k_{n,\revprog},k_{n,\prog},t_{\vec{k}}(n),t_{\seq S}(n),
,t_{\seq T}(n),t_{\seq U}(n))
\end{multline*}
for some polynomial $P_2$ that does not depend on the parameters.

This holds because the permutation applied consists in a number of polynomial operations (recall that $\alpha$ is polynomially computable and fixed for this specific construction) and a bounded number of calls to $\seq S$, $\seq T$,$\seq U$  and $\vec{k}$.

Also, since $\alpha$ is polynomially computable, $l_n$ (which is part of the output of $\alpha_n$) is also bounded by a polynomial of $n$ so that we can \xpr{remove} $l_n$ from the right-hand side of the previous inequalities and \xpr{incorporate} it in the polynomials $P_1,P_2$.
%\begin{equation*}
%\length{\Chi{\haine5^{\vec{k}_n}}}\le P_3(\log{S_n}+\log{T_n}+\log{U_n}+\log{n},\length{p_{\vec{k}}},\length{p_{\seq{S}}},\length{p_{\seq{T}}},\length{p_{\seq{U}}}), 
%\end{equation*}
%for some polynomials $P_3$ that does not depend on any of the parameters.
From this point on, the proof is identical to the proof of Remark~\ref{rem:inequhiera}. (We are free to chose whether $\prod_{i \in \N}S_i/T_i$ is equal to $0$ or not.)
\end{proof}

\subsection{Domino problem}

\begin{theorem}
It is undecidable whether an extremely expansive SFT is empty.
\end{theorem}

\begin{proof}
Let $\M$ be an arbitrary TM with program $p'$. For all $n \in \N$, we define $\alpha_n$ as follows:

$l_n =1$. $\alpha_n(0,0,0)=(0,0,0)$ if $\halt{p'}{n}{0^n}$ is true (\ie if $\M$ does not halt within $n$ steps). $\alpha_n$ is undefined in all other cases.

$(\alpha_n)_{n \in \N}$ is a polynomially computable sequence of permutations. $\Omega_{G_n}$ is a singleton, equal to $\{\dinf 0 \}$, if and only if $\M$ does not halt within $n$ steps. Otherwise $\Omega_{G_n}$ is empty.

Let us construct the sequence of RPCA $(F_n)_{n \in \N}$ as in Lemma~\ref{lem:univppa} corresponding to the $\alpha$ and $\vec{k}, \seq S, \seq T, \seq U$ that satisfy the inequalities and for which $\prod_{i \in \N} S_i/T_i=0$. Then, Remark~\ref{rem:univenonempt} implies that $\Omega_{F_0}$ (equivalently, $\orb{F_0}$) is non-empty if and only if $\Omega_{G_n}$ is non-empty for all $n$, which is equivalent to that $\M$ does not halt over input $0^{\infty}$.

In addition, Remark~\ref{rem:univextrexp} implies that if $\Omega_{F_0}$ is non-empty, then $\NE(F_0)=\{0\}$.

Therefore, for every TM $\M$, we have constructed a 2D SFT $\orb{F_0}$ that is non-empty if and only if $\M$ does not halt over input $0^{\infty}$ and if $\orb{F_0} \neq \emptyset$, then $\orb{F_0}$ is extremely expansive. This concludes the proof of the undecidability.
\end{proof}

It follows from the previous proof that we have actually proved the following: Let $A$ be the family of forbidden patterns that define empty SFT, and let $B$ be the family that defines non-empty extremely expansive SFT (with unique direction of non- expansiveness $\infty$). There does not exists a recursively enumerable set $X$ that contains $B$ and is disjoint from $A$. In other words, if an algorithm correctly recognizes all non- empty extremely-expansive SFT, then it must also (falsely) recognize an empty SFT.

\subsection{Intrinsic universality}
The second application concerns the universality properties of RPCA.

\begin{theorem}\label{c:intru}
For any computably enumerable set of non-empty PPA, there exists a PPA that completely simulates all of them.
\end{theorem}

\begin{proof}
First of all, we can assume that all the PPA are over the field list $\C[\other]$ with the corresponding directions. This is true because we can encode, in polynomial time, all the left-moving fields into a unique left-moving field, and similarly for the other types of fields. Then, saying that a set of PPA is computably enumerable is equivalent to saying that the corresponding set of permutations that define these PPA is computable enumerable.

In addition, for every computably enumerable set of PPA $X$ (over the field list $\C[\other]$), there exists a \emph{polynomially computable sequence} $(G_n)_{n\in\N}$ of PPA that contains exactly the elements of $X$. Equivalently, there exists a \emph{polynomially computable sequence} of permutations $(\alpha_n)_{n \in \N}$ that contains exactly the permutations of the PPA in $X$.

(Let $g$ be a fixed element of $X$. The polynomial algorithm of $(\alpha_n)_{n \in \N}$ takes as input $n$, runs the algorithm that enumerates $X$ for $n$ steps and sets $\alpha_n$ equal to the last permutation that was output. If no permutation has yet been output, then $\alpha_n$ is set equal to $g$.)

If we use this sequence $\alpha$ and sequences $\vec{k}, \seq S, \seq T, \seq U$ that satisfy the inequalities to define the sequence $(F_n)_{n \in \N}$ as in Lemma~\ref{lem:univppa}, then, since by assumption $\Omega_{G_n} \neq \emptyset$ for all $n \in \N$, Remark~\ref{rem:univenonempt} implies that $F_0$ completely simulates $G_n$, for all $n \in \N$.
\end{proof}

Theorem~\ref{c:intru} applies, up to a conjugacy, to computably enumerable sets of nonempty RPCA.
In some sense, it gives a deterministc version of the result in \cite{lafitteweiss}.
The same result is not true for the non-computably-enumerable set of all nonempty RPCA, thanks to an argument by Hochman \cite{hochmanuniv} and Ballier \cite{balliermedvedev}.
Nevertheless, the corollary applies to the family of all reversible (complete) cellular automata, since the family of RCA is computably enumerable.
Unfortunately, it gives an RPCA (partial CA) that simulates all RCA (full CA) instead of an RCA. The existence of an RCA that simulates all RCA seems to be a much more difficult question and is still open, (see for instance \cite{guillaume1}.

\section{Synchronizing computation}\label{s:comput}
%The results of this subsection could be rather direct corollaries of either of the following sections, but presenting them here allows to
We now introduce one more trick in our construction: the encoding of an infinite sequence inside an infinite nested simulation by encoding increasing finite prefixes of the infinite sequence inside the alphabets of the RPCA of the nested simulation.

Let $\C[\syncomp]\defeq\C[\unive] \sqcup [\mhist,\mhist_{+1},\prog,\revprog]$. In this simulation, we do not use a field $\lev$ in order to store the parameter $n$. Instead, it will be obtained as the length of field $\mhist$. $p'$ is the program of a TM. It is used to reject some nested simulation sequences, depending on the infinite sequence that is stored (through its increasing finite prefixes) in the alphabets of the RPCA.

\begin{algo}{syncomp}{\syncomp}{M,\vec{\nu},\vec{k},\seq S,\seq T,\seq U,p'}
\STATE{$\chekk[\pi_\mhist=\pi_{\mhist_{+1}}]$}\label{al:syncomp:mhistconsistency}
\IF{$\bina{\pi_\age}=0$}
\STATE{$\chekk[\halt{p'}{\length{\mhist}}{\mhist}]$}\label{al:syncomp:medvedev}
\STATE{$\chekk[\emp[\motvide]{\head_{-1},\head_{+1},\mail_{-1},\mail_{+1},\newinfo}]$}\label{al:syncomp:empty}
	\STATE{$\chekka[M,\bina{\pi_\addr},\pi_\info,\vec{k}_{\length{\mhist}+1}]$} \COMMENT{Check that the lengths of the simulated letter are correct} \label{al:syncomp:alph}
	\STATE{$\hier[M,\bina{\pi_\addr},\pi_\info,\vec{k}_{\length{\mhist}+1},\prog,%\sh[\vec{v}_{\length{\mhist}+1,\prog}]{\pi_\prog}
\pi_{\prog}]$} \COMMENT{$\prog$ of the simulated letter is the same}\label{al:syncomp:prog}
	\STATE{$\hier[M,\bina{\pi_\addr},\pi_\info,\vec{k}_{\length{\mhist}+1},\revprog,%\sh[\vec{v}_{\length{\mhist}+1,\revprog}]{\pi_\revprog}
\pi_\revprog]$} \COMMENT{$\revprog$ is also the same}\label{al:syncomp:revprog}
	\STATE{$\hier[M,\bina{\pi_\addr},\pi_\info,\vec{k}_{\length{\mhist}+1},\mhist,\pi_{\mhist}]$} \COMMENT{$\mhist$ of the simulating letters is a prefix of $\mhist$ of the simulated}\label{al:syncomp:infinitesequence}
\ENDIF
	\STATE{$\unive[M$, $\vec{\nu}$, $\vec{k}_{\length{\mhist}}$, $\bina{\pi_\addr}$, $\bina{\pi_\age}$, $\seq{S}_{\length{\mhist}}$, $\seq{T}_{\length{\mhist}}$, $\seq{U}_{\length{\mhist}}$, $\pi_\prog$, $\pi_\revprog]$} \label{al:syncomp:unive}
	\STATE{$\coordi[\seq{S}_{\length{\mhist}},\seq{T}_{\length{\mhist}}]$}
\end{algo}

%This is polynomially computable in its parameters.

\begin{lemma}\label{l:mhist}
Let $\seq S,\seq T,\seq U$ be polynomially checkable sequences of integers and $p'$ be the program of a TM. Let us fix the field list $\C[\syncomp]\defeq [0,\ldots,13]$, the corresponding \emph{fixed} direction vector $\vec{\nu}_{\syncomp}$ and a polynomially checkable sequence of $14$-uples $\seq{\vec{k}}\in(\N^{14})^{\N}$. Let $F$ be the IPPA with directions $\vec{\nu}_{\syncomp}$ and permutation 
\begin{equation*}
\syncomp[14,\vec{\nu}_{\syncomp},\vec{k},\seq S,\seq T, \seq U,\pi_{\mhist},p']
\end{equation*} 
and $p,p^{-1}$ be the programs for this permutation and its inverse, respectively.

For all $w \in \haine2^{*}$, let %$\vec{k}_w \defeq \vec{k}_{\norm{w}}$, 
$S_w\defeq S_{\length{w}}$, $T_w \defeq T_{\length{w}}$%, $U_w \defeq U_{\norm{w}}$ 
and $F_w$ be the restriction of $F$ to the subalphabet

\begin{equation*}
\A_{w}\defeq \haine5^{\vec{k}_{\length{w}}}\cap\emp[w]{\mhist,\mhist_{+1}}\cap\emp[p]{\prog}\cap\emp[p^{-1}]{\revprog},
\end{equation*}

and assume that the following inequalities hold for all $n \in \N$:

\[\both{
\I(\vec{k}_n,\vec{k}_{n+1},S_n,T_n,U_n,p,p^{-1})\\
k_{n,\prog}\geq\length p\\
k_{n,\revprog}\geq\length{p^{-1}}\\
k_{n,\lev} \geq \norm{n}

~.}\]

Then, $F_w\simu[S_w,T_w,0,\Phi_w]\bigsqcup_{\begin{subarray}{c}a\in \haine2\end{subarray}}{F_{wa}}$ completely exactly, where
\begin{equation*}
\Sigma_w \defeq \A_{w}^{\Z} \cap \gra{0}{0}{S_w}{T_w} \cap \emp[\motvide]{\head_{-1},\head_{+1},\mail_{-1},\mail_{+1},\newinfo} \cap \tilde{\Phi}^{-1}(\bigsqcup_{\begin{subarray}{c} a \in \haine2 \end{subarray}}\A_{wa}^{\Z}), \text{ and } \Phi_w\defeq\tilde{\Phi}_{\info}{\restr{\Sigma_w}}.
\end{equation*}
\end{lemma}

\begin{proof}
Let $w \in \haine2^*$ and $\length{w}=n$. By definition, $S_w\defeq S_n$, $T_w\defeq T_n$ and $U_w \defeq U_n$. If $p'$ halts on input $w$ within $n$ steps, then the check of line~\ref{al:syncomp:medvedev} will reject every configuration, which means that $\dom(F_{w}) = \emptyset$. But, in this case, $wa$ will also be rejected by $p'$ within $n+1$ steps, for all $a \in \haine2$, so that $\dom(\bigsqcup_{\begin{subarray}{c}a\in \haine2\end{subarray}}{F_{wa}}) = \emptyset$, too. By definition, the empty PCA strongly, completely simulates itself for all possible choices of the simulating parameters, so that the claim is true in this case.

Suppose, then, that $p'$ does not halt on input $w$ within $n$ steps. Then, the check of line~\ref{al:syncomp:medvedev} is always true, so that we can ignore it in the rest of the proof. As in the previous proofs, we have to show three things: that $F_{w}$ $(S_n,T_n,0)$-simulates $\bigsqcup_{\begin{subarray}{c}a\in \haine2\end{subarray}}{F_{wa}}$ with decoding function $\Phi_{w}$ (simulation), that $\Phi_{w}$ is injective (exactness) and that $\Omega_{F_{w}} \subseteq \dom(\Phi_w)$ (completeness).

For the simulation, it is easy to see that if $b \in \A_{wa}^{\Z}$, where $a \in \haine2$ and $c \in \Phi_w^{-1}(b)$, then $c$ is not rejected by the checks of lines~\ref{al:syncomp:mhistconsistency},\ref{al:syncomp:empty},\ref{al:syncomp:alph},\ref{al:syncomp:prog}, \ref{al:syncomp:revprog} and \ref{al:syncomp:infinitesequence}. Then, simulation follows easily from the choice of the program $p$ and Lemma~\ref{universal}. 

Exactness is also direct. The values of all the fields of $c$ are uniquely determined by $b$ and the form $\Phi_w$.

Completeness also follows the general pattern of the previous proofs, but there is a small difference: we can show that if  $c \in \gra{0}{0}{S_n}{T_n} \cap F_{w}^{-2T}(\A_{w}^{\Z})$ (the difference is that we have $2T$ instead of $T$ in the exponent), then $c \in \Phi_{w}^{-1}(\bigsqcup_{\begin{subarray}{c} a \in \haine2 \end{subarray}}\A_{wa}^{\Z})$. This is enough to ensure completeness of the simulation.
 
Indeed, if 
\begin{equation*}
c \in \gra{0}{0}{S_n}{T_n} \cap F_{w}^{-T}(\A_{w}^{\Z}),
\end{equation*}
then lines~\ref{al:syncomp:empty},\ref{al:syncomp:prog},\ref{al:syncomp:revprog} and \ref{al:syncomp:infinitesequence} ensure that 
\begin{equation*}
c \in \A_{w}^{\Z} \cap \gra{0}{0}{S_n}{T_n}\cap\emp[\motvide]{\head_{-1},\head_{+1},\mail_{-1},\mail_{+1},\newinfo} \cap \Phi^{-1}((\bigsqcup_{\begin{subarray}{c} a \in \haine2 \end{subarray}}\A_{wa})^{\Z}).
\end{equation*}
Let $b \in (\bigsqcup_{\begin{subarray}{c} a \in \haine2 \end{subarray}}\A_{wa})^{\Z}$ be such that $c \in \Phi^{-1}(b)$. The problem is that we still cannot know that $\pi_{\mhist}(b)$ is the same in all cells, because line~\ref{al:syncomp:infinitesequence} only checks that at every cell $i$, $\pi_{\mhist}(c)=w$ (which we know that it is constant) is a prefix of $\pi_{\mhist}(b_i)$. However, we could still have that $\pi_{\mhist}(b_i)=w0$ and $\pi_{\mhist}(b_j)=w1$, for some $i\neq j$. This is why we need to take $2T$ steps instead of $T$ steps.

Indeed, since $F_w^{2T}(c)$ exists, $F^2(b)$ also exists, and line~\ref{al:syncomp:mhistconsistency} ensures that $\pi_{\mhist}(b_i)=\pi_{\mhist_{+1}}(b_i)=\pi_{\mhist}(b_j)$, for all $i,j \in \Z$.
%Then, we also know that $b \in F^{-1}(b)$. 
The argument for this is similar to the argument used in the proof of Lemma~\ref{koo}. %, and it is the same every time we want to make sure that some field is horizontally constant using PPA. 
Therefore, $b \in \bigsqcup_{\begin{subarray}{c} a \in \haine2 \end{subarray}}\A_{wa}^{\Z}$ and this concludes the proof of the Lemma.
\end{proof}

\subsection{Satisfying the inequalities}

\begin{remark}\label{rem:inequcomputa}
We can find $\vec{k} \in (\N^{14})^{\N}$ and $\seq{S},\seq{T},\seq{U} \in \Ns^{\N}$ such that the inequalities of Lemma~\ref{l:mhist} are satisfied  and $\prod_{i<n} S_i/T_i =0$.
\end{remark}

\begin{proof}
The proof is almost identical to the proof of Remark~\ref{rem:inequunivppa} and is omitted. We just make a few comments:

First of all, the inequalities depend on $w \in \haine2^{*}$, but in fact, if $\length{w}=\length{w'}$, then we have exactly the same inequalities for $w$ and $w'$, so that actually the inequalities can be translated to a set of inequalities that depend on $n$.

Second, notice that line~\ref{al:syncomp:medvedev} is computable in polynomial time, and since the program $p'$ is fixed in advance, its contribution to $\norm{\A_w}$, $t_p$ and $\length{p}$ is constant and does not depend on the choice of parameters. 

Finally, we can choose $k_{w,\mhist} \defeq n$, (where $n \defeq \length{w}$) which means that this field only contributes a polynomial of $n$ to the various inequalities, so that it can be \xpr{incorporated} into the polynomials and the problem can be reduced to the cases that have already been dealt with.
\end{proof}

\subsection{Realizing computational degrees}

The statement of Lemma~\ref{l:mhist} falls exactly into the situation described in Lemma~\ref{l:nonvides}. For all $n \in \N$, let $\B_n=\haine2$. Then, for all $w \in \haine2^n (= \prod_{i < n} \B_i)$, we have defined $S_w,T_w, F_w$ and $\Phi_w$ such that $F_w$ exactly, completely $(S_w,T_w,0)$-simulates $\bigsqcup_{b\in\B_n}F_{ub}$.

The check of line~\ref{al:syncomp:medvedev} forces that if $z \in \prod_{i \in \N} \B_i$, then $\rocks[\infty]z{\seq\Phi}\ne\emptyset$ if and only if $\halt{p'}{\infty}{z}$ is true, or in other words, $z \in \X_p$. Indeed, we have that $\dom(F_{z_{\co{0}{n}}})=\emptyset$ for some $n$ if and only if $\halt{p'}{\infty}{z}$ is not true, or in other words, if and only if $p'$ halts over $z$ within $n$ steps.

Therefore, Lemma~\ref{l:nonvides} implies that $\Omega_{F_{\motvide}}= \rocks[\infty]Z{\seq\Phi} = \bigsqcup_{z\in \X_p}\rocks[\infty]z{\seq\Phi}$.

\begin{lemma}\label{l:comphomeo}
For any effectively closed subset $Z\subset\haine2^\N$, there exists an extremely expansive RPCA $F$ and a  computable, left-invertible map from $\Omega_{F}$ onto the Cartesian product $Z\times\haine2^\N$.
\end{lemma}
One could even prove that the computable map is two-to-one, and almost one-to-one for any reasonable (topological or measure-theoretical) notion.
Also, this SFT can be effectively constructed from $Z$.

\begin{proof}
We construct $\syncomp[14,$ $\vec{\nu}_{\syncomp}$, $\vec{k}$, $\seq S$, $\seq T$,  $\seq U$, $\pi_{\mhist}$, $p']$ for some sequences that satisfy the inequalities and a program $p'$ that recognizes $Z$. In addition, assume that $\prod_{i<n} S_i/T_i =0$.

It follows from Lemma \ref{l:nonvides} that 
\begin{equation*}
\Omega_{F_{\motvide}}=\bigsqcup_{z \in \X_p}\rocks[\infty]z{\seq\Phi}=
\bigsqcup_{\begin{subarray}c z\in \X_p\\\seq t\in\prod_{i\in\N}\co0{T_i}\\\seq s\in\prod_{i\in\N}\co0{S_i}\end{subarray}}\bigcap_{n\in\N}\sigma^{\anib[\seq S]{s_{\co0n}}}F_{\motvide}^{\anib[\seq T]{t_{\co0n}}}\Phi_{z_0}^{-1}\cdots\Phi_{z_{\co{0}{n}}}^{-1}(\Omega_{F_n}).
\end{equation*}

Consider the map that associates, to each configuration $x \in \Omega_{F_{\motvide}}$, the unique triple $(z,\seq s,\seq t)$ such that $x\in\bigcap_{n\in\N}\sigma^{\anib[\seq S]{s_{\co0n}}}F_0^{\anib[\seq T]{t_{\co0n}}}\Phi_{z_0}^{-1}\cdots\Phi_{z_{\co{0}{n}}}^{-1}(\Omega_{F_n})$.
This map is computable, since, for all $n\in\N$, $S_n$ is for instance given by $\bina{\pi_\age\Phi_0\Phi_1\cdots\Phi_{n-1}}$ and $z_{\co{0}{n}}$ is given by $\pi_{\mhist}\Phi_0\Phi_1\cdots\Phi_{n-1}$.

Conversely, from the tripe $(z,\seq{s},\seq{t})$, one can build construct a configuration in $\Omega_{F_{\motvide}}$, as explained in \cite[Proposition 4.2]{gacs}.

The result follows from the obvious computable homeomorphisms between $\Omega_F$ and $\orb F$, and between $\prod_{i\in\N}\co0{S_i}\times\co0{T_i}$ and $\haine2^{\N^2}$.

Finally, $\NE(F_{\motvide}) = \{0\}$, because $\orb{F_{\motvide}}=\bigsqcup_{z \in \X_p} \orb{F_{\rocks[\infty]z{\seq\Phi}}}$ and $\NE(\orb{F_{\rocks[\infty]z{\seq\Phi}}}) = \{0\}$, due to the exact complete simulation and $\prod_{i<n} S_i/T_i =0$.
\end{proof}

The following was proven in \cite{simpson} for general 2D SFT. Here, we can also restrict the set of non-expansive directions.

\begin{theorem}\label{t:comphomeo}
For any effectively closed subset $X$, there exists a Medvedev-equivalent extremely expansive 2D SFT whose Turing degrees are the cones above the Turing degrees of $X$.
\end{theorem}

A fortiori, all Medvedev (and Mu\v cnik) degrees contain an extremely expansive SFT.
\begin{proof}
It is enough to notice that $\Omega_{F}$ and $\orb{F}$ are computably homeomorphic.
\end{proof}
%Note that, contrary to previous (nonexpansive) constructions, the fact that $c\mapsto\orb c$ is a (clearly computable) bijection here and thus each extremly expansive 2D SFT is computably homeomorphic to some 1D effective subshift, gives us a more symbolic-dynamic-constructive proof of the second result in \cite{miller}: all Medvedev and Mu\v cnik degrees contain 1D effective subshifts.

The second component in the computable homeomorphism cannot easily be taken out: it is pointed in \cite{vanierdegrees} that all aperiodic subshifts admit a cone of Turing degrees (that is one degree and all degrees above it).

Let us make some final comments: In this chapter, we are inspired and draw mainly on the work of Durand, Romashchenko and Shen \cite{drs}. Reading that paper, one has the feeling that the construction of that paper can be done in a reversible way, except for the exchange of information. Working out the details needed to make that intuition work is (as proven by this chapter) messy and even tedious, sometimes, but we manage to obtain results for which there is no known alternative proof. We also feel that our construction can also shed some light on the construction of Durand, Romashchenko and Shen. More specifically, we always write explicitly the inequalities that need to be satisfied, and for each one we explain at least once why it is needed. Also, we construct \xpr{once and for all} the rules and then prove that they have the desired behaviour, instead of using their more informal approach where some rule is created and then it is modified, resulting in a new rule, for which it is taken for granted that the previous argumentation still holds.

\chapter{Expansive directions}\label{sec:expdir}

In Lemma~\ref{lem:nonexpsftrestr}, we described a necessary condition for the set of non-expansive directions of an SFT: if $X$ is an SFT, then $\NE(X)$ is effectively closed. %, \ie there exists a TM that takes as input the description of a direction (in some encoding) $l \in \Rb$ and halts if and only if $l \in \E(X)$. Equivalently, there exists a TM that enumerates a sequence of open subintervals of $\R \sqcup \{\infty\}$ such that the union of these intervals is equal to $\E(X)$.
In this section, we are going to show that this is in fact a characterization of sets of non-expansive directions of SFTs.
\begin{theorem}\label{thm:nonexpansive}
If $\NE_0 \subseteq \Rb$ is effectively closed, then there exists %an alphabet $\A$ and
an SFT $X \subseteq \A^{\Z^2}$ such that $\NE(X)=\NE_0$.
\end{theorem}

This is mentioned as Open Problem~11.1 in Mike Boyle's Open Problems for Symbolic Dynamics \cite{opsd}. We only answer the first part of that problem, since our constructions do not have any SFT direction. The second part of the problem, concerning 2D SFT with an SFT direction is much more difficult to answer, since it is inextricably related to the expansiveness of RCA.

It is enough to prove this for sets of non-expansive directions that are included in $[-1,1]$, or even $[0,r]$, for some $0 < r <1$, because we can cover the set of directions with a finite number of rotations of $[-1,1]$ (and $[0,r]$). Therefore, even though the fact that we are using PPA might seem problematic (since $\NE(F) \subseteq [-1,1]$ in this case), this is not the case.

The key idea consists in constructing subshifts with a unique direction of non-expansiveness through a nested simulation of RPCA, so that we can use Lemma~\ref{prop:hochman}.
This idea was introduced in \cite{nexpdir}, in a non-effective way;
%where for every (not necessarily effectively) closed set of directions, a subshift (not necessarily of finite type) with exactly the given set of directions as set of non-expansive directions was constructed. Our construction can be seen as an ``SFTvization'' of Hochman's construction. Even though our exposition is self-contained,
we will try to emphasize the obstacle that has to be overcome when trying to \xpr{SFTize} this construction. %we have to overcome when we try to use Hochman's construction in an SFT setting, because it gives a good idea of the restrictions imposed by the SFT condition as well as of the flexibility of Hochman's construction.

\section{Directive encoding}\label{subsection:direncoding}

Proposition~\ref{prop:hochman} states that, if we manage to implement a certain kind of nested simulation, then we will obtain a subshift with a unique direction of non-expansiveness, equal to $\anib[\seq S/\seq T]{\seq D}$, where $\seq S,\seq T\in\Ns^\N$ and $\seq D\in\Z^\N$.
\cite[Lemma~5.6]{nexpdir} shows that all directions can be written in this form (when the sequences $\seq S, \seq T, \seq D$ are allowed to be chosen without any constraints).
But the sequences of nested simulations that are possible with our SFT construction are more constrained: for example, the sequences $\seq S$ and $\seq T$ must be polynomially checkable.% and the time needed to check $(i,S_{i+1})$ and $(i,T_{i+1})$ must be less than $S_i$. 
This immediately imposes some restrictions, since, for example, it implies that $S_i$ cannot grow like an exponential tower of height $i$. This is not excluded from the construction of Hochman, since he takes $S$ \xpr{sufficiently large}, in order to make some \xpr{error term} sufficiently small. A large part of our construction is to show that we can satisfy these restrictions at the same time, or, in other words, that the error terms can be made sufficiently small even if $\seq S$ grows relatively slowly. At the same time, we have to take care of some technical details. 

Let us begin the construction by giving some additional necessary definitions:

To any vector $\vec\varepsilon\in\R_+^n$ and any \dfn{directive word} $\vec d\defeq(D_i,W_i)_{0\le i<n}\in(\N^2\setminus\{(0,0)\})^n$, where $n\in\N$, we associate the direction interval $\Theta_{\vec\varepsilon}(\vec d)\defeq(\prod_{0\le i<n}R_i)[-1,1]+\anib[\vec R]{\vec D}\subset\Rb$, where $R_i\defeq1/(D_i+W_i+1+\varepsilon_i)\le1/2$; recall that $\anib[\vec R]{\vec D}\defeq\sum_{0\le i<n}D_i\prod_{0\le j\le i}R_j$. It follows immediately by the definition that $\Theta_{\vec\varepsilon}(\vec d)=R_0(\Theta_{\varepsilon\restr{\co1n}}(\vec d\restr{\co1n})+D_0)$ for any $\vec d=(D_i,W_i)_{0\le i < n}$.

We extend these definitions for infinite sequences in the natural way: To any sequence $\seq\varepsilon\in\R_+^\N$ and any \dfn{directive sequence} $\seq d\defeq(D_n,W_n)_{n\in\N}\in(\N^2\setminus\{(0,0)\})^\N$ we associate the direction $\theta_{\seq\varepsilon}(\seq d)\defeq\anib[\seq R]{\seq D}\in\R$, the unique element of $\bigcap_{n\in\N}\Theta_{\varepsilon\restr{\co0n}}((D_i,W_i)_{0\le i<n})$ (uniqueness follows from the fact that $R_i \le 1 /2$, for all $i \in \N$).

If $F_0\simu[S_0,T_0,D_0S_0]F_1\simu[S_1,T_1,D_1S_1]\ldots$ and for all $n\in\N$, $T_n=(D_n+W_n+1+\epsilon_n)S_n$, then observe that Proposition~\ref{prop:hochman} can be seen as saying that $\NE(F_0)=\{\theta_{\seq\varepsilon}(\seq D,\seq W%(D_n,W_n,\varepsilon_n)_{n\in\N}
)\}$. This is point of contact between this section and the rest of the thesis.

\cite[Lemma 5.6]{nexpdir} can now be reformalized as the following.
\begin{lemma}%[Lemma 5.6 in \cite{nexpdir}]
\label{lem:hochgeomstuff}
For all $x \in [0,1]$, there exist a sequence $\seq\varepsilon\in\R_+^\N$ and a directive sequence $\seq d$ such that $\theta_{\seq\varepsilon}(\seq d) = x$.
\end{lemma}

We refine this statement in two ways: first, we will show that the sequence $\seq\varepsilon$ can be fixed and second, we will restrict the alphabet of acceptable directive sequences to $\parint\defeq\{(0,1),(1,1),(1,0)\}$. By doing that, we will \xpr{lose} a small part on the right endpoint of the interval $[0,1]$.

\begin{lemma}\label{lem:geomstuff}
Let $\seq\varepsilon\in[0,\sqrt2-1]^\N$ be a sequence. Then,
\begin{equation*}
\theta_{\seq\varepsilon}(\parint^\N)
=[0,\anib[(1/(2+\varepsilon_i)_i)]{111\ldots}%(1)_{n\in\N}}
].
\end{equation*}
\end{lemma}

Though the interval $[0,1[$ cannot be covered fully with a fixed, non-trivial sequence, the convergence of the sequence $(\varepsilon_n)_{n \in \N}$ to $0$ can be sped up suitably, in order to realise any number arbitrarily close to $1$.
\begin{proof}
Let us prove by induction over $n\in\N$ that for any such sequence $\seq\varepsilon$, we have that %we have both $\anib[(1/(2+\varepsilon_i))_{i}]{01\ldots12}\ge\frac1{\sqrt2}$ and
\begin{equation*}
\theta_{\varepsilon\restr{\co0n}}(\parint^n)=[-\prod_{i<n}\frac1{2+\varepsilon_i},\anib[(1/(2+\varepsilon_i))_{i<n}]{1\ldots12}]\supset[0,\frac1{\sqrt2}].
\end{equation*}%, where $\anib[(1/(2+\varepsilon_i))_i]{01\ldots12}$ denotes $\sum_{1\le i<n}\prod_{j<i}\frac1{2+\varepsilon_j}+\prod_{j<n}\frac1{2+\varepsilon_j}$.

The base of the induction follows from $\epsilon_0 \le \sqrt{2}-1$.
Let us assume that the inductive hypothesis is true for some $n\in\N$, and prove it for $n+1$.
Note that $
\anib[(1/(2+\varepsilon_i))_{i < n+1}]{1\ldots12}
%=
%\sum_{1\le i\le n}\prod_{j<i}\frac1{2+\varepsilon_j}+\prod_{j\le n}\frac1{2+\varepsilon_j}
%=
%\frac1{2+\varepsilon_0}\left(1+\sum_{2\le i\le n}\prod_{1\le j<i}\frac1{2+\varepsilon_j}+\prod_{1\le j\le n}\frac1{2+\varepsilon_j}\right)
=
\frac1{2+\varepsilon_0}\left(1+\anib[(1/(2+\varepsilon_i))_{1 \le i<n+1}]{1\ldots12}\right)
$
By the induction hypothesis (which we can apply to the truncated sequence $(\epsilon_n)_{n \geq 1}$, since it satisfies the assumption, too) and $\varepsilon_0\le\sqrt2-1$, we have
$\anib[(1/(2+\varepsilon_i))_{i < n+1}]{1\ldots12}\ge\frac1{1+\sqrt2}\left(1+\frac1{\sqrt2}\right)=\frac1{\sqrt2}$.
%\begin{eqnarray*}
%	\anib[(1/(2+\varepsilon_i))_{0\le i}]{01\ldots12}&=&\frac1{2+\varepsilon_0}(1+\anib[(1/(2+\varepsilon_i))_{1\le i}]{01\ldots12})\\
%	&>&\frac1{2+\varepsilon_0}(1+\frac34)\\
%	&=&\frac37\cdot\frac74=\frac34~.
%\end{eqnarray*}

The set $\theta_{\varepsilon\restr{\co{0}{n+1}}}(\parint^{n+1})$ can be decomposed, in terms of the first directive letter, into a union of three intervals
\begin{eqnarray*}
\frac1{2+\varepsilon_0}\theta_{\varepsilon\restr{\cc1n}}(\parint^n)\cup\frac1{3+\varepsilon_0}(\theta_{\varepsilon\restr{\cc1n}}(\parint^n)+1)\cup\frac1{2+\varepsilon_0}(\theta_{\varepsilon\restr{\cc1n}}(\parint^n)+1)
~.
\end{eqnarray*}
Following the induction hypothesis, the first interval is equal to
\begin{equation*}
\frac1{2+\varepsilon_0}[-\prod_{1\le i\le n}\frac1{2+\varepsilon_i},\anib[(1/(2+\varepsilon_i))_{1\le i \le n}]{1\ldots12}]
=[-\prod_{0\le i\le n}\frac1{2+\varepsilon_i},\frac{1}{2+\varepsilon_0}\anib[(1/(2+\varepsilon_i))_{1\le i\le n}]{1\ldots12}].
\end{equation*}
The second interval is equal to
\begin{equation*}
\frac1{3+\varepsilon_0}(1+[-\prod_{1\le i\le n}\frac1{2+\varepsilon_i},\anib[(1/(2+\varepsilon_i))_{1\le i\le n}]{1\ldots12}])=[\frac1{3+\varepsilon_0}(1-\prod_{1\le i\le n}\frac1{2+\varepsilon_i}),\frac1{3+\varepsilon_0}\anib[(1/(2+\varepsilon_i))_{1\le i\le n}]{1\ldots12}].
\end{equation*}
The third interval is equal to
\begin{equation*}
\frac1{2+\varepsilon_0}(1+[-\prod_{1\le i\le n}\frac1{2+\varepsilon_i},\anib[(1/(2+\varepsilon_i))_{1\le i\le n}]{1\ldots12}])=[\frac1{2+\varepsilon_0}-\prod_{0\le i\le n}\frac1{2+\varepsilon_n},\anib[(1/(2+\varepsilon_i))_{0\le i\le n}]{1\ldots12}].
\end{equation*}
It is clear that the smallest point of these three intervals is $-\prod_{0\le i\le n}\frac1{2+\varepsilon_i}$ (the last two intervals are in $\R_+$), and the largest is $\anib[(1/(2+\varepsilon_i))_{0\le i\le n}]{1\ldots12}$ (the first interval is obtained through a translation by $-1$ of the third one, or through a homothecy by $\frac{2+\varepsilon_0}{3+\varepsilon_0}$ of the second one).

We proved earlier that $\anib[(1/(2+\varepsilon_i))_{i < n+1}]{1\ldots12}\ge\frac1{1+\sqrt2}\left(1+\frac1{\sqrt2}\right)=\frac1{\sqrt2}$. It follows from this that the upper bound $\frac1{2+\varepsilon_0}\anib[(1/(2+\varepsilon_i))_{1\le i\le n}]{1\ldots12}$ of the first interval is larger than $\frac{1}{(2+\varepsilon_0)\sqrt2}$, while the smaller bound %$\frac1{3+\varepsilon_0}(1-\prod_{1\le i\le n}\frac1{2+\varepsilon_i})$
of the second interval is lower than $\frac{1}{3+\varepsilon_0}$, which is less than $\frac{1}{(2+\varepsilon_0)\sqrt2}$, as can be easily verified.%, since $3-2\sqrt2-\varepsilon_0(\sqrt2-1)\ge0$. 
In other words, there is no hole between these two intervals.

Using the same arguments, one can easily see that the upper bound of the second interval is $\frac1{3+\varepsilon_0}(1+\anib[(1/(2+\varepsilon_i))_{1\le i\le n+1}]{1\ldots12})$, which is larger than $\frac{1+1/\sqrt2}{3+\varepsilon_0}$ while the lower bound of the third interval is smaller than $\frac{1}{2+\varepsilon_0}$, which is smaller than $\frac{1+1/\sqrt2}{3+\varepsilon_0}$, since $\sqrt2-1+\varepsilon_0\frac1{\sqrt2}>0$. There is no hole here either, and globally, we get the full interval $\left[-\prod_{0\le i\le n}\frac1{2+\varepsilon_i},\anib[(1/(2+\varepsilon_i))_{0\le i\le n}]{1\ldots12}\right]$.

The proof of the statement is finished by observing that $\prod_{0\le i<n}\frac1{2+\varepsilon_i}\to0$ and $\anib[(1/(2+\varepsilon_i)_{i<n})]{1\ldots12}\to\anib[(1/(2+\varepsilon_i)_{i\in\N})]{1\ldots}$.
\end{proof}

\section{Computing directions}
Lemma~\ref{lem:nonexpsftrestr} stated that the set of non-expansive directions of an SFT is \emph{effectively} closed, which means that there exists a program which takes as (infinite) input the description of a direction in $\Rb$ and halts (after having read finitely many bits of the input) if and only if the direction is expansive.
In Subsection \ref{ss:comput}, it was suggested that the good way to represent directions in order to compute with them was by the two coordinates of some intersection with the unit circle. Each slope then has two (opposite) valid representations.
When restricting to closed subsets of $\R \subseteq \Rb$ (\ie when we are not talking about the horizontal direction), the notion of effectively closed set of direction is the same with the above representation as with the usual definition of $\R$. This is due to the facts that the functions $\sin$ and $\cos$ and their inverses are computable and that the function $x \to 1/x$ is uniformly continuous away from $0$.

%$\theta$ is indeed real-valued, so we can have computability considerations without bothering of the influence of the representation.
The following remark states that directive sequences give another, equivalent representation for directions.

\begin{remark}\label{rem:compslope}
Let $\seq\varepsilon \in \R^{\N}$ be computable. Then, $\theta_{\seq\varepsilon}$ is a computable function.
%Inversely, given a slope $l \in \R$, we can computably find a sequence of directives $\seq d \in \theta_{\seq\varepsilon}^{-1}$.
%\\
%Consequently, for any effectively closed subset of $Z\subset[0,1[$, $\theta_{\seq\varepsilon}^{-1}(Z)$ is effectively closed in $(\N^2\setminus\{(0,0)\})^\N$.
\end{remark}
The computation is actually uniform in $\seq\varepsilon$, in the sense that it could be considered as part of the input.
\begin{proof}
This follows from the fact that the diameter of $\Theta_{\varepsilon\restr{\co0n}}(\vec d)$ is at most $2^{-n}$, since $R_i \le 1/2$, for all $i \in \N$ and directive sequence $\vec d$.
%Recall that $\theta_{\seq\varepsilon}(\seq d)$ is the unique element of $\bigcap_{n\in\N}(\prod_{0\le i<n}R_i)[-1,1]+\anib[R\restr{\co0n}]{0D\restr{\co0n}}$.
%So, to approximate it up to precision $\delta$, it is enough to compute $\anib[R\restr{\co0n}]{0D\restr{\co0n}}$ (computable from $D_i,W_i$ for all $i<n$, since $\varepsilon\restr{\co0n}$ is computable) for the first $n$ such that $\prod_{0\le i<n}R_i<\delta$ (this $n$ will be found eventually).

%The following algorithm outputs a rational number at distance less than $2^{-m}$ from $\theta_{\seq\varepsilon}((D_i,W_i)_{i<k})$.
%\begin{algorithmic}
%\REQUIRE{$\Chi{m,(D_i,W_i)_{i<k}}$}
%\STATE{$n\leftarrow0$}
%\WHILE{$\prod_{0\le i<n}R_i>2^{-m}$}
%\STATE{$n\leftarrow n+1$}
%\ENDWHILE
%\RETURN{$\anib[R\restr{\co0n}]{0D\restr{\co0n}}$}
%\end{algorithmic}
\end{proof}

Remark~\ref{rem:compslope} implies that effectively closed sets of slopes can be equivalently described by an effectively closed set of directive sequences. This is the computational description of directive sequences that we are going to use in the next chapter.
%In the next subsection, we will use the TM given by Remark~\ref{rem:compslope} when realizing a given effectively closed set as the set of non-expansive directions.

\section{Realization of sets of non-expansive directions}

Let $\C[\reali]=\C[\unive]\sqcup[\mhist,\mshift,\mshift_{+1},\prog,\revprog]$.

The following permutation will also be parametrized by an effectively closed set $\NE_0 \subseteq [0,1/2]$, which is represented by the program $p'$ of the TM that recognizes $\NE_0$ as a set of directive sequences. $\NE_0$ is the set of non-expansive directions that we are trying to realize.

We identify the set $\parint \defeq \{(0,1),(1,1),(1,0)\}$ with $\haine3$ (through any bijection). If $a \in \parint$, then $D_a, W_a$ will denote the projection of $a$ onto the first and second coordinate, respectively.

In the following algorithm, $p'$ is the program of a TM that recognizes some set of non-expansive directions.%$\seq{S}, \seq{U} \in \N^{\N}$ are parameters of the rule and $\seq{T} \defeq \seq{T}_{\seq{S},\seq{U}}$ is defined with respect to these two sequences as follows: $T_n$

\begin{algo}{reali}{\reali}{M,\vec{\nu},p',\seq{\vec{k}},\seq S,\seq U}
\STATE{$\chekk[\pi_\mshift=\pi_{\mshift_{+1}}]$}\label{al:reali:mhistconsistency}
\IF{$\bina{\pi_\age}=0$}
\STATE{$\chekk[\halt{p'}{\length{\mhist}}{\mhist}]$}\label{al:reali:medvedev}
\STATE{$\chekk[\emp[\motvide]{\head_{-1},\head_{+1},\mail_{-1},\mail_{+1},\newinfo}]$}\label{al:reali:empty}
	\STATE{$\chekka[M,\bina{\pi_\addr},\pi_\info,\vec{k}_{\length{\mhist}+1}]$} \label{al:reali:alph}
	\STATE{$\hier[M,\bina{\pi_\addr},\pi_\info,\vec{k}_{\length{\mhist}+1},\prog,%\sh[\vec{v}_{\length{\mhist}+1,\prog}]{\pi_\prog}
\pi_{\prog}]$}\label{al:reali:prog}
	\STATE{$\hier[M,\bina{\pi_\addr},\pi_\info,\vec{k}_{\length{\mhist}+1},\revprog,%\sh[\vec{v}_{\length{\mhist}+1,\revprog}]{\pi_\revprog}
\pi_\revprog]$}\label{al:reali:revprog}
	\STATE{$\hier[M,\bina{\pi_\addr},\pi_\info,\vec{k}_{\length{\mhist}+1},\mhist,\pi_{\mhist}\pi_{\mshift}]$}\label{al:reali:infinitesequence}
\ENDIF
\IF{$\bina{\pi_\age}=0$}
	\STATE{$\exch[\info,\mail_{+1}]$}\label{al:reali:startshift}
\ENDIF
\IF{$\bina{\pi_\age}=D_{\mshift}S_{\length{\mhist}}$}
	\STATE{$\exch[\info,\mail_{+1}]$}\label{al:reali:stopshift}
\ENDIF
	\STATE{$\unive[M$, $\vec{\nu}$, $\vec{k}_{\length{\mhist}}$, $\bina{\pi_\addr}$, $\bina{\pi_\age}-D_{\mshift}S_{\length{\mhist}}$, $S_{\length{\mhist}}$, $S_{\length{\mhist}}(D_{\mshift}+W_{\mshift}+1)+4U_{\length{\mhist}}$ $,U_{\length{\mhist}}$, $\pi_\prog$, $\pi_\revprog]$}\label{al:reali:unive}

%\IF{$\bina{\pi_\age}=\seq{S}_{\norm{\mhist}}+4\seq{U}_{\norm{\mhist}}$}
	%\STATE{$\exch[\info,\mail_{+1}]$}\label{al:reali:startshift}
%\ENDIF
%\IF{$\bina{\pi_\age}=\seq{S}_{\norm{\mhist}}+4\seq{U}_{\norm{\mhist}}+D_{\mshift}
%\seq{S}_{\norm{\mhist}}$}
	%\STATE{$\exch[\info,\mail_{+1}]$}\label{al:reali:stopshift}
%\ENDIF

\STATE{$\coordi[\seq{S}_{\length{\mhist}},\seq{S}_{\length{\mhist}}(D_{\mshift}+W_{\mshift}+1)+4\seq{U}_{\length{\mhist}}]$}\label{al:reali:coordi}
\end{algo}

%This rule is also polynomially computable in its parameters.

This is like the simulation of the computation degrees, only that we keep a more complicated register in the $\mhist$ field and we use the values of $\mshift$ to perform a ``macro-shift'' before the simulation. We also note that $\seq{T}$ is not given as a parameter of the construction. Instead, it is determined by the sequences $\seq{S}, \seq{U}$ and the value of field $\mshift$.

\begin{lemma}\label{lem:nexpdirsimul}
Let $\seq S,\seq U$ be polynomially checkable sequences of integers and $p'$ be the program of a TM. Let us fix the field list $\C[\reali]\defeq [0,\ldots,14]$, the corresponding direction vector $\vec{\nu}_{\reali}$ and a polynomially checkable sequence of $15$-uples $\seq{\vec{k}}\in(\N^{14})^\N$. Let $F$ be the IPPA with directions $\vec{\nu}_{\reali}$ and permutation 
%\begin{equation*}
$\reali[14,\vec{\nu}_{\reali},p',\vec{k},\seq S,\seq U]$
%\end{equation*} 
and $p,p^{-1}$ be the programs for this permutation and its inverse, respectively.

For all $w \in \parint^*$ and $a \in \parint$, let $\vec{k}_{w,a} \defeq \vec{k}_{\length{w}}$, $S_{w,a}\defeq S_{\length{w}}$, $U_{w,a} \defeq U_{\length{w}}$, $T_{w,a} \defeq S_{w,a}(D_a+W_a+1)+4U_{w,a}$ and $F_{w,a}$ be the restriction of $F$ to the subalphabet

\begin{equation*}
\A_{w,a}\defeq \haine5^{\vec{k}_{\length{w}}}\cap\emp[w]{\mhist}\cap\emp[a]{\mshift,\mshift_{+1}}\cap\emp[p]{\prog}\cap\emp[p^{-1}]{\revprog}.
\end{equation*}

Assume that the following inequalities hold, for all $w\in\parint^*$ and $a,a' \in \parint$:
\[\both{
U_{w,a}\ge\max\{t_p({\haine5^{\seq{\vec{k}}_{wa,a'}}}),t_{p^{-1}}({\haine5^{\seq{\vec{k}}_{wa,a'}}})\}\\
S_{w,a}\ge\max\{2U_{w,a},\length{\Chi{\haine5^{\vec{k}_{wa,a'}}}}\}\\
%T_{w,a} = S_{\length{w}}(D_a+W_a+1)+4U_{\length{w}}\\
k_{w,a,\addr},k_{w,a,\addr_{+1}}\ge\norm{S_{w,a}}\\
k_{w,a,\age},k_{w,a,\age_{+1}}\ge\norm{T_{w,a}}\\
k_{w,a,\head_{-1}},k_{w,a,\head_{+1}}\ge\max\length{\Chi{\haine4\times (Q_p \cup Q_{p^{-1}})\times\{-1,+1\}}}\\
k_{w,a,\info},k_{w,a,\newinfo},k_{w,a,\mail_{-1}},k_{w,a,\mail_{+1}}\ge1\\
k_{w,a,\prog}\geq\length p\\
k_{w,a,\revprog}\geq\length{p^{-1}}\\
k_{w,a,\mhist} \geq\length{w}\\
k_{w,a,\mshift}=k_{w,a,\mshift_{+1}}\geq 1
~.}\]

Then, $F_{w,a}$ completely exactly simulates $\bigsqcup_{\begin{subarray}{c}a'\in \parint\end{subarray}}{F_{wa,a'}}$ with parameters $(S_{w,a},T_{w,a},D_aS_{w,a},\Phi_{w,a})$, where $\Phi_{w,a}=\tilde{\Phi}_{\info}{\restr{\Sigma_{w,a}}}$ and 
%\begin{multline*}
$\Sigma_{w,a} \defeq \A_{w,a}^{\Z} \cap \gra{0}{0}{S_{w,a}}{T_{w,a}}\cap\emp[\motvide]{\head_{-1},\head_{+1},\mail_{-1},\mail_{+1},\newinfo}
\cap \Phi^{-1}(\bigsqcup_{\begin{subarray}{c} a' \in \parint \end{subarray}}\A_{wa,a'}^{\Z}).$
%\end{multline*}
\end{lemma}

The only difference between this proof and the proof of Lemma~\ref{l:mhist} is that we shift all the encodings $D_aS_{w,a}$ cells (equivalently, $D_a$ macro-cells) to the right before starting the simulation.

Also, it is not difficult to see that the usual inequalities hold: $t_p(\A_{n}) = t_{p^{-1}}(\A_{n})=P(\length{\A_n}+t_{\vec{k}}(n)+t_{\seq{S}}(n)+t_{\seq{T}}(n)+t_{\seq{U}}(n))$ and $\length{p},\length{p^{-1}} = O(\length{p_{\vec{k}}}+\length{p_{\seq{S}}}+\length{p_{\seq{T}}}+
\length{p_{\seq{U}}})$. Recall that $p'$ is fixed in advance so it is a constant in what matters complexity.

\begin{proof}
If $p'$ halts on input $w$ within $\length{w}$ steps, then the check of line~\ref{al:reali:medvedev} will reject every configuration, which means that $F_{w,a} = \emptyset$. But, in this case, $wa$ will also be rejected by $p'$ within $\length{wa}$ steps, for all $a \in \parint$, so that $\bigsqcup_{\begin{subarray}{c}a'\in \parint\end{subarray}}{F_{wa,a'}} = \emptyset$, too. By definition, the empty PCA strongly, completely simulates itself for all possible choices of the simulating parameters, so that the claim is true in this case.

Suppose, then, that $p'$ does not halt on input $w$ within $\length{w}$ steps. Then, the check of line~\ref{al:reali:medvedev} is always true, so that we can ignore it in the rest of the proof. As in the previous proofs, we have to show three things: that $F_{w,a}$ $(S_{w,a},T_{w,a})$ simulates $\bigsqcup_{\begin{subarray}{c}a\in \haine2\end{subarray}}{F_{wa,a'}}$ with decoding function $\Phi_{w,a}$ (simulation), that $\Phi_{w,a}$ is injective (exactness) and that $\Omega_{F_{w,a}} \subseteq \dom{\Phi_{w,a}}$ (completeness).

For the simulation, it is easy to see that if $b \in \A_{wa,a'}^{\Z}$, where $a' \in \haine2$ and $c \in \Phi_{w,a}^{-1}(b)$, then $c$ is not rejected by the checks of lines~\ref{al:syncomp:mhistconsistency},\ref{al:syncomp:empty},\ref{al:syncomp:alph},\ref{al:syncomp:prog}, \ref{al:syncomp:revprog} and \ref{al:syncomp:infinitesequence}. 

Then, line~\ref{al:reali:startshift} copies \emph{all} the info bits onto $\mail_{+1}$. During the next $S_{w,a}D_a$ steps, no permutation is applied. The only thing happening to the configuration is that the encodings that are in $\mail_{+1}$ travel to the right at the speed of one cell per time step. After $S_{w,a}D_a$ steps, they are copied back to the $\info$ tape by line~\ref{al:reali:stopshift}. Every letter has travelled exactly $S_{w,a}D_a$ cells to the right, which corresponds to $D_a$ macro-cells. Formally, $\Phi(F^{D_aS_{w,a}}(c))=\sigma^{-D_a}(b)$.

Then, from Fact~\ref{fact:shiftandpermutation} and since the only rule applied from $\age=D_aS_{w,a}$ is $\coordi \circ \unive$, we obtain that $\Phi(F^{D_aS_{w,a}+S_{w,a}+4{U}_{w,a}}(c))=$ $F_{wa,a'}(\sigma^{-D_a}(b))$. After $\age=D_aS_{w,a}+S_{w,a}+4{U}_{w,a}$, nothing else changes in the configuration until $\age$ becomes $0$ again. Line~\ref{al:reali:coordi} ensures that $\age$ goes from $0$ to $(D_a+W_a+1)\seq{S}_{w,a}+4\seq{U}_{w,a}$. This concludes the proof of the simulation part.

Exactness of the simulation is easy to see. The values of all the fields of $c \in \Phi_{w,a}^{-1}(b)$ are uniquely determined by $b$ and $\Sigma_{w,a}$.

For the completeness, we show that if  $c \in \gra{0}{0}{S_{w,a}}{T_{w,a}} \cap F_{w,a}^{-2T_{w,a}}(\A_{w,a}^{\Z})$, then $c \in \Phi_{w,a}^{-1}(\bigsqcup_{\begin{subarray}{c} a' \in \parint \end{subarray}}\A_{wa,a'}^{\Z})$.
 
Indeed, if $c \in \gra{0}{0}{S_{w,a}}{T_{w,a}} \cap F_{w,a}^{-T_{w,a}}(\A_{w,a}^{\Z})$, then lines~\ref{al:reali:empty},\ref{al:reali:prog},\ref{al:reali:revprog} and \ref{al:reali:infinitesequence} ensure that 
%\begin{multline*}
$c \in \A_{w,a}^{\Z} \cap \gra{0}{0}{S_{w,a}}{T_{w,a}}\cap \emp[\motvide]{\head_{-1},\head_{+1},\mail_{-1},\mail_{+1},\newinfo} \cap \Phi^{-1}((\bigsqcup_{\begin{subarray}{c} a' \in \parint \end{subarray}}\A_{wa,a'})^{\Z}).$
%\end{multline*}
Let $b \in (\bigsqcup_{\begin{subarray}{c} a' \in \parint \end{subarray}}\A_{wa,a'})^{\Z}$ be such that $c \in \Phi^{-1}(b)$. We still cannot know that $\pi_{\mshift}(b_i)$ is the same for all $i \in \Z$. %This is why we need to take $2T$ steps instead of $T$ steps.

We deal with this problem in a similar way as in Section~\ref{s:comput}, since $F_{w,a}^{2T}(c)$ exists, this means that $F^2(b)$ exists, and line\ref{al:reali:mhistconsistency} ensures that $\pi_{\mshift}(b_i)=\pi_{\mshift_{+1}}(b_i)=\pi_{\mshift}(b_j)$, for all $i,j \in \Z$.
Therefore, $b \in \bigsqcup_{\begin{subarray}{c} a' \in \haine2 \end{subarray}}\A_{wa,a'}^{\Z}$ and this concludes the proof.
\end{proof}

\subsection{Satisfying the inequalities}
Unlike the previous cases, the set of inequalities that we want to satisfy does not depend on $n$, but instead on a word $w \in \parint^*$ and $a,a' \in \parint$. However, we will now see that the inequalities can be translated to some inequalities about the polynomially computable sequences $\seq{S},\seq{U}$.% and satisfied in a strong way.

\begin{remark}
We can find $\vec{k} \in (\N^{15})^{\N}$ and $\seq{S}, \seq{U} \in \N^{\N}$ such that the inequalities of Lemma~\ref{lem:nexpdirsimul} are satisfied.

In addition, we can have $\seq\varepsilon_n \defeq 4U_n / S_n < \sqrt{2}-1$, for all $n \in \N$ and $\theta_{\seq\varepsilon}(\parint^{\N}) \supseteq [0,1/2]$.
\end{remark}

\begin{proof}
Let $w \in \parint^*$ and $a,a' \in \parint$. Let $n \defeq \length{w}$ be the length of $w$. Let us write again the inequalities:

\[\both{
U_{w,a}\ge\max\{t_p({\haine5^{\seq{\vec{k}}_{wa,a'}}}),t_{p^{-1}}({\haine5^{\seq{\vec{k}}_{wa,a'}}})\}\\
S_{w,a}\ge\max\{2U_{w,a},\length{\Chi{\haine5^{\vec{k}_{wa,a'}}}}\}\\
k_{w,a,\prog}\geq\length p\\
k_{w,a,\revprog}\geq\length{p^{-1}}\\
k_{w,a,\head_{-1}},k_{w,a,\head_{+1}}\ge\length{\Chi{\haine4\times (Q_p \cup Q_{p^{-1}})\times\{-1,+1\}}}\\

k_{w,a,\addr},k_{w,a,\addr_{+1}}\ge\norm{S_{w,a}}\\
k_{w,a,\age},k_{w,a,\age_{+1}}\ge\norm{T_{w,a}}\\
k_{w,a,\info},k_{w,a,\newinfo},k_{w,a,\mail_{-1}},k_{w,a,\mail_{+1}}\ge1\\
k_{w,a,\mhist} \geq\length{w}\\
k_{w,a,\mshift}=k_{w,a,\mshift_{+1}}\geq 1
~.}\]

According to the definition, $\vec{k}_{w,a} = \vec{k}_n$, $S_{w,a}=S_n$, $U_{w,a}=U_n$ and $T_{w,a}=S_n(D_a+W_a+1)+4U_n$: Therefore, we can write the above inequalities as follows:

\[\both{
U_n\ge\max\{t_p({\haine5^{\seq{\vec{k}}_{n+1}}}),t_{p^{-1}}({\haine5^{\seq{\vec{k}}_{n+1}}})\}\\
S_{n}\ge\max\{2U_{n},\length{\Chi{\haine5^{\vec{k}_{n+1}}}}\}\\
%T_{w,a} = S_{\nor}(D_a+W_a+1)+4U_{\length{w}}\\
k_{n,\prog}\geq\length p\\
k_{n,\revprog}\geq\length{p^{-1}}\\
k_{n,\head_{-1}},k_{n,\head_{+1}}\ge\length{\Chi{\haine4\times (Q_p \cup Q_{p^{-1}})\times\{-1,+1\}}}\\

k_{n,\addr},k_{n,\addr_{+1}}\ge\norm{S_{n}}\\
k_{n,\age},k_{n,\age_{+1}}\ge\norm{S_n(D_a+W_a+1)+4U_n}\\
k_{n,\info},k_{n,\newinfo},k_{n,\mail_{-1}},k_{n,\mail_{+1}}\ge1\\
k_{n,\mhist} \geq n\\
k_{n,\mshift}=k_{n,\mshift_{+1}}\geq 1
~.}\]

Unlike the previous proofs, we cannot choose $\vec{k}_{\seq{S},\seq{T}} \in \N^{\N}$ such that all of the inequalities except the first two are satisfied as equalities. This is because the inequalities about $\age$ and $\age_{+1}$ depend on $a \in \parint$ and not only on $n$. However, $D_a+W_a \le 2$, for all $a\in \parint$, so that we can replace these inequalities with $k_{n,\age},k_{n,\age_{+1}}\ge\norm{3S_n+4U_n}$ and show that this new set of inequalities can be satisfied. (Here, it is essential that $\parint$ is a finite set. Bounding $D_a+W_a$ from above is one of the reasons that we had to do a little more work with the directive sequences.)

The rest of the proof follows the usual pattern. We choose $S_n = Q^{n+n_0}$ and $U_n=(n+n_0)^r$ for some suitable values of $n_0,r$ and $Q$.
%For all $n \in \N$ and $\seq{S}$ and $\seq{U}$,  let us choose $\vec{k}_n\defeq\vec{k}_{n,\seq{S},\seq{U}}$ such that all of the inequalities except the first two are equalities. Then,% the sequence $\vec{k}$ is polynomially checkable,
% $\length{\Chi{\haine5^{\vec{k}_n}}}\le P_1(\log{S_n},\log{{U}_n},n)$ and $\max\{t_p({\haine5^{\vec{k}}}),t_{p^{-1}}({\haine5^{\vec{k}}})\} \le P_2(\log{S_n},\log{U_n},n)$, for some polynomials $P_1,P_2$. %These follow by definition of $\Chi{\cdot}$ and $\vec{k}_{n}$ and the fact that the complexity of $p$ is bounded by the sum of a polynomial term plus the complexity of computing $S_n,T_n,U_n$ from $n$ (which according to the assumption of polynomial checkability is polynomial in $\log{S_n},\log{T_n}$ and $\log{U_n}$, respectively).

%Therefore, it is enough to find polynomially computable sequences $\seq{S},\seq{U}$ that satisfy the following inequalities, for all $n \in \N$:
%\[\both{
%U_n\ge P_2(\log{S_{n+1}},\log{U_{n+1}},n+1)\\
%S_n\ge\max\{2{U}_n,P_1(\log{S_{n+1}},\log{U_{n+1}},n+1)\}.
%T_n\ge 4U_n+S_n.
%~}\]
For the second claim, let us recall more specifically in which order $n_0,r$ and $Q$ are chosen. In the proof of Remark~\ref{rem:inequhiera}, we showed that there exists $n_0$ and $r$ that work for \emph{every} $Q$. Therefore, by choosing $S_n = Q^{n+n_0}$ and $U_n=(n+n_0)^r$, for some sufficiently large $Q$, we can make $\seq\varepsilon \defeq 4U_n / S_n$ smaller than $\sqrt{2}-1$ and $\anib[(1/(2+\varepsilon_i)_i)]{111\ldots}$ larger than $1/2$.
\end{proof}

For all $w,a$, we have $T_{w,a}= (D_a+W_a+1)S_{w,a}+4U_{w,a}=(D_a+W_a+1+4U_{w,a}/S_{w,a})S_{w,a}=(D_a+W_a+1+\epsilon_{w,a})S_{w,a}$, where $\epsilon_{w,a} < \sqrt{2}-1$, so that we are in the situation described in Lemma~\ref{lem:geomstuff} and $\theta_{\seq\varepsilon}(\parint^{\N}) \supseteq [0,1/2]$. % implies that the set of directions that can we can try to realize contains $\NE_0 \subseteq [0,1/2]$.
Since $\NE_0 \subseteq [0,1/2]$ by assumption, this means that every direction in $\NE_0$ is representable as a sequence in $\theta_{\seq\varepsilon}(\parint^{\N})$.

\subsection{Realization}
For all $z \in \X_{p'}$, we have a sequence of complete, exact simulations given by
 
\begin{equation*}
%F_{\epsilon,z_0}\simu[S_{\epsilon,z_0},T_{\epsilon,z_0},D_{z_0}S_{\epsilon,z_0}]
%F_{z_0,z_1}\simu[S_{z_0,z_1},T_{z_0,z_1},D_{z_0,z_1}S_{z_0,z_1}]\ldots
F_{z_{\co{0}{n}},z_n} \simu[S_{z_{\co{0}{n}},z_n},T_{z_{\co{0}{n}},z_n},D_{z_n}S_{z_{\co{0}{n}},z_n}] F_{z_{\co{0}{n+1}},z_{n+1}},
%F_{z_0,z_1}
\end{equation*}

for all $n \in \N$.

Therefore, according to a Lemma~\ref{l:nonvides}, we have that $\Omega_F= \bigsqcup_{z \in \X_{p'}} \rocks[\infty]z{\Phi}$, and $\orb{F_{\epsilon}}=\bigsqcup_{z \in \X_{p'}} \orb{F_{\rocks[\infty]z{\Phi}}}$. In addition, we know that $\NE(\orb{F_{\rocks[\infty]z{\Phi}}})=\theta(z)$, by Lemma~\ref{lem:iterrelsimulexp} and that every direction in $\cc{0}{1/2}$ can be represented in such a way, by Lemma~\ref{lem:geomstuff}. Finally, we know by Lemma~\ref{lem:basicstuffaboutNE} that $\NE(\orb{F_{\epsilon}})= \bigsqcup_{z \in \X_{p'}} \NE(\orb{F_{\rocks[\infty]z{\Phi}}}) = \bigsqcup_{z \in \X_{p'}} \theta(z) = \NE_{p'}=\NE_0$.

Therefore, for every effectively closed set of directions which is included in $[0,1/2]$, recognized by a TM with program $p'$, we have constructed a 2D SFT with exactly this set as set of non-expansive directions. According to our previous discussions, this is enough to realize arbitrary effectively closed sets of directions as the set of non-expansive directions of 2D SFT.
This concludes the proof of Theorem~\ref{thm:nonexpansive}.

\chapter*{Conclusion and Open Questions}

We have provided a general method for constructing extremely-expansive 2D SFT's of finite type and we have shown that this class of 2D SFT's has very rich computational, dynamical and geometrical properties. At the same time, our method throws some light on the essence of self-similar and hierarchical constructions and we hope that it might help to better understand previous works with hierarchical constructions, especially \cite{drs}. ( On the other hand, the difficulty of \cite{gacs} only partly comes from the hierarchical simulation, so our work is certainly not sufficient to explain this construction better.)

Regarding future work, we believe that the following questions about extremely-expansive 2D SFTs are very natural: First of all, can (a variant of) our method produce a \emph{minimal} extremely-expansive SFT? Is the emptiness problem undecidable for \emph{minimal} extremely-expansive SFT's? Recently, Durand and Romashchenko \cite{drs2} described a method for constructing minimal (but not extremely-expansive) SFT's and answer the second question positively. It seems that their technique can be readily generalized to our framework. Second, is it possible to realize all effective subshifts, in the sense of \cite{projsft, drs, aubrun} with 2D extremely-expansive SFT's? This would be an improvement with of the result of \cite{drs, aubrun} since it would (in some sense) further lower the dimension of the realizing subshift by one. Third, is it possible to construct extremely expansive SFT covers for square substitutions? This question goes back to the construction of Mozes \cite{mozes}, which constructs SFT covers for square substitutions  without any directions of expansivess. Recently, Ollinger and Legloannec \cite{legloannec} constructed 4-way deterministic covers. We believe that the answer to this question is also positive. Finally, and this is certainly the most interesting, but also difficult question, is it possible to use our method in order to construct reversible, self-simulating CA, \ie is it possible to turn the partial rules, with which we have been working in this thesis, to complete rules, while at the same time keeping the good properties of self-simulation? This could find an application to the problem of the undecidability of expansiveness for reversible CA.

\bibliography{entropies}

\begin{thebibliography}{10}

\bibitem{aubrun}
Nathalie Aubrun and Mathieu Sablik.
\newblock Simulation of effective subshifts by two-dimensional subshifts of
  finite type.
\newblock {\em Acta Applicandae Mathematicae}, 126(1):35--63, 2013.

\bibitem{balliermedvedev}
Alexis Ballier.
\newblock Universality in symbolic dynamics constrained by medvedev degrees.
\newblock {\em CoRR}, abs/1304.5418, 2013.

\bibitem{bennett}
C.~H. Bennett.
\newblock Logical reversibility of computation.
\newblock {\em IBM Journal of Research and Development}, 17(6):525--532,
  November 1973.

\bibitem{berger}
R.~Berger.
\newblock {\em The Undecidability of the Domino Problem}.
\newblock American Mathematical Society memoirs. American Mathematical Society,
  1966.

\bibitem{opsd}
Mike Boyle.
\newblock Open problems in symbolic dynamics.
\newblock {\em Contemporary mathematics}, 469:69--118, 2008.

\bibitem{expsubd}
Mike Boyle and Douglas Lind.
\newblock Expansive subdynamics.
\newblock {\em Transactions of the American Mathematical Society},
  349(1):55--102, 1997.

\bibitem{drs2}
Bruno Durand and Andrei Romashchenko.
\newblock Quasiperiodicity and non-computability in tilings.
\newblock {\em CoRR}, abs/1504.06130, 2015.

\bibitem{drs}
Bruno Durand, Andrei Romashchenko, and Alexander Shen.
\newblock Fixed-point tile sets and their applications.
\newblock {\em Journal of Computer and System Sciences}, 78(3):731 -- 764,
  2012.

\bibitem{gacs1}
Peter G\'{a}cs.
\newblock Reliable computation with cellular automata.
\newblock {\em Journal of Computer and System Sciences}, 32(1):15--78, 1986.

\bibitem{gacs}
Peter G\'acs.
\newblock Reliable cellular automata with self-organization.
\newblock {\em Journal of Statistical Physics}, 102(1--2):45--267, 2001.

\bibitem{gacshoyruprojas}
Peter Gács, Mathieu Hoyrup, and Cristóbal Rojas.
\newblock Randomness on computable probability spaces—a dynamical point of
  view.
\newblock {\em Theory of Computing Systems}, 48(3):465--485, 2011.

\bibitem{gray}
Lawrence Gray.
\newblock A reader's guide to {G}\'acs's “{P}ositive {R}ates” paper.
\newblock {\em Journal of Statistical Physics}, 103(1--2):1--44, 2001.

\bibitem{pierreunpub}
Pierre Guillon, Jarkko Kari, and Charalampos Zinoviadis.
\newblock On determinism in subshifts.
\newblock Unpublished manuscript, 2015.

\bibitem{zinoviadis1}
Pierre Guillon and Charalampos Zinoviadis.
\newblock Densities and entropies in cellular automata.
\newblock In {\em How the World Computes - Turing Centenary Conference and 8th
  Conference on Computability in Europe, CiE 2012, Cambridge, UK, June 18-23,
  2012. Proceedings}, pages 253--263, 2012.

\bibitem{hedlund}
G.A. Hedlund.
\newblock Endomorphisms and automorphisms of the shift dynamical system.
\newblock {\em Mathematical systems theory}, 3(4):320--375, 1969.

\bibitem{hinman}
Peter~G. Hinman.
\newblock A survey of mučnik and medvedev degrees.
\newblock {\em Bull. Symbolic Logic}, 18(2):161--229, 06 2012.

\bibitem{hochmanuniv}
Michael Hochman.
\newblock A note on universality in multidimensional symbolic dynamics.
\newblock {\em Discrete and Continuous Dynamical Systems - Series S},
  2(2):301--314, 2009.

\bibitem{projsft}
Michael Hochman.
\newblock On the dynamics and recursive properties of multidimensional symbolic
  systems.
\newblock {\em Inventiones Mathematic{\ae}}, 176(1):131--167, April 2009.

\bibitem{nexpdir}
Michael Hochman.
\newblock Expansive directions for $\mathbbm{Z}^2$ actions.
\newblock {\em Ergodic Theory \& Dynamical Systems}, 31(1):91--112, 2011.

\bibitem{entrsft}
Michael Hochman and Tom Meyerovitch.
\newblock A characterization of the entropies of multidimensional shifts of
  finite type.
\newblock {\em Annals of Mathematics}, 171(3):2011--2038, 2010.

\bibitem{vanierdegrees}
Emmanuel Jeandel and Pascal Vanier.
\newblock Turing degrees of multidimensional {S}{F}{T}s.
\newblock {\em Theor. Comput. Sci.}, 505:81--92, 2013.

\bibitem{karipapasoglou}
J.~Kari and P.~Papasoglu.
\newblock Deterministic aperiodic tile sets.
\newblock {\em Geometric and Functional Analysis GAFA}, 9(2):353--369, 1999.

\bibitem{nilpind}
Jarkko Kari.
\newblock The nilpotency problem of one-dimensional cellular automata.
\newblock {\em SIAM Journal on Computing}, 21(3):571--586, 1992.

\bibitem{jarkkoppa}
Jarkko Kari.
\newblock Representation of reversible cellular automata with block
  permutations.
\newblock {\em Mathematical Systems Theory}, 29(1):47--61, 1996.

\bibitem{jarkkosmall}
Jarkko Kari.
\newblock A small aperiodic set of wang tiles.
\newblock {\em Discrete Mathematics}, 160(1-3):259--264, 1996.

\bibitem{jarkkoundec}
Jarkko Kari.
\newblock On the undecidability of the tiling problem.
\newblock In {\em {SOFSEM} 2008: Theory and Practice of Computer Science, 34th
  Conference on Current Trends in Theory and Practice of Computer Science,
  Nov{\'{y}} Smokovec, Slovakia, January 19-25, 2008, Proceedings}, pages
  74--82, 2008.

\bibitem{kurdyumov}
G.~L. Kurdyumov.
\newblock An example of a non-ergodic one-dimensional homogeneous random medium
  with positive transition probabilities.
\newblock {\em Soviet Mathematics Doklady}, 19(1), 1978.

\bibitem{laffite}
Gr{\'e}gory Lafitte and Michael Weiss.
\newblock Computability of tilings.
\newblock In {\em IFIP}, volume 273 of {\em International Federation for
  Information Processing}, pages 187--201, 2008.

\bibitem{lafitteweiss}
Gr{\'{e}}gory Lafitte and Michael Weiss.
\newblock Tilings: simulation and universality.
\newblock {\em Mathematical Structures in Computer Science}, 20(5):813--850,
  2010.

\bibitem{legloannec}
Bastien Le~Gloannec and Nicolas Ollinger.
\newblock Substitutions and strongly deterministic tilesets.
\newblock In {\em How the World Computes}, pages 462--471. Springer, 2012.

\bibitem{lukkarila}
Ville Lukkarila.
\newblock The 4-way deterministic tiling problem is undecidable.
\newblock {\em Theoretical Computer Science}, 410(16):1516 -- 1533, 2009.
\newblock Theory and Applications of Tilings.

\bibitem{morita}
Kenichi Morita.
\newblock Computation-universality of one-dimensional one-way reversible
  cellular automata.
\newblock {\em Information Processing Letters}, 42(6):325 -- 329, 1992.

\bibitem{mozes}
Shahar Mozes.
\newblock Tilings, substitution systems and dynamical systems generated by
  them.
\newblock {\em Journal d'analyse math\'ematique}, 53:139--186, 1988.

\bibitem{nasu154}
Masakazu Nasu.
\newblock Textile systems and one-sided resolving automorphisms and
  endomorphisms of the shift.
\newblock {\em Ergodic Theory and Dynamical Systems}, 28:167--209, 2 2008.

\bibitem{ollingersimulation}
Nicolas Ollinger.
\newblock The intrinsic universality problem of one-dimensional cellular
  automata.
\newblock In Helmut Alt and Michel Habib, editors, {\em STACS 2003}, volume
  2607 of {\em Lecture Notes in Computer Science}, pages 632--641. Springer
  Berlin Heidelberg, 2003.

\bibitem{twobytwo}
Nicolas Ollinger.
\newblock Two-by-two substitution systems and the undecidability of the domino
  problem.
\newblock In Arnold Beckmann, Costas Dimitracopoulos, and Benedikt L\"owe,
  editors, {\em %Logic and theory of algorithms, Computability in Europe (
  CiE'2008}, volume 5028 of {\em LNCS}, pages 476--485, Athens, Greece, June
  2008. Springer Berlin / Heidelberg.

\bibitem{robinson}
Raphael~M. Robinson.
\newblock Undecidability and nonperiodicity for tilings of the plane.
\newblock {\em Inventiones Mathematic\ae}, 12(3), 1971.

\bibitem{simpson}
Stephen~G. Simpson.
\newblock Medvedev degrees of 2-dimensional subshifts of finite type, 2007.

\bibitem{guillaume1}
Guillaume Theyssier.
\newblock How common can be universality for cellular automata?
\newblock In {\em STACS 2005}, volume 3404 of {\em Lecture Notes in Computer
  Science}, pages 121--132. 2005.

\bibitem{wang}
Hao Wang.
\newblock Proving theorems by pattern recognition ii.
\newblock {\em Bell System Technical Journal, The}, 40(1):1--41, Jan 1961.

\bibitem{zinoviadis2}
Charalampos Zinoviadis.
\newblock Hierarchy and expansiveness in 2d subshifts of finite type.
\newblock In {\em Language and Automata Theory and Applications - 9th
  International Conference, {LATA} 2015, Nice, France, March 2-6, 2015,
  Proceedings}, pages 365--377, 2015.

\end{thebibliography}
% Include the BibTeX bibliograph containing your references 
\end{document}